\documentclass[9pt]{amsart}
\usepackage{amssymb,amsthm,amsmath}
\usepackage[numbers,sort&compress]{natbib}
\usepackage{color}
\usepackage{graphicx}
\usepackage{tikz}
\usepackage{mathrsfs}
\title[ ]{ A KAM theorem for the Hamiltonian with finite zero normal frequencies and its applications}
\author{ Yuan Wu and Xiaoping Yuan}

\address{School of Mathematical Sciences, Furan University, Shanghai 200433, P. R. China} \email{14110840003@fudan.edu.cn}
\address{School of Mathematical Sciences, Furan University, Shanghai 200433, P. R. China}
\email{xpyuan@fudan.edu.cn}


\theoremstyle{plain}
\newtheorem{theorem}{Theorem}[section]

\newtheorem{lemma}[theorem]{Lemma}
\newtheorem{proposition}[theorem]{Proposition}

\theoremstyle{definition}
\newtheorem{definition}[theorem]{Definition}

\newtheorem{remark}[theorem]{Remark}

\begin{document}


\begin{abstract}
In this paper,\ we investigate the existence of KAM tori for an infinite dimensional  Hamiltonian system with finite number of zero normal frequencies. By constructing a constant quantity we show that,  for ``most" frequencies in the sense of Lebesgue measure, either if the quantity is zero, there is a KAM tori or if the quantity is not zero, there is no KAM tori in some domain.  As application, \ we show that the nonlinear Schr\"{o}dinger equation with a zero frequency possesses many quasi-periodic solutions.
\end{abstract}
\maketitle

\section{Introduction and the main results}
Kuksin \cite{ku1989pe} and Wayne \cite{wa1990pe} initiated the study of the existence of lower (finite) dimensional invariant tori for nearly integrable Hamiltonian systems of infinite dimension (Also see in P\"{o}schel \cite{po1996ak}).\ More exactly,\ consider a Hamiltonian function
 \numberwithin{equation}{section}
\begin{eqnarray}\label{002}
\underline{H} = \langle\omega,y\rangle
          +\sum_{j\in \mathbb{Z}}\Omega_{j}z_{j}\bar{z}_{j}+\varepsilon R(x,y,z,\bar{z}),
\end{eqnarray}
where $(x,y,z,\bar{z}) \in \mathbb{T}^{n}\times \mathbb{R}^{n} \times \mathcal{H}\times \mathcal{H}$ and $ \mathcal{H}$ is some Hilbert space.\ We also endow the Hamiltonian $ \underline{H} $ with a symplectic structure $ dy\wedge dx + \mathbf{i}d\bar{z}\wedge dz =
\overset{n}{\underset{j\geq1}{\sum}}dy_{j}\wedge dx_{j} + \mathbf{i}\underset{j\in \mathbb{Z}}{\sum}d\bar {z}_{j}\wedge d{z}_{j}$.

Clearly,\ when $ \varepsilon = 0$,\ $ \mathcal{\underline{T}}^{n}_{0}=\mathbb{T}^{n}\times \{y=0\}\times\{z=0\}\times\{\bar{z}=0\}$ is a $ n$-dimensional invariant torus with rotational frequency $\omega$ for the Hamiltonian system defined by $ \underline{H}$.\ When $ \varepsilon$ is sufficiently small,\ assuming that
\begin{eqnarray}
&&\label{03}\langle k,\omega\rangle \neq 0,\  k\in \mathbb{Z}^{n}\setminus \{0\},\\
\label{04}
&&\langle k,\omega\rangle \pm \Omega_{j} \neq 0,\  k\in \mathbb{Z}^{n}, j\in \mathbb{Z},\\
&&\label{05}
\langle k,\omega\rangle \pm \Omega_{i}\pm \Omega_{j} \neq 0,\  k\in \mathbb{Z}^{n}, i,j\in \mathbb{Z},
\end{eqnarray}
Melnikov \cite{me1965on,me1968ac},\ Eliasson \cite{el1988pe},\ Kuksin \cite{ku1989pe,ku1993ne},\ Wayne \cite{wa1990pe} and P\"{o}schel \cite{po1989on} proved that for ``most'' (in the sense of Lebesgue measure) of the parameters $ \omega$,\ the invariant torus $ \mathcal{\underline{T}}^{n}_{0} $ can be preserved undergoing a small perturbation $ \varepsilon R(x,y,z,\bar{z})$.\ Now the conditions (\ref{04}) are called the first Melnikov's conditions,\ while (\ref{05}) are called the second Melnikov's conditions.\ Actually,\ in Kuksin's famous monograph \cite{ku1993ne},\ he
has concluded that only in the non-degenerate cases non-zero frequencies ($ \Omega_{j}\neq 0$,\ for $\forall j\in \mathbb{Z} $) and simple frequencies ($\Omega_{i}\neq\Omega_{j}$,\ for $\forall i\neq j $ and $ i,j\in \mathbb{Z} $),\ many $n$-dimensional invariant tori do exist.\ More precisely,\ in (\ref{04}),\ letting $ k=0 $,\ we get $ \Omega_{i}\neq 0 $;\ and in (\ref{05}),\ letting $ k=0 $,\ we get $ \Omega_{i}\neq\Omega_{j} $ when $ i\neq j$,\ i.e. the multiplicity of $ \Omega_{j}$ is $ 1$ ($\Omega^{\#}_{j}$ = 1).\ Therefore,\ Kuksin \cite{ku1993ne} writes
\vskip8pt
 { \it In the degenerate case\\
 $c).$ $ 0 \in {M}$ or $ \mu_{j}=\mu_{k}$ for some $ j\neq k$,\ no preservation theorem for the tori $ \mathcal{T}^{n,m}(p)$,\ formulated in terms of the unperturbed equation (8) with $ \varepsilon = 0$ only, is known yet.}
\vskip8pt
Here,\ the notation $ M = \{ \mu_{1},...,\mu_{2m}\}$,\ the notation $ \mu_{j} $ is $ \Omega_{j}$ and the notation $ \mathcal{T}^{n,m}(p) $ is $ \mathcal{\underline{T}}^{n}_{0} $ in our paper. Actually, there are two problems: Problem 1 is  whether or not one can construct a KAM theorem for the multiplicity $\Omega_j^{\sharp}>1$; Problem 2 is  whether or not one can construct a KAM theorem for some $\Omega_j=0$.

 At present time, Problem 1 has been deeply investigated when the perturbation is bounded. \ Bourgain \cite{bo1994co,bo1995co,bo1998qu,bo2004gr} developed Craig and Wayne's method \cite{cr1993ne} in studying the degenerate case $ \Omega_{j}=\Omega_{k} $ for some $ j\neq k $,\ and successfully in \cite{bo1997on} proved that there are many lower dimensional invariant tori (quasi-periodic solutions) for nonlinear Schr\"{o}dinger equations and nonlinear wave equations.\ We also mention the work of Eliasson-Kuksin \cite{el2010kam} where it is proved that there are linearly stable  KAM tori for the higher spatial dimensional nonlinear Schr\"{o}dinger equation by the classical KAM method. There are a lot of subsequent works in this line. We do not list them here.

However,\ for Problem 2,\ there have been fewer results.\ In this paper,\ we will make  attempt.\ In order to state our theorem,\ we need some preparations.

We first give some notations.\ For a fixed $ b$,\ pick a set
\[
\mathcal{J} = \{j_{1} < j_{2} < ...< j_{b}\} \subseteq  \mathbb{N}_{+}
\]
with
\[
\mathbb{N}_{+} = \{1,2,...\}.
\]
Note
\[
z^{*}=(z_{0},z),\  \bar{z}^{*}=(\bar{z}_{0},\bar{z}),
\]
where
\[
z_{0}= (z_{j_{m}})_{j_{m}\in \mathcal{J}}, z= (z_{j})_{j\in \mathbb{N}_{+}\setminus \mathcal{J}},
\]
and
\[
\bar{z}_{0}= (\bar{z}_{j_{m}})_{j_{m}\in \mathcal{J}}, \bar{z}= (\bar{z}_{j})_{j\in \mathbb{N}_{+}\setminus \mathcal{J}}.
\]

Consider a Hamiltonian
\begin{eqnarray}\label{001}
H(x,y,z^{*},\bar{z}^{*},\xi) = N(y,z^{*},\bar{z}^{*},\xi) + R(x,y,z^{*},\bar{z}^{*},\xi),
\end{eqnarray}
which is defined on a phase space
\[
(x,y,z^{*},\bar{z}^{*}) \in \mathcal{P}^{a,p} = \mathbb{T}^{n}\times \mathbb{C}^{n} \times l^{a,p}\times l^{a,p}.
\]
Here $\mathbb{T}^{n}$ is the usual $n-$torus and $l^{a,p} $ is the Hilbert space of all complex sequences $ v^{*}=(v_{0},v)$ for any $ v_{0}= (v_{j_{m}})_{j_{m}\in \mathcal{J}}, v= (v_{j})_{j\in \mathbb{N}_{+}\setminus \mathcal{J}} $ with
\[
\|v^{*}\|^{2}_{a,p}=\sum_{j_{m}\in \mathcal{J}}|v_{j_{m}}|^{2}j_{m}^{2p}e^{2aj_{m}}+ \sum_{j\in \mathbb{N}_{+}\setminus \mathcal{J}}|v_{j}|^{2}j^{2p}e^{2aj} < \infty.
\]
Let $ N(y,z^{*},\bar{z}^{*},\xi)$ be an integrable Hamiltonian which depends on parameters $ \xi\in \Pi $,\ $ \Pi$ a positive measure parameter set in $ \mathbb{R}^{n}$,\ and is of the form
\begin{eqnarray}\nonumber
 N(y,z^{*},\bar{z}^{*},\xi) &=&  N(y,z_{0},\bar{z}_{0},z,\bar{z},\xi)\\
\nonumber         &=&\langle\omega(\xi),y\rangle+\langle\Omega_{0}(\xi)z_{0},\bar{z}_{0}\rangle
          +\langle\Omega(\xi)z,\bar{z}\rangle\\
 &&\nonumber\mbox{( noting  $(\Omega_{0}z_{0})_{m}=\Omega_{j_{m}}z_{j_{m}}, {j_{m}\in \mathcal{J}}$ and $ (\Omega z)_{j}= \Omega_{j}z_{j}, {j\in \mathbb{N}_{+}\setminus \mathcal{J}} $)}\\
 \nonumber &=&\langle\omega(\xi),y\rangle+\sum_{j_{m}\in \mathcal{J}}\Omega_{j_{m}}(\xi)z_{j_{m}}\bar{z}_{j_{m}}
          +\sum_{j\in \mathbb{N}_{+}\setminus   \mathcal{J}}\Omega_{j}(\xi)z_{j}\bar{z}_{j}\\
 &&\nonumber\mbox{( letting  $\Omega_{j_{m}}=0$ for any $ j_{m}\in \mathcal{J}$ )},
\end{eqnarray}
where
\[
\omega(\xi)=(\omega_{1}(\xi),...,\omega_{n}(\xi))
\]
is called the tangential frequency and
\[
\Omega^{*}(\xi)=\Omega^{*}(\xi)=(\Omega_{0}(\xi),\Omega(\xi))=(\Omega_{j_{1}}(\xi),...,\Omega_{j_{b}}(\xi),...,\Omega_{j}(\xi),...)
\]
is called the normal frequency.\\
We assume that $ H$ is smooth sufficiently.\ When $ R=0 $, $ \mathcal{T}^{n}_{0}=\mathbb{T}^{n}\times \{y=0\}\times\{z^{*}=0\}\times\{\bar{z}^{*}=0\}$ is a $ n$-dimensional invariant torus with frequency $\omega(\xi)$ for the Hamiltonian system defined by $ H$ in (\ref{001}).\ The aim in this paper is to prove the persistence of a large portion of rotational tori under small perturbations.

To make this quantitative we introduce complex neighborhoods of $ \mathcal{T}^{n}_{0}$
\[
D(s,r,r)=\{(x,y,z^{*},\bar{z}^{*})\in \mathcal{P}^{a,p}:|\Im x| \leq s,|y|\leq r^{2},
\|z^{*}\|_{a,p}+\|\bar{z}^{*}\|_{a,p}\leq r \}
\]
where $ |\cdot| $ is the sup-norm for complex vectors and without abusing the notation,\
for $ k\in \mathbb{Z}^{n} $,\ we denote
\[
|k| = \underset{1\leq j \leq n}{\sum} |k_{j}|.
\]
We also denote
\begin{eqnarray}\nonumber
|v|_{2} = \sqrt{\underset{1\leq j \leq n}{\sum} |v_{j}|^{2}},\  v\in \mathbb{C}^{n}.
\end{eqnarray}
If $ A$ is a matrix of finite order,\ define
\[
 ||A|| = \underset{|v|_{2}\neq 0}{\sup}\frac{|Av|_{2}}{|v|_{2}},  v\in \mathbb{C}^{n},
\]
where the operator norm is reduce by $ |\cdot|_{2}$.\\
Let $ \beta = (...,\beta_{j_{m}},...,\beta_{j},...)_{j_m\in \mathcal{J},j\in\mathbb{N}_{+}\setminus\mathcal{J}}$ and $ \gamma = (...,\gamma_{j_{m}},...,\gamma_{j},...)_{j_m\in \mathcal{J},j\in\mathbb{N}_{+}\setminus\mathcal{J}}$, $ \beta_{j_{m}},\beta_{j}\in \mathbb{N}$ and $ \gamma_{j_{m}},\gamma_{j}\in\mathbb{N}$ with finitely many nonzero components of positive integers.\ The product $ \{z^{*}\}^\beta \{\bar{z}^{*}\}^\gamma = \prod_{j_m\in \mathcal{J}}z^{\beta_{j_{m}}}_{0m}\bar{z}^{\gamma_{j_{m}}}_{0m}\prod_{j\in \mathbb{N}_{+}\setminus\mathcal{J}}z^{\beta_{j}}_{j}\bar{z}^{\gamma_{j}}_{j} $.

To state the main results,\ let
\[
F(x,y,z^{*},\bar{z}^{*},\xi) =\underset{\beta,\gamma\in \mathbb{N}^{\mathbb{N}}}{\sum}{F}^{\beta\gamma}(x,y,\xi)\{z^{*}\}^\beta \{\bar{z}^{*}\}^\gamma,
\]
where $ {F}^{\beta\gamma}(x,y,\xi) =  \underset{\alpha\in \mathbb{N}^{n},k\in\mathbb{ Z}^{n}}{\sum}\hat{F}^{\alpha\beta\gamma}(k,\xi)e^{\mathbf{i}\langle k,x\rangle}y^\alpha $ is analytic in parameter $ \xi \in \Pi $ in the sense of Whitney.\ Define the weight norm of $ F$ by
\[
 \parallel F \parallel_{D(s,r,r),\Pi}=\underset{\|z^{*}\|_{a,p}+\|\bar{z}^{*}\|_{a,p}\leq r}{\sup}\underset{\beta,\gamma}{\sum}\|{F}^{\beta\gamma}\||z^{*}|^\beta|\bar{z}^{*}|^\gamma,
\]
where $ \langle,.,\rangle$ is the standard inner product and $ \|{F}^{\beta\gamma}\| $ is short for
\[
\|{F}^{\beta\gamma}\| = \underset{\alpha\in \mathbb{N}^{n},k\in \mathbb{Z}^{n}}{\sum}|\widehat{F}^{\alpha\beta\gamma}(k,\xi)|_{\Pi}r^{2\alpha}e^{|k|s},
\]
and
\[
|\hat{F}^{\alpha\beta\gamma}(k,\xi)|_{\Pi}=\underset{\xi \in\Pi} {\sup}\underset{0\leq|\iota|\leq 1}{\max}
 \left(\left|\frac{\partial^{\iota}\widehat{F}^{\alpha\beta\gamma}(k,\xi)}{\partial \xi^{\iota}}\right|\right),
\]
with $ |k|= \underset{1\leq b \leq n}{\sum}|k_{b}|$.\\
In the normal direction of the Hamiltonian vector field,\ we define
\[
X_{F} = (F_{y},-F_{x},\mathbf{i}F_{\bar{z}^{*}},-\mathbf{i}F_{{z}^{*}}),
\]
by
\begin{eqnarray}\nonumber
\|X_{F}\|_{D(s,r,r),\Pi} &=&\|F_{y}\|_{D(s,r,r),\Pi}
+\frac{1}{r^{2}}\|F_{x}\|_{D(s,r,r),\Pi}\\
\nonumber &&+\frac{1}{r}\left(\underset{j_{m}\in \mathcal{J}}{\sum}\|F_{\bar{z}_{j_{m}}}\|^{2}_{D(s,r,r)j_{m}^{2p}e^{2aj_{m}},\Pi}+\sum_{j\in \mathbb{N}_{+}\setminus \mathcal{J}}\|F_{\bar{z}_{j}}\|^{2}_{D(s,r,r),\Pi}j^{2p}e^{2aj}\right)^{\frac{1}{2}}\\
\nonumber &&
+\frac{1}{r}\left(\underset{j_{m}\in \mathcal{J}}{\sum}\|F_{{z}_{j_{m}}}\|^{2}_{D(s,r,r)j_{m}^{2p}e^{2aj_{m}},\Pi}+\sum_{j\in \mathbb{N}_{+}\setminus \mathcal{J}}\|F_{z_{j}}\|^{2}_{D(s,r,r),\Pi}j^{2p}e^{2aj}\right)^{\frac{1}{2}}.
\end{eqnarray}

Now,\ we have the following theorem:
\begin{theorem}\label{mainthm} Consider a perturbation of the integrable Hamiltonian
\begin{eqnarray}\label{001*}
H(x,y,z^{*},\bar{z}^{*},\xi) = N(y,z^{*},\bar{z}^{*},\xi) + R(x,y,z^{*},\bar{z}^{*},\xi)
\end{eqnarray}
defined on the domain $ D(s,r,r)\times \Pi $,\ where
\[
 N(y,z^{*},\bar{z}^{*},\xi)=\langle\omega(\xi),y\rangle+\langle\Omega_{0}(\xi)z_{0},\bar{z}_{0}\rangle
          +\langle\Omega(\xi)z,\bar{z}\rangle,\ \Omega_{0}={0}
\]
is a family of parameter dependent integrable Hamiltonian and
\[
R(x,y,z^{*},\bar{z}^{*},\xi)=\sum_{k\in\mathbb{Z}^{n},\alpha\in \mathbb{N}^{n},\beta,\gamma\in \mathbb{N}^{\mathbb{N}}}
\widehat{R}^{\alpha\beta\gamma}(k,\xi)e^{\mathbf{i}\langle k,x\rangle}y^{\alpha}\{z^{*}\}^{\beta}\{\bar{z}^{*}\}^{\gamma}
\]
is the perturbation.\ Suppose the tangential frequencies and the normal frequencies satisfy the following assumptions:\\
$ \mathbf{(A)} $: $\mathbf{Nondegeneracy}$.\ The map $ \xi \hookrightarrow \omega(\xi) $ is a lipeomorphism between $\Pi$
and its image,\ that is,\ a homeomorphism which is Lipschitz continuous in both directions.\ Moreover,\ for all integer vectors $ (k,l)\in \mathbb{Z}^{n}\times \mathbb{Z}^{\infty} $ with $ 1 \leq |l| \leq 2 $,
\begin{eqnarray}\label{c1}
\mid \{\xi : \langle k,\omega(\xi)\rangle + \langle l,\Omega(\xi)\rangle = 0 \} \mid = 0,
\end{eqnarray}
and
\begin{eqnarray}\label{c2}
\langle l,\Omega(\xi)\rangle \neq 0 \   on \ \Pi,
\end{eqnarray}
where $ |\cdot|$ denotes Lebesgue measure for sets,\ $ |l|=\sum_j|l_{j}|$ for integer vectors,\ and $ \langle \cdot \rangle$ is the usual scalar product.\\
$\mathbf{(B)} $:\ $\mathbf{Spectral\  Asymptotics}$.\ There exists $ d \geq 1 $ and $ \delta < d-1 $ such that
\begin{eqnarray}\label{c3}
\Omega_{j}(\xi)= j^{d} + ... + O(j^{\delta}),\  j\in \mathbb{N}_{+}\setminus \mathcal{J};
\end{eqnarray}
where the dots stands for fixed lower order terms in $ j $,\ allowing also negative exponents.\ More precisely,\ there exists a fixed,parameter-independent sequence $ \overline{\Omega} $ with $ \overline{\Omega}= j^{d} + ...$ such that the tails $\widetilde{\Omega} = \Omega-\overline{\Omega}$ give rise to a Lipschitz map
\[
\widetilde{\Omega}: \Pi \rightarrow l^{-\delta}_{\infty},
\]
where $ l^{p}_{\infty} $ is the space of all real sequences with finite norm $ |v|_{p} = \max\left\{\underset{j_{m}\in \mathcal{J}}{\sup} |v_{j_{m}}|j_{m}^{p},\underset{j\in \mathbb{N}_{+}\setminus \mathcal{J}}{\sup} |v_{j}|j^{p}\right\} $.\\
$\mathbf{(C)} $:\ $\mathbf{Regularity}$.\ The perturbation $ R $ is real analytic in the space coordinates and Lipschitz
in the parameters,\ and for each $\xi\in\Pi$ its Hamiltonian vector field $ X_{R}=(R_{y},-R_{x},\mathbf{i}R_{\bar{z}^{*}},-\mathbf{i}R_{{z^{*}}})^{T} $ defines near $ \mathcal{T}^{n}_{0}$ a real analytic map
\begin{eqnarray}\nonumber
X_{R}:\mathcal{P}^{a,p} \rightarrow \mathcal{P}^{a,\bar{p}},
\begin{cases}
\bar{p}\geq p\  \mbox{for}\  d > 1,\\
\bar{p} > p \  \mbox{for}\  d = 1.
\end{cases}
\end{eqnarray}
We may assume that $ p-\bar{p} \leq \delta < d-1 $ by increasing $ \delta $,\ if necessary.\ And we may also assume that
\[
|\omega|_{\Pi}+|\Omega|_{-\delta,\Pi}\leq M <\infty,\
|\omega|^{-1}_{\omega(\Pi)} \leq L < \infty,
\]
and for $ d=1$,\ we have a $ \kappa > 0$ and a constant $ a \geq 1$ such that
\[
\left|\frac{\Omega_{i}-\Omega_{j}}{i-j}-1\right|\leq \frac{a}{j^{\kappa}},\  i\neq j\in \mathbb{N}_{+}\setminus \mathcal{J}.
\]
The perturbation $ R(x,y,z^{*},\bar{z}^{*},\xi)$ also satisfies the small assumption:
\[
\varepsilon=\|X_{R}\|_{D(s,r,r),\Pi}\leq c\gamma,
\]
where $\gamma\in (0,1] $ is another parameter,\ and $ c $ depends on $ n,s,r $.\
Then there exists a subset $ \Pi_{\gamma} \subset \Pi $ with the estimate
\[
\begin{split}
\mathbf{Meas}\Pi_{\gamma} \geq (\mathbf{Meas}\Pi)(1-O(\gamma)).
\end{split}
\]
For each $ \xi \in \Pi_{\gamma}$,\ there is a symplectic map
\[
\Phi:\ D(\frac{7}{16}s,0,0)\times \Pi_{\gamma}\rightarrow D(s,r,r)\times \Pi,
\]
such that $H$ is conjugated to

\[
\breve{H} =\breve{N}(y,z^{*},\bar{z}^{*},\xi)+\breve{R}(x,y,z^{*},\bar{z}^{*},\xi),
\]
where
\begin{eqnarray}
\nonumber \breve{N}(y,z^{*},\bar{z}^{*},\xi) &=& \breve{N}^{x}(\xi)+\langle\breve{\omega}(\xi),y\rangle
          +\langle\breve{\Omega}(\xi)z,\bar{z}\rangle+ \langle\breve{N}^{z_{0}}(\xi),z_{0}\rangle +\langle\breve{N}^{\bar{z}_{0}}(\xi),\bar{z}_{0}\rangle
\\
 \nonumber    && +\langle\breve{N}^{z_{0}z_{0}}(\xi)z_{0},z_{0}\rangle
 +\langle\breve{N}^{z_{0}\bar{z}_{0}}(\xi)z_{0},\bar{z}_{0}\rangle
+\langle\breve{N}^{\bar{z}_{0}\bar{z}_{0}}(\xi)\bar{z}_{0},\bar{z}_{0}\rangle,
\end{eqnarray}
and
\[
\begin{split}
\breve{R}(x,y,z^{*},\bar{z}^{*},\xi) &=\sum_{k\in\mathbb{ Z}^{n},\alpha\in \mathbb{N}^{n},\beta,\gamma\in \mathbb{N}^{\mathbb{N}},2|\alpha|+|\beta|+|\gamma|\geq3}
\widehat{\breve{R}^{\alpha\beta\gamma}}(k,\xi)e^{\mathbf{i}\langle k,x\rangle}y^{\alpha}\{z^{*}\}^{\beta}\{\bar{z}^{*}\}^{\gamma}.
\end{split}
\]
Moreover,\ the following estimates hold:\\

$(\mathbf{1})$ for each $ \xi\in \Pi_{\gamma}$,\ the symplectic map
\[
\Phi:\ D(\frac{7}{16}s,0,0)\times \Pi_{\gamma}\rightarrow D(s,r,r)\times \Pi,
\]
satisfies:
\[
||\Phi-id||_{D(\frac{7}{16}s,0,0),\Pi_{\gamma}}
\lessdot \varepsilon,
\]
and
\[
\|D\Phi-Id\|_{D(\frac{7}{16}s,0,0),\Pi_{\gamma}}
 \lessdot \varepsilon;
\]

$(\mathbf{2})$ the frequencies $\breve{\omega}(\xi) $ and $\breve{\Omega}(\xi)$ satisfy:
\[
\begin{split}
|\breve{\omega}(\xi)-\omega(\xi)|_{\Pi_{\gamma}}
+
|\breve{\Omega}(\xi)-\Omega(\xi)|_{-\delta,\Pi_{\gamma}}
& \lessdot \varepsilon;
\end{split}
\]

$(\mathbf{3})$ the perturbation $\breve{R}(x,y,z^{*},\bar{z}^{*},\xi) $ satisfies:
\[
\begin{split}
\|X_{\breve{R}}\|_{D(\frac{7}{16}s,0,0),\Pi_{\gamma}} \lessdot\varepsilon.
\end{split}
\]

Furthermore,\ if$ \sqrt{|\breve{N}^{z_{0}}(\xi)|_{{2}}^{2}
+|\breve{N}^{\bar{z}_{0}}(\xi)|_{{2}}^{2}}:=\delta_0=  0 $,\ then there exist a Cantor set $ \Pi_{\gamma} \subset \Pi $,\ a Lipschitz continuous family of tori embedding $ \Phi: \mathbb{T}^{n}\times \Pi_{\gamma} \rightarrow \mathcal{P}^{a,\bar{p}}$,\ and a Lipschitz continuous map $ \omega_{\ast}:\Pi_{\gamma} \rightarrow \mathbb{R}^{n}$,\ such that for each $ \xi $ in $ \Pi_{\gamma}$ the map $ \Phi $ restricted to $ \mathbb{T}^{n}\times \{ \xi\}$ is a real analytic embedding of a rotational torus with the frequencies $\omega_{\ast}$ for the Hamiltonian system defined by (\ref{001*}) at $ \{\xi\}$.\\
If  $ \sqrt{|\breve{N}^{z_{0}}(\xi)|_{{2}}^{2}
+|\breve{N}^{\bar{z}_{0}}(\xi)|_{{2}}^{2}}= \delta_{0} > 0 $,\ then there exist a fixed $ m$ which is large enough such that $ \delta_{0} > 20\varepsilon_{m-1}^{\frac{7}{6}}$ ,\ a district $\Xi_{m} = \{(x,y,z^{*},\bar{z}^{*}):|\Im x| \leq s_{m},|y|\leq r_{m}^{2},
\|z^{*}\|_{a,p}+\|\bar{z}^{*}\|_{a,p}\leq \varepsilon^{\frac{7}{6}}_{m-1} \}$,\ a Cantor set $ \Pi_{m} \subset \Pi $,\ a Lipschitz continuous family of embedding $ \Phi^{m-1}: {\Xi}_{m}\times \Pi_{m} \rightarrow{D}_{0}\times \Pi_{0} $ ($ {D}_{0}={D}_{0} $ and $ \Pi_{0}=\Pi$),\ and a Lipschitz continuous map $ \omega_{m}: \Pi_{m} \rightarrow \mathbb{R}^{n}$,\ such that for each $ \xi $ in $ \Pi_{m}$,\ there is no torus in the domain $ \Phi^{m-1}(\Xi_{m}\times \{\xi\})\subset{D}\times \Pi $ for the Hamiltonian system defined by (\ref{001*}).
\end{theorem}

Some remarks and a `` guide to the proof '' of  Theorem \ref{mainthm}.

\subsection{} The classical KAM theory is also developed to deal with some one dimensional partial differential equations (PDEs) of unbounded perturbation.\ See, for example, \cite{be2011br,be2013ka,ba2014ka,ba2016ka,ba2018ti,be2019la,fe2015qu,ka2003ka,ku2000an,li2011ka,zh2011ka} and references therein.\ For the degenerate case the sequence $ \{\Omega_{j}\}$ is dense at some finite-point,\ in \cite{yu2018ka} it is showed that some shallow water equations such as Benjamin-Bona-Mahony equation and the generalized $d$-dimensional Pochhammer-Chree equation subject to some boundary conditions possess many (a family of initial values of positive Lebesgue measure of finite dimension)
smooth solutions which are quasi-periodic in time. KAM theory is also applied to many other partial differential equations, for example, \ see \cite{el2016ka,ge2003ka} for the application of KAM theory to beam equation.

\subsection{}

Conditions (\ref{c1})-(\ref{c3}) are the same as those in \cite{po1996ak} when we only consider non-zero normal frequencies. \ See Section 5 in \cite{po1996ak} for more details.

\subsection{}
The basic tool for the proof of Theorem \ref{mainthm} is the usual Newton type iteration,\ as often happens in KAM theory. However, the arguments used in this paper are quite complicated sine zero normal frequencies come out.\ Therefore,\ we give a `` guide to the proof '' of  Theorem \ref{mainthm} for the readers' convenience:
\subsubsection{ The linearized equation (Section 2)}
~~~~~~~~~~~~~~~~~~~~~~~~~~~~~~~~~~~~~~~~~~~~~~~~~~~~~~~~~~~~~~~~~~~~~~~~~~~~~~~~~~~~~~~~~~~~~~~~~~~~~~

This is the heart of the proof.\ The idea consists in a quadratic-convergent iterative procedure apt to reduce at each step of the scheme,\ which is done in order to beat the divergence introduced by small divisors arising in the inversion of non-elliptic differential operators.\ In this paper,\ since there are finite zero normal frequencies,\
The main difficulties we encounter are
\begin{eqnarray}
\label{*}\langle k,\omega(\xi)\rangle \pm \Omega_{j_{m}}=0,\ 1\leq m \leq b,
\end{eqnarray}
and
\begin{eqnarray}
\label{**}\langle k,\omega(\xi)\rangle  \pm\Omega_{j_{m}} \pm \Omega_{j_{n}}=0,\ 1\leq m,n \leq b,
\end{eqnarray}
when $ k= 0$.\\
To overcome these difficulties,\ a basic idea for this article is that we preserve the terms related to (\ref{*}) and (\ref{**}).\\
Let
\begin{eqnarray}\nonumber
R &=& R^{x}(\xi)+\langle R^{y}(x,\xi),y\rangle+\langle R^{z_{0}z}(x,\xi)z_{0},z\rangle
+\langle R^{z_{0}\bar{z}}(x,\xi)z_{0},\bar{z}\rangle\\
\nonumber&&+\langle R^{\bar{z}_{0}{z}}(x,\xi)\bar{z}_{0},{z}\rangle
+\langle R^{\bar{z}_{0}\bar{z}}(x,\xi)\bar{z}_{0},\bar{z}\rangle+\langle R^{z}(x,\xi),z\rangle+\langle R^{\bar{z}}(x,\xi),\bar{z}\rangle\\
\nonumber&&
+\langle R^{zz}(x,\xi)z,z\rangle
+\langle R^{z\bar{z}}(x,\xi)z,\bar{z}\rangle
+\langle R^{\bar{z}\bar{z}}(x,\xi)\bar{z},\bar{z}\rangle\\
\nonumber&&+\langle R^{z_{0}}(x,\xi),z_{0}\rangle+\langle R^{\bar{z}_{0}}(x,\xi),\bar{z}_{0}\rangle
+\langle R^{z_{0}z_{0}}(x,\xi)z_{0},z_{0}\rangle\\
\nonumber&&
+\langle R^{z_{0}\bar{z}_{0}}(x,\xi)z_{0},\bar{z}_{0}\rangle
+\langle R^{\bar{z}_{0}\bar{z}_{0}}(x,\xi)\bar{z}_{0},\bar{z}_{0}\rangle.
\end{eqnarray}
Thus,\ the terms we will preserve are $ \widehat{R^{z_{0}}}(0,\xi),\widehat{R^{\bar{z}_{0}}}(0,\xi),\widehat{R^{z_{0}z_{0}}}(0,\xi),
\widehat{R^{z_{0}\bar{z}_{0}}}(0,\xi),
\widehat{R^{\bar{z}_{0}\bar{z}_{0}}}(0,\xi)$.\\
Since
\[
H_{0}= N_{0}+R_{0}=\langle\omega(\xi),y\rangle+\langle\Omega_{00}(\xi)z_{0},\bar{z}_{0}\rangle
          +\langle\Omega_{0}(\xi)z,\bar{z}\rangle,\  \Omega_{00}=0,
\]
there exists a real-analytic symplectic transformation $ \Phi_{0}$,\ such that
\begin{eqnarray}\label{m}
H_{0}\circ\Phi_{0}= ( N_{0}+R_{0})\circ \Phi_{0}= N_{1}+R_{1}=H_{1},
\end{eqnarray}
where the new normal form
\begin{eqnarray}
\label{q1} N_{1} &=& \widehat{N^{x}_{0}}(\xi)+\sum_{j=1}^{n}\omega_{1}^{j}(\xi)y_{j}
+\sum_{j\in \mathbb{N_{+}}\setminus \mathcal{J}}\Omega_{1}^{j}(\xi)z_{j}\bar{z}_{j}
      + \langle\widehat{N^{z_{0}}_{0}}(\xi),z_{0}\rangle\\
 \nonumber&& +\langle\widehat{N^{\bar{z}_{0}}_{0}}(\xi),\bar{z}_{0}\rangle
 +\langle\widehat{N^{z_{0}z_{0}}_{0}}(\xi)z_{0},z_{0}\rangle
 +\langle\widehat{N^{z_{0}\bar{z}_{0}}_{0}}(\xi)z_{0},\bar{z}_{0}\rangle
+\langle\widehat{N^{\bar{z}_{0}\bar{z}_{0}}_{0}}(\xi)\bar{z}_{0},\bar{z}_{0}\rangle,
\end{eqnarray}
while the new perturbation is of smaller size:
\begin{equation}\nonumber
\begin{split}
\|X_{R_{1}}\|_{D(s_{1},r_{1},r_{1}),\Pi_{1}}\lessdot \|X_{R_{0}}\|^{\frac{4}{3}}_{D(s_{0},r_{0},r_{0}),\Pi_{0}}.
\end{split}
\end{equation}
The parameter $ \xi$ appearing in (\ref{q1}) will vary in small compact set $ \Pi_{1}$ (of relatively large Lebesgue measure).\\
Obviously,\ after $1$-th iteration,\ we obtain a new normal form $ N_{1}$,\ which has more terms than the usual KAM normal form.\ Thus,\ our iterative scheme from $ H_{1}$ is non-standard and,\ from a technical point of view,\ represents the most novel part of the proof.

Similarly,\ for $ H_1$ in (\ref{m}),\ there exists a real-analytic symplectic transformation $ \Phi_{1}$,\ such that
\[
H_{1}\circ\Phi_{1}= ( N_{1}+R_{1})\circ \Phi_{1}= N_{2}+R_{2}=H_{2},
\]
where the new normal form
\begin{eqnarray}
\nonumber N_{2} &=& \sum^{v-1}_{j=0}\widehat{N^{x}_{j}}(\xi)+\sum_{j=1}^{n}\omega_{2}^{j}(\xi)y_{j}
+\sum_{j\in \mathbb{N_{+}}\setminus \mathcal{J}}\Omega_{2}^{j}(\xi)z_{j}\bar{z}_{j}
      + \langle\sum^{1}_{j=0}\widehat{N^{z_{0}}_{j}}(\xi),z_{0}\rangle\\
 \nonumber&& +\langle\sum^{1}_{j=0}\widehat{N^{\bar{z}_{0}}_{j}}(\xi),\bar{z}_{0}\rangle
 +\langle\sum^{1}_{j=0}\widehat{N^{z_{0}z_{0}}_{j}}(\xi)z_{0},z_{0}\rangle
 +\langle\sum^{1}_{j=0}\widehat{N^{z_{0}\bar{z}_{0}}_{j}}(\xi)z_{0},\bar{z}_{0}\rangle
+\langle\sum^{1}_{j=0}\widehat{N^{\bar{z}_{0}\bar{z}_{0}}_{j}}(\xi)\bar{z}_{0},\bar{z}_{0}\rangle,
\end{eqnarray}
while the new perturbation is of smaller size:
\begin{equation}\nonumber
\begin{split}
\|X_{R_{2}}\|_{D(s_{2},r_{2},r_{2}),\Pi_{2}}\lessdot \|X_{R_{1}}\|^{1+\vartheta}_{D(s_{1},r_{1},r_{1}),\Pi_{1}},
\end{split}
\end{equation}
where $ \vartheta \in (\frac{1}{6},\frac{1}{3})$ and the parameter $ \xi$ will vary in small compact set $ \Pi_{2}$ (of relatively large Lebesgue measure).\\
Since the preserved terms are put into the normal form $ N_{1}$,\ the homological equations in this iteration are of the following forms
\begin{eqnarray}
\label{1}\omega\cdot \partial_{x}F_{1} + A_{1}F_{1} + F_{1}B_{1} = R_{1},\\
\label{2}\omega\cdot \partial_{x}F_{2} + A_{2}F_{2}  = R_{2},\\
\label{3}\omega\cdot \partial_{x}F_{3} + \Lambda F_{3} + F_{3}\Lambda = R_{3},
\end{eqnarray}
where $ A_{1},A_{2},B_{1}$ depending only on $ \xi$ are not diagonal while $ \Lambda$ is diagonal.\  (Different from $ 1-$th iteration with normal form $ N_{0}$,\ (\ref{3}) is the only homological equations we have to solve.)\ More narrowly,\ equation (\ref{1}) is derived from the homological equation of the coefficients $ F^{z_{0}z_{0}},F^{z_{0}\bar{z}_{0}},F^{\bar{z}_{0}\bar{z}_{0}}$,\ whose any $k$-th Fourier coefficient matrixes related to the preserved terms are finite dimension (less than $ 4b^{2}\times 4b^{2}$).\ Thus,\ by introducing Kronecker product and column straightening,\ the coefficient equation can be solved provided that its any $ k$ -th
Fourier coefficient matrixes are non-degenerate and satisfy some non-resonant conditions.\ For equations (\ref{2}) and (\ref{3}),\ they are also solvable as long as any $ k$ -th
Fourier coefficient matrixes are non-degenerate and satisfy some non-resonant conditions.\ Therefore,\ once small divisor conditions will be given appropriately (See subsection \ref{non} for more small divisor conditions),\ any estimates we need can be obtained by some complicated computations and the KAM machinery still works well.

\subsubsection{ The iterative lemma (Section 4)}
~~~~~~~~~~~~~~~~~~~~~~~~~~~~~~~~~~~~~~~~~~~~~~~~~~~~~~~~~~~~~~~~~~~~~~~~~~~~~~~~~~~~~~~~~~~~~~~~~~~~~~

We want to construct,\ inductively,\ real-analytic symplectic transformations $ \Phi_{m}, m\geq 0$,\ such that
\begin{eqnarray}\label{q3}
H_{m}\circ\Phi_{m}= ( N_{m}+R_{m})\circ \Phi_{m}= N_{m+1}+R_{m+1}=H_{m+1},
\end{eqnarray}
where the sequences of the new normal form $ N_{m+1}$
\begin{eqnarray}
\nonumber N_{m+1} &=& \sum^{m}_{j=0}\widehat{N^{x}_{j}}(\xi)+\sum_{j=1}^{n}\omega_{m+1}^{j}(\xi)y_{j}
+\sum_{j\in \mathbb{N_{+}}\setminus \mathcal{J}}\Omega_{m+1}^{j}(\xi)z_{j}\bar{z}_{j}
      + \langle\sum^{m}_{j=0}\widehat{N^{z_{0}}_{j}}(\xi),z_{0}\rangle\\
 \nonumber&& +\langle\sum^{m}_{j=0}\widehat{N^{\bar{z}_{0}}_{j}}(\xi),\bar{z}_{0}\rangle
 +\langle\sum^{m}_{j=0}\widehat{N^{z_{0}z_{0}}_{j}}(\xi)z_{0},z_{0}\rangle
 +\langle\sum^{m}_{j=0}\widehat{N^{z_{0}\bar{z}_{0}}_{j}}(\xi)z_{0},\bar{z}_{0}\rangle
+\langle\sum^{m}_{j=0}\widehat{N^{\bar{z}_{0}\bar{z}_{0}}_{j}}(\xi)\bar{z}_{0},\bar{z}_{0}\rangle,
\end{eqnarray}
while the sequences of perturbations $ R_{m+1}$ are of smaller and smaller size:
\begin{equation}\nonumber
\begin{split}
\|X_{R_{m+1}}\|_{D(s_{m+1},r_{m+1},r_{m+1}),\Pi_{m+1}}\lessdot \|X_{R_{m}}\|^{1+\vartheta}_{D(s_{m},r_{m},r_{m}),\Pi_{m}}.
\end{split}
\end{equation}
The parameter $ \xi$ will vary in smaller and smaller compact sets $ \Pi_{m}$ (of relatively large Lebesgue measure)
\[
\Pi_{0}\supset \Pi_{1}\supset\cdot\cdot\cdot\Pi_{m}\supset\Pi_{m+1}\supset\cdot\cdot\cdot\supset\Pi_{\infty}\supset \bigcap_{m=1}^{\infty} \Pi_{m}.
\]
The smallness assumption on $\|X_{R_{0}}\|_{D(s_{0},r_{0},r_{0}),\Pi_{0}}$ will allow to turn on the iteration procedure.

The symplectic map $\Phi^{m}$ will be sought of the form
\[
\Phi^{m}= \Phi^{m-1}\circ\Phi_{m}= \Phi_{0}\circ\cdot\cdot\cdot\circ \Phi_{m}.
\]
In order to work for the approach,\ one has to show that
\begin{eqnarray}\label{q4}
&&\Phi_{m} : D(s_{m+1},r_{m+1},r_{m+1})\rightarrow D(s_{m},r_{m},r_{m}),  ( \forall  m\geq 0),\\ \label{q5}&&\Phi^{m-1} : D(s_{m},r_{m},r_{m})\rightarrow D(s_{0},r_{0},r_{0}),  ( \forall  m\geq 1).
\end{eqnarray}
Relations (\ref{q4}) and (\ref{q5}) are checked in Section 4.

The linearized equation associated to (\ref{q3}) is thoroughly discussed in Section 4.\ This is the place where small divisors arise.\ Such divisors have the form
\begin{itemize}
\item[(1)]
$ \mathcal{R}_{kl}(m)=\{\xi\in\Pi_{m}:\mid \langle k,\omega_{m}(\xi)\rangle+\langle l,\Omega_{m}(\xi)\rangle\mid<\frac{\gamma_{m}\langle l\rangle_{d}}{|k|^{\tau}},K_{m-1}< |k|\leq K_{m},(k,l)\in\mathcal{Z}\}$,
\item[(2)]
$ \mathcal{R}_{1k}(m) =\{\xi \in \Pi_{m}:\mid |\mathbf{i}\langle k,\omega_{m}(\xi)\rangle I_{3b^{2}}- B_{1m}(\xi)|_{d}\mid < \frac{\gamma_{1m}}{|k|^{\tau_{1}}},0 <|k|\leq K_{m} \}$,\\
where $ \tau_{1}= 3b^{2}\tau, \gamma_{1m}=\gamma_{m}/ m^{18b^{4}},\ L_{1}=3b^{2}\times4b^{2}(diam\Pi_{m})^{n-1}$,
\item[(3)]
$ \mathcal{R}_{3k}(m) =\{\xi \in \Pi_{m}:\mid |\mathbf{i}(\langle k,\omega_{m}(\xi)\rangle\pm \Omega^{j}_{m}) I_{4b^{2}}+B_{3m}(\xi)|_{d}\mid<\frac{\gamma_{3m}}{|k|^{\tau_{3}}}, |k|\leq K_{m} \}$,\\
where  $ \tau_{3}= 4b^{2}\tau, \gamma_{3m}=\gamma_{m}/ m^{32b^{4}},\ L_{3}=4b^{2}\times5b^{2}(diam\Pi_{1})^{n-1}$,
\item[(4)]
$ \mathcal{R}_{4k}(m) =\{\xi \in \Pi_{m}:\mid |\mathbf{i}\langle k,\omega_{m}(\xi)\rangle I_{2b^{2}}+B_{4m}(\xi)|_{d}\mid<\frac{\gamma_{4m}}{|k|^{\tau_{4}}},0 <|k|\leq K_{m} \}$,\\
where $\tau_{4}= 2b^{2}\tau, \gamma_{4m}=\gamma_{m}/ m^{8b^{4}}, L_{4}=2b^{2}\times3b^{2}(diam\Pi_{1})^{n-1}$,
\end{itemize}
where $ |\cdot|_{d}$ denotes the determinant of a matrix and $ K_{m}$ is a suitable Fourier ``cut-off'' introduced originally by Arnold \cite{ar1963pr}.

\subsubsection{Convergence of the KAM scheme and proof of theorem \ref{mainthm}( Section 5 and Section 6 )}
~~~~~~~~~~~~~~~~~~~~~~~~~~~~~~~~~~~~~~~~~~~~~~~~~~~~~~~~~~~~~~~~~~~~~~~~~~~~~~~~~~~~~~~~~~~~~~~~~~~~~~

Once the iterative step is set up,\ it has to be equipped with estimates.\ This technique part follows the corresponding part in \cite{po1996ak}.\ Particularly,\ the key results of theorem \ref{mainthm} concerning the new Hamiltonian $ \breve{H}$ and the measure of $ \Pi_{\infty} $ follow easily.\ From the fast convergence of $ N_{m}$ to
\begin{eqnarray}\nonumber
\breve{N} &=&  \breve{N}^{x}(\xi)+\langle\breve{\omega}(\xi),y\rangle
          +\langle\breve{\Omega}(\xi)z,\bar{z}\rangle+ \langle\breve{N}^{z_{0}}(\xi),z_{0}\rangle +\langle\breve{N}^{\bar{z}_{0}}(\xi),\bar{z}_{0}\rangle
\\
 \nonumber    && +\langle\breve{N}^{z_{0}z_{0}}(\xi)z_{0},z_{0}\rangle
 +\langle\breve{N}^{z_{0}\bar{z}_{0}}(\xi)z_{0},\bar{z}_{0}\rangle
+\langle\breve{N}^{\bar{z}_{0}\bar{z}_{0}}(\xi)\bar{z}_{0},\bar{z}_{0}\rangle +\breve{R}(x,y,z^{*},\bar{z}^{*},\xi),
\end{eqnarray}
it follows that
when $ \breve{N}^{z_{0}}(\xi)=\vec{0} $ and $\breve{N}^{\bar{z}_{0}}(\xi)=\vec{0}$,\ $ \Phi(\mathcal{T}^{n}_{0} \times \{\xi\})$ is an invariant torus of $ H$;\ when $
\breve{N}^{z_{0}}(\xi)\neq \vec{0} $ or $\breve{N}^{\bar{z}_{0}}(\xi)\neq\vec{0}$,\ that is,\ $ \sqrt{|\breve{N}^{z}_{0}(\xi)|_{{2}}^{2}
+|\breve{N}^{\bar{z}_{0}}(\xi)|_{{2}}^{2}}= \delta_{0} > 0 $.\ Since $ \underset{m\rightarrow\infty}{\lim}\widehat{J_{m}^{z_{0}}} (\xi)= \breve{N}^{z_{0}}(\xi)$ and $ \underset{m\rightarrow\infty}{\lim}\widehat{J_{m}^{\bar{z}_{0}}} (\xi)= \breve{N}^{\bar{z}_{0}}(\xi)$,\ there exists a fixed $ m_{0}$ such that for any $ m > m_{0} $,
\begin{eqnarray}\label{a000*}
\sqrt{|{J}_{m}^{{z}_{0}}(\xi)|_{2}^{2}
+|{J}_{m}^{\bar{z}_{0}}(\xi)|_{2}^{2}}\geq \frac{\delta_{0}}{2}.
\end{eqnarray}
More exactly,\ we will choose sufficiently large $ m $ such that
\begin{eqnarray}\label{a00*}
 \delta_{0} > 20\varepsilon^{\frac{7}{6}}_{m-1}.
\end{eqnarray}

Consider the Hamiltonian equation defined by $ H_{m}= N_{m} + R_{m} $ and fixed an initial value $ \|z^{*}(0)\|_{a,p}+\|\bar{z}^{*}(0)\|_{a,p}\leq \varepsilon^{\frac{7}{6}}_{m-1} $.\ By making use of (\ref{a000*}), (\ref{a00*}) and some ordinary differential equation tools,\ we have
\[
\|z^{*}(1)\|_{a,p}+\|\bar{z}^{*}(1)\|_{a,p}> \varepsilon^{\frac{7}{6}}_{m-1}.
\]
That is,\ there exists no torus in  the domain $ \Phi^{m-1}(\Xi_{m}\times \{\xi\}) $ for the Hamiltonian $ H$ in (\ref{001}) when we denote $\Xi_{m} = \{(x,y,z^{*},\bar{z}^{*}):|\Im x| \leq s_{m},|y|\leq r_{m}^{2},
\|z^{*}\|_{a,p}+\|\bar{z}^{*}\|_{a,p}\leq \varepsilon^{\frac{7}{6}}_{m-1} \}$.\ Theorem \ref{mainthm},\ at this point,\ is completely proven.

\subsection{Application to nonlinear Schr\"{o}dinger equation (NLS) (Section 7)}
~~~~~~~~~~~~~~~~~~~~~~~~~~~~~~~~~~~~~~~~~~~~~~~~~~~~~~~~~~~~~~~~~~~~~~~~~~~~~~~~~~~~~~~~~~~~~~~~~~~~~~

Consider a specific nonlinear Schr\"{o}dinger equation
\begin{equation}\label{p}
\mathbf{i}u_{t}-u_{xx}+|u|^{2}u=0
\end{equation}
on the finite $x$-interval $[0,2\pi]$ with periodic boundary conditions
\[
u(t,x)=u(t,x+2\pi)=0,\  u(x,t)=u(-x,t).
\]

When applying this abstract theorem to PDEs,\ one meets two difficulties: (1) to study its structure of Hamiltonian in order to extract dynamical information; (2) to verify $ \breve{N}^{z_{0}}(\xi)=\vec{0} $ and $\breve{N}^{\bar{z}_{0}}(\xi)=\vec{0}$ or not.

Concerning (1), in the context of the NLS of (\ref{p}),\
we obtain a Hamiltonian with one zero normal frequency and other frequencies satisfying non-resonant conditions.\ Thus,\ the KAM machinery works well.\ Concerning (2),\ we find that $ \widehat{R_{m}^{z_{0}}}(0,\xi)={0} $ and $\widehat{R_{m}^{\bar{z}_{0}}}(0,\xi)={0}$ in any $ m-$th iteration.\ Therefore,\ we have $ \breve{N}^{z_{0}}(\xi)={0} $ and $\breve{N}^{\bar{z}_{0}}(\xi)={0}$,\ that is,\ there still exist many invariant tori of quasi-periodic oscillations in a sufficiently small neighborhood of the origin for the Schr\"{o}dinger equation of (\ref{p}).\ Detailed,\ quantitative results are collected in Section 7.

\section{The linearized equation}
Assume that all the assumptions of Theorem \ref{mainthm} are satisfied.\ Set $ r_{0}=r,\  s_{0}=s, \gamma_{0}=\gamma,\ \varepsilon_{0}=\varepsilon$,\ and $ H_{0}=H$.\
Recall that the Hamiltonian
\begin{eqnarray}
\label{009}H_{0}=H_{0}(x,y,z^{*},\bar{z}^{*},\xi)
=N_{0}(y,z^{*},\bar{z}^{*},\xi)+R_{0}(x,y,z^{*},\bar{z}^{*},\xi),
\end{eqnarray}
where
\begin{eqnarray}\label{010}
N_{0}(y,z^{*},\bar{z}^{*},\xi) &=&\langle\omega(\xi),y\rangle+\langle\Omega_{00}(\xi)z_{0},\bar{z}_{0}\rangle
          +\langle\Omega_{0}(\xi)z,\bar{z}\rangle\\
\nonumber &=&\sum_{1\leq j\leq n}\omega^{j}_{0}y_{j}+\sum_{j_{m}\in \mathcal{J} }\Omega^{m}_{00}z_{j_{m}}\bar{z}_{j_{m}}
          +\sum_{j\in \mathbb{N_{+}}\setminus \mathcal{J}}\Omega^{j}_{0}z_{j}\bar{z}_{j}\\
&&\nonumber \mbox{ (by noting $ z_{j_{m}}=z_{0{m}}$ for convenience) }\\
  \nonumber &=&\sum_{1\leq j\leq n}\omega^{j}_{0}y_{j}+\sum_{1\leq m \leq b }\Omega^{m}_{00}z_{0{m}}\bar{z}_{0{m}}
          +\sum_{j\in \mathbb{N_{+}}\setminus \mathcal{J}}\Omega^{j}_{0}z_{j}\bar{z}_{j},
\end{eqnarray}
with $ \Omega^{m}_{00}=0, 1\leq m \leq b$ and $ 0 <... \Omega^i_{0}<...<\Omega^{j}_{0}<...\rightarrow +\infty $ for $ i< j \in \mathbb{N_{+}}\setminus \mathcal{J}$.\\
Denote $ R_{0}(x,y,z^{*},\bar{z}^{*},\xi)=R_{0}^{low}(x,y,z^{*},\bar{z}^{*},\xi)+R_{0}^{high}(x,y,z^{*},\bar{z}^{*},\xi)$.\ Then we have
\begin{eqnarray}
\label{011} R_{0}^{low}&=&\sum_{\alpha\in \mathbb{N}^{n},\beta,\gamma\in \mathbb{N}^{\mathbb{N}},2|\alpha|+|\beta|+|\gamma|\leq2}R_{0}^{\alpha\beta\gamma}(x,\xi)
    y^{\alpha}\{z^{*}\}^{\beta}\{\bar{z}^{*}\}^{\gamma},\\
\label{012}R_{0}^{high}&=&\sum_{\alpha\in \mathbb{N}^{n},\beta,\gamma\in \mathbb{N}^{\mathbb{N}},2|\alpha|+|\beta|+|\gamma|\geq3}
R_{0}^{\alpha\beta\gamma}(x,\xi)y^{\alpha}\{z^{*}\}^{\beta}\{\bar{z}^{*}\}^{\gamma}.
\end{eqnarray}
We desire to eliminate the terms $ R_{0}^{low}$ by the coordinate transformation $ \Phi_{0} $,\ which is obtained as the time-1-map $ X_{F_{0}}^{t}\mid_{t=1}$ of a Hamiltonian vector field $ X_{F_{0}} $,\ where $F_{0}(x,y,z^{*},\bar{z}^{*},\xi)$ is of the form
\begin{eqnarray}
\nonumber F_{0}(x,y,z^{*},\bar{z}^{*},\xi) &=&F_{0}^{low}(x,y,z^{*},\bar{z}^{*},\xi)\\
 \nonumber                         &=&\sum_{\alpha\in \mathbb{N}^{n},\beta,\gamma\in \mathbb{N}^{\mathbb{N}},2|\alpha|+|\beta|+|\gamma|\leq2}
F_{0}^{\alpha\beta\gamma}(x,\xi)y^{\alpha}\{z^{*}\}^{\beta}\{\bar{z}^{*}\}^{\gamma}.
\end{eqnarray}
Using Taylor formula,\ we have
\begin{eqnarray}
 \label{013} H_{1} &=& H_{0} \circ X^t_{F_{0}}\mid_{t=1} \\
 \nonumber     &=& N_{0}+\{N_{0},F_{0}\}+ \int_0^1 (1-t)\{\{N_{0},F_{0}\},F_{0} \}\circ X^t_{F_{0}}\mathrm{d}t \\
 \nonumber      &&+R_{0}^{low}+\int_0^1\{R_{0}^{low},F_{0}\}\circ X^t_{F_{0}}\mathrm{d}t+R_{0}^{high}\circ X^t_{F_{0}}\mid_{t=1}.
\end{eqnarray}
Then we obtain the modified homological equation
\begin{eqnarray}
\label{014} {N_{0}}+\{N_{0},F_{0}\}+ R_{0}^{low} = N_{1},
\end{eqnarray}
where
\begin{eqnarray}\label{015}
N_{1}&=& N_{0}+\widehat{N_{0}}\\
\nonumber &=& N_{0}+\widehat{R^{x}_{0}}(0,\xi)+\langle\widehat{R^{y}_{0}}(0,\xi),y\rangle+
\langle\widehat{R^{z_{0}}_{0}}(0,\xi),z_{0}\rangle
+\langle\widehat{R^{\bar{z}_{0}}_{0}}(0,\xi),\bar{z}_{0}\rangle
+\langle\widehat{R^{z_{0}z_{0}}_{0}}(0,\xi)z_{0},z_{0}\rangle
\\
\nonumber&&+\langle\widehat{R^{z_{0}\bar{z}_{0}}_{0}}(0,\xi)z_{0},\bar{z}_{0}\rangle+
\langle\widehat{R^{\bar{z}_{0}\bar{z}_{0}}_{0}}(0,\xi)\bar{z}_{0},\bar{z}_{0}\rangle+ \sum_{j\in \mathbb{N_{+}}\setminus \mathcal{J}}\widehat{R^{z_{j}\bar{z}_{j}}_{0}}(0,\xi)z_{j}\bar{z}_{j},
\end{eqnarray}
and
\begin{eqnarray}\label{016}
R_{1} &=& \int_0^1 (1-t)\{\{N_{0},F_{0}\}+R_{0}^{low},F_{0}\}\circ X^t_{F_{0}}\mathrm{d}t + R_{0}^{high}\circ X^t_{F_{0}}\mid_{t=1}.
\end{eqnarray}

For convenience,\  for any $ j\geq 0 $,\ we also note
\begin{eqnarray}
\nonumber \widehat{N^{x}_{j}}(\xi)&=&\widehat{R^{x}_{j}}(0,\xi),\\
\nonumber\widehat{N^{y}_{j}}(\xi)&=&\widehat{R^{y}_{j}}(0,\xi),\\
\nonumber\widehat{N^{z_{0}}_{j}}(\xi)&=&\widehat{R^{z_{0}}_{j}}(0,\xi),\\
\nonumber\widehat{N^{\bar{z}_{0}}_{j}}(\xi)&=&\widehat{R^{\bar{z}_{0}}_{j}}(0,\xi),
\end{eqnarray}
and
\begin{eqnarray}
\nonumber\widehat{N^{z_{0}z_{0}}_{j}}(\xi)&=&\widehat{R^{z_{0}z_{0}}_{j}}(0,\xi),\\
\nonumber\widehat{N^{z_{0}\bar{z}_{0}}_{j}}(\xi)&=&\widehat{R^{z_{0}\bar{z}_{0}}_{j}}(0,\xi),\\
\nonumber\widehat{N^{\bar{z}_{0}\bar{z}_{0}}_{j}}(\xi)
&=&\widehat{R^{\bar{z}_{0}\bar{z}_{0}}_{j}}(0,\xi),\\
\nonumber\widehat{N^{z_{k}\bar{z}_{k}}_{j}}(\xi)&=&\widehat{R^{z_{k}\bar{z}_{k}}_{j}}(0,\xi),
\end{eqnarray}
where $ k\in \mathbb{N_{+}}\setminus \mathcal{J}$ for the last term.

\subsection{\ The solution of homological equation (\ref{014})} Following Kuksin and P\"{o}schel's notations in \cite{ku1996in},\ we have
\begin{lemma}\label{a} Consider a perturbation of the integrable Hamiltonian
\[
H_{0}=H_{0}(x,y,z^{*},\bar{z}^{*},\xi)=N_{0}(y,z^{*},\bar{z}^{*},\xi)+R_{0}(x,y,z^{*},\bar{z}^{*},\xi),
\]
where
\begin{eqnarray}\label{017}
N_{0}(y,z^{*},\bar{z}^{*},\xi) = \sum_{1\leq j\leq n}\omega_{0}^{j}(\xi)y_{j}+\sum_{1\leq m \leq b }\Omega^{m}_{00}z_{0{m}}\bar{z}_{0{m}}
          +\sum_{j\in \mathbb{N_{+}}\setminus \mathcal{J}}\Omega^{j}_{0}(\xi)z_{j}\bar{z}_{j}
\end{eqnarray}
is a parameter dependent integrable Hamiltonian and
\[
R_{0}(x,y,z^{*},\bar{z}^{*},\xi)=R_{0}^{low}(x,y,z^{*},\bar{z}^{*},\xi)+R_{0}^{high}(x,y,z^{*},\bar{z}^{*},\xi).
\]
Suppose assumption $ \mathbf{(A)} $ and $ \mathbf{(B)} $ are fulfilled for $ \omega_{0}(\xi)$ and $ \Omega_{0}(\xi)$,\
\begin{eqnarray}\label{017*}
\|X_{R^{low}_{0}}\|_{D(s_{0},r_{0},r_{0}),\Pi_{0}}& \leq\|X_{R_{0}}\|_{D(s_{0},r_{0},r_{0}),\Pi_{0}},
\end{eqnarray}
and
\begin{eqnarray}\label{017**}
\|X_{R^{high}_{0}}\|_{D(s_{0},r_{0},r_{0}),\Pi_{0}}& \leq 1,
\end{eqnarray}
for some $ 0 < s_{0},r_{0} \leq 1$.\ For some fixed constant $ \tau > n+1 $,\ let
\begin{eqnarray}\label{018}
\mathcal{R}_{kl}(0)=\{\xi\in\Pi_{0}: |\langle k,\omega_{0}(\xi)\rangle+\langle l,\Omega_{0}(\xi)\rangle|<\frac{\gamma_{0}\langle l\rangle_{d}}{1+ |k|^{\tau}}, |k|> K_{0},
(k,l)\in \mathcal{Z}\},
\end{eqnarray}
where $ \langle l\rangle_{d}= \max(1,|\sum j^{d}l_{j}|),\mathcal{Z}=\{(k,l)| (k,l)\neq 0, |l|\leq 2\}\subset \mathbb{Z}^{n}\times\mathbb{Z}^{\infty}$,\
and let
\begin{eqnarray}\label{019}
\Pi_{1}=\Pi_{0}\setminus\bigcup_{|k|> K_{0},(k,l)\in \mathcal{Z}}\mathcal{R}_{kl}(0),
\end{eqnarray}
where $  K_{0}$ will be given later.\ Then for each $\xi\in\Pi_{1}$,\ the homological equation has a solution $ F_{0}(x,y,z^{*},\bar{z}^{*},\xi)$ with the estimates
\begin{equation}\label{020}
\begin{split}
\|X_{F_{0}}\|_{D(s_{0}-\sigma_{0},r_{0},r_{0}),\Pi_{1}}&\lessdot \frac{B_{\sigma_{0}}}{\gamma_{0}}\| X_{R_{0}}\|_{D(s_{0},r_{0},r_{0}),\Pi_{0}},\\
\|X_{\widehat{N_{0}}}\|_{D(s_{0},r_{0},r_{0}),\Pi_{1}}&\lessdot \|X_{R_{0}}\|_{D(s_{0},r_{0},r_{0}),\Pi_{0}},
\end{split}
\end{equation}
where $ 0<\sigma_{0}=s_{0}/40\leq\frac{1}{4},B_{\sigma_{0}}=\sum_{k}(1+|k|)^{2}(1+|k|^{\tau})^{4}e^{-2|k|\sigma_{0}},
t=2\tau+n+2 $ and $ a\lessdot b$ means there exists a constant $ c>0$ depending on  $ n $ and $\tau$ such that $ a\leq cb$.
Moreover,\ let
\[
\gamma_{1}=\frac{3}{4}\gamma_{0},M_{1}=\frac{3}{2}M_{0},s_{1}=s_{0}-5\sigma_{0},r_{1}=\eta_{0}r_{0},
\eta^{3}_{0}=\frac{\varepsilon_{0}}{\gamma_{0}\sigma^{t}_{0}},K^{\tau+1}_{0}=\frac{1}{\gamma_{0}},
\varepsilon_{1}=\frac{\varepsilon^{\frac{4}{3}}_{0}}{(\gamma_{0}\sigma^{t}_{0})^{\frac{1}{3}}},
\]
then the new Hamiltonian $ H_{1}(x,y,z,\bar{z},\xi)$ has the form
\[
H_{1}(x,y,z^{*},\bar{z}^{*},\xi)=N_{1}(y,z^{*},\bar{z}^{*},\xi)+R_{1}(x,y,z^{*},\bar{z}^{*},\xi),
\]
where
\begin{eqnarray}
\nonumber N_{1} &=& \widehat{N^{x}_{0}}(\xi)+\sum_{j=1}^{n}\omega_{1}^{j}(\xi)y_{j}
+\sum_{j\in \mathbb{N_{+}}\setminus \mathcal{J}}\Omega_{1}^{j}(\xi)z_{j}\bar{z}_{j}
      + \langle\widehat{N^{z_{0}}_{0}}(\xi),z_{0}\rangle\\
 \nonumber&& +\langle\widehat{N^{\bar{z}_{0}}_{0}}(\xi),\bar{z}_{0}\rangle
 +\langle\widehat{N^{z_{0}z_{0}}_{0}}(\xi)z_{0},z_{0}\rangle
 +\langle\widehat{N^{z_{0}\bar{z}_{0}}_{0}}(\xi)z_{0},\bar{z}_{0}\rangle
+\langle\widehat{N^{\bar{z}_{0}\bar{z}_{0}}_{0}}(\xi)\bar{z}_{0},\bar{z}_{0}\rangle,
\end{eqnarray}
with
\begin{equation}\label{021}
\begin{split}
\omega_{1}^{j}(\xi)&=\omega_{0}^{j}(\xi)+\widehat{N^{y_{j}}_{0}}(\xi),\  1\leq j\leq n,\\
\Omega_{1}^{j}(\xi)&=\Omega_{0}^{j}(\xi)+\widehat{N^{z_{j}\bar{z}_{j}}_{0}}(\xi),\  j\in \mathbb{N_{+}}\setminus \mathcal{J},
\end{split}
\end{equation}
and
\[
R_{1}(x,y,z^{*},\bar{z}^{*},\xi)=\int_0^1\{ (1-t)\widehat{N_{0}}+tR^{low}_{0},F_{0}\}\circ X^t_{F_{0}}\mathrm{d}t +R_{0}^{high}\circ X^t_{F_{0}}\mid_{t=1},
\]
with the following estimates holds:\\

$(\mathbf{a})$ for each $ \xi\in \Pi_{1}$,\ the symplectic map $\Phi_{0}=X_{F_{0}}^{t}\mid_{t=1}$ satisfies:
\[
\begin{split}
\|\Phi_{0}-id\|_{r_{0},D(s_{0}-3\sigma_{0},r_{0},r_{0}/4),\Pi_{1}}&\lessdot
\|X_{F_{0}}\|_{r_{0},D(s_{0}-\sigma_{0},r_{0},r_{0}),\Pi_{0}},\\
\|D\Phi_{0}-Id\|_{r_{0},r_{0},D(s_{0}-4\sigma_{0},r_{0},r_{0}/8),\Pi_{1}}&\lessdot \frac{1}{\sigma_{0}}\| X_{F_{0}}\|_{r_{0},D(s_{0}-\sigma_{0},r_{0},r_{0}),\Pi_{0}};
\end{split}
\]

$(\mathbf{b})$ the frequencies $\omega_{1}(\xi) $ and $\Omega_{1}(\xi)$ satisfy:
\[
\begin{split}
|\omega_{1}(\xi)-\omega_{0}(\xi)|_{\Pi_{0}}+|\Omega_{1}(\xi)-\Omega_{0}(\xi)|_{-\delta,\Pi_{0}}
&\lessdot\| X_{R_{0}}\|_{D(s_{0},r_{0},r_{0}),\Pi_{0}};
\end{split}
\]

$(\mathbf{c})$ the perturbation $R_{1}(x,y,z^{*},\bar{z}^{*},\xi) $ satisfies:
\begin{equation}\label{022}
\begin{split}
\|X_{R_{1}}\|_{D(s_{1},r_{1},r_{1}),\Pi_{1}}\lessdot \varepsilon_{1};
\end{split}
\end{equation}

$(\mathbf{d})$ the measure of the $ \Pi_{1}$ of $ \Pi_{0}$ satisfies:
\begin{equation}\label{023}
\begin{split}
\mathbf{Meas}\Pi_{1} \geq (\mathbf{Meas}\Pi_{0})(1-O(\gamma^{\mu}_{0})),
\end{split}
\end{equation}
where
$
\mu =
\begin{cases}
 1\  &for\   d > 1,\\
        \frac{\kappa}{\kappa+1}\ &for\  d=1,
\end{cases}
$ with $ \kappa> 0$.
\end{lemma}

\begin{proof} These results can be seen clearly in \cite{ku1996in}.
\end{proof}

Thus we have
\[
H_{1}=H_{1}(x,y,z^{*},\bar{z}^{*},\xi)=N_{1}(y,z^{*},\bar{z}^{*},\xi)+R_{1}(x,y,z^{*},\bar{z}^{*},\xi),
\]
and
\[
R_{1}(x,y,z^{*},\bar{z}^{*},\xi) =R_{1}^{low}(x,y,z^{*},\bar{z}^{*},\xi)+T_{1}(x,y,z^{*},\bar{z}^{*},\xi)+R_{1}^{high}(x,y,z^{*},\bar{z}^{*},\xi),
\]
where $ R_{1}^{low}(x,y,z^{*},\bar{z}^{*},\xi) $ denotes the truncation of $ R_{1}(x,y,z^{*},\bar{z}^{*},\xi) $ given by

\[
\begin{split}
R_{1}^{low}(x,y,z^{*},\bar{z}^{*},\xi) &=\sum_{|k|\leq K_{1},\alpha\in \mathbb{N}^{n},\beta,\gamma\in \mathbb{N}^{\mathbb{N}},2|\alpha|+|\beta|+|\gamma|\leq2}
\widehat{R}_{1}^{\alpha\beta\gamma}(k,\xi)y^{\alpha}\{z^{*}\}^{\beta}\{\bar{z}^{*}\}^{\gamma}e^{\mathbf{i}\langle k,x\rangle},\\
T_{1}(x,y,z^{*},\bar{z}^{*},\xi)&=\sum_{|k|>K_{1},\alpha\in \mathbb{N}^{n},\beta,\gamma\in \mathbb{N}^{\mathbb{N}},2|\alpha|+|\beta|+|\gamma|\leq2}
\widehat{R}_{1}^{\alpha\beta\gamma}(k,\xi)y^{\alpha}\{z^{*}\}^{\beta}\{\bar{z}^{*}\}^{\gamma}e^{\mathbf{i}\langle k,x\rangle},\\
R_{1}^{high}(x,y,z^{*},\bar{z}^{*},\xi) &=\sum_{\alpha\in \mathbb{N}^{n},\beta,\gamma\in \mathbb{N}^{\mathbb{N}},2|\alpha|+|\beta|+|\gamma|\geq3}
\widehat{R}_{1}^{\alpha\beta\gamma}(k,\xi)y^{\alpha}\{z^{*}\}^{\beta}\{\bar{z}^{*}\}^{\gamma}e^{\mathbf{i}\langle k,x\rangle}.
\end{split}
\]
Then the estimate of $T_{1} $'s norm writes
\begin{eqnarray}\label{024}
\|T_{1}\|_{D(s_{1}-\sigma,r_{1},r_{1}),\Pi_{1}}
 & \leq& \sum_{\substack{\alpha\in \mathbb{N}^{n},\beta,\gamma\in \mathbb{N}^{\mathbb{N}},|k|> K_{1},\\
\|z^{*}\|_{a,p}+\|\bar{z}^{*}\|_{a,p}\leq r_{1}}}|\widehat{R^{\alpha\beta\gamma}_{1}}(k,\xi)|_{\Pi_{1}}
r^{2\alpha}|z^{*}|^{\beta}|\bar{z}^{*}|^{\gamma}e^{|k|(s_{1}-\sigma)}\\
                     \nonumber   & \leq& \|R_{1}\|_{D(s_{1},r_{1},r_{1}),\Pi_{1}}\sum_{l>K_{1}}4^{n}l^{n-1}e^{-l\sigma}\\
\nonumber &\lessdot& K^{n}_{1}e^{-K_{1}\sigma}\|R_{1}\|_{D(s_{1},r_{1},r_{1}),\Pi_{1}} ,
\end{eqnarray}
where $ K_{1}$ determines later.\\

We desire to eliminate the term $R_{1}^{low} $ by the coordinate transformation $ \Phi_{1} $ which is obtained as the time-1-map $ X_{F_{1}}^{t}\mid_{t=1}$ of a Hamiltonian vector field $ X_{F_{1}} $,\ then we have
\begin{eqnarray}
 \nonumber H_{2} &=& H_{1} \circ X^t_{F_{1}}\mid_{t=1} \\
 \nonumber      &=& N_{1}+\{N_{1},F_{1}\}+ \int_0^1 (1-t)\{\{N_{1},F_{1}\},F_{1} \}\circ X^t_{F_{1}}\mathrm{d}t \\
 \nonumber      &&+R_{1}^{low}+\int_0^1\{R_{1}^{low},F_{1}\}\circ X^t_{F_{1}}\mathrm{d}t+(T_{1}+R_{1}^{high})\circ X^t_{F_{1}}\mid_{t=1}\\
 \nonumber      &=& N_{2}+ R_{2}+\{N_{1},F_{1}\}+R_{1}^{low}-\widehat{N}_{1}\\
   \label{025}    &=& N_{2}+ R_{2},
\end{eqnarray}
where
\begin{eqnarray}
\label{026} N_{2}&=&N_{1}+\widehat{N_{1}}\\
\nonumber &=& \sum^{1}_{j=0}\widehat{N^{x}_{j}}(\xi)+\langle \omega_{2},y\rangle+ \sum_{j\in \mathbb{N_{+}}\setminus \mathcal{J}} \Omega^{j}_{2}z_{j}\bar{z}_{j}+
\langle\sum^{1}_{j=0}\widehat{N^{z_{0}}_{j}}(\xi),z_{0}\rangle
+\langle\sum^{1}_{j=0}\widehat{N^{\bar{z}_{0}}_{j}}(\xi),\bar{z}_{0}\rangle
\\
\nonumber &&+\langle\sum^{1}_{j=0}\widehat{N^{z_{0}z_{0}}_{j}}(\xi)z_{0},z_{0}\rangle
+\langle\sum^{1}_{j=0}\widehat{N^{z_{0}\bar{z}_{0}}_{j}}(\xi)z_{0},\bar{z}_{0}\rangle+
\langle\sum^{1}_{j=0}\widehat{N^{\bar{z}_{0}\bar{z}_{0}}_{j}}(\xi)\bar{z}_{0},\bar{z}_{0}\rangle,
\end{eqnarray}
and
\begin{eqnarray}
\label{027}
R_{2}&=&\int_0^1 \{(1-t)\{N_{1},F_{1}\}+tR_{1}^{low},F_{1} \}\circ X^t_{F_{1}}\mathrm{d}t
       +(P_{1}+R_{1}^{high})\circ X^t_{F_{1}}\mid_{t=1}.
\end{eqnarray}
Similarly,\ we can obtain the homological equation
\begin{equation}\label{028}
 \{N_{1},F_{1}\}+ R_{1}^{low} = N_{2}- N_{1}.
\end{equation}

Let $ \partial_{\omega}=\omega\cdot\partial_{x}$.\ Then the homological equation (\ref{028}) decomposes into
\begin{eqnarray}
\label{0281}\partial_{\omega_{1}}F^{z_{0i}z_{0j}}_{1} +\mathbf{i}\sum_{l}\left(\widehat{N^{z_{0i}\bar{z}_{0l}}_{0}}(\xi)F^{z_{0l}z_{0j}}_{1}+ \widehat{F^{z_{0j}{z}_{0l}}_{1}}N^{\bar{z}_{0l}z_{0i}}_{0}(\xi)\right)\\
\nonumber-\mathbf{i}\sum_{l}\left(\widehat{N^{z_{0i}z_{0l}}_{0}}(\xi)F^{z_{0j}\bar{z}_{0l}}_{1}
+\widehat{F^{z_{0j}\bar{z}_{0l}}_{1}}N^{{z}_{0l}z_{0i}}_{0}(\xi)\right)&=& R^{z_{0i}z_{0j}}_{1}, 1\leq i,j\leq b;\\
\label{0282}\partial_{\omega_{1}}F^{z_{0i}\bar{z}_{0j}}_{1}
+2\mathbf{i}\sum_{l}\left(\widehat{N^{\bar{z}_{0i}\bar{z}_{0l}}_{0}}(\xi)F^{z_{0l}z_{0j}}_{1}+ \widehat{F^{z_{0j}{z}_{0l}}_{1}}N^{\bar{z}_{0l}\bar{z}_{0i}}_{0}(\xi)\right)\\
\nonumber
-\mathbf{i}\sum_l
\left(\widehat{N^{{z}_{0i}\bar{z}_{0l}}_{0}}(\xi)F^{z_{0l}\bar{z}_{0j}}_{1} -\widehat{N^{z_{0i}\bar{z}_{0l}}_{0}}(\xi)
F^{{z}_{0l}\bar{z}_{0j}}_{1}\right)\\
\nonumber-2\mathbf{i}\sum_{l}\left(\widehat{N^{z_{0i}{z}_{0l}}_{0}}(\xi)F^{\bar{z}_{0l}\bar{z}_{0j}}_{1}+ \widehat{F^{\bar{z}_{0j}\bar{z}_{0l}}_{1}}N^{{z}_{0l}z_{0i}}_{0}(\xi)\right)&=&R^{z_{0i}\bar{z}_{0j}}_{1}, 1\leq i,j\leq b;\\
\label{0283}\partial_{\omega_{1}}F^{\bar{z}_{0i}\bar{z}_{0j}}_{1}-\mathbf{i}\sum_l
\left(\widehat{N^{{z}_{0i}\bar{z}_{0l}}_{0}}(\xi)F^{\bar{z}_{0l}\bar{z}_{0j}}_{1} +\widehat{F^{\bar{z}_{0l}\bar{z}_{0j}}_{1}}N^{{z}_{0l}\bar{z}_{0i}}_{0}(\xi)\right)\\
\nonumber+\mathbf{i}\sum_{l}\left(\widehat{N^{\bar{z}_{0i}\bar{z}_{0l}}_{0}}(\xi)F^{z_{0j}\bar{z}_{0l}}_{1}
+\widehat{F^{z_{0j}\bar{z}_{0l}}_{1}}N^{\bar{z}_{0l}\bar{z}_{0i}}_{0}(\xi)\right)
&=&R^{\bar{z}_{0i}\bar{z}_{0j}}_{1}, 1\leq i,j\leq b;
\end{eqnarray}
\begin{eqnarray}
\label{0284}\left(\partial_{\omega_{1}}+\mathbf{i}\Omega_{1}^{j}\right)
F^{z_{0i}z_{j}}_{1} +\mathbf{i}\sum_{l}\widehat{N^{z_{0i}\bar{z}_{0l}}_{0}}(\xi)
F^{z_{0l}z_{j}}_{1}\\
\nonumber -2\mathbf{i}\sum_{l}\widehat{N^{z_{0i}z_{0l}}_{0}}(\xi)F^{\bar{z}_{0l}{z}_{j}}_{1}
&=& R^{z_{0i}z_{j}}_{1}, 1\leq i\leq b,j\in \mathbb{N_{+}}\setminus \mathcal{J};\\
\label{0285}\left(\partial_{\omega_{1}}-\mathbf{i}\Omega_{1}^{j}\right)F^{{z}_{0i}\bar{z}_{j}}_{1}
-\mathbf{i}\sum_{l}\widehat{N^{{z}_{0i}\bar{z}_{0l}}_{0}}(\xi)F^{{z}_{0l}\bar{z}_{j}}_{1}\\
\nonumber -2\mathbf{i}\sum_{l}\widehat{N^{z_{0i}z_{0l}}_{0}}(\xi)
F^{\bar{z}_{0l}\bar{z}_{j}}_{1} &=&R^{z_{0i}\bar{z}_{j}}_{1}, 1\leq i\leq b,j\in \mathbb{N_{+}}\setminus \mathcal{J};\\
\label{0286}
\left(\partial_{\omega_{1}}+\mathbf{i}\Omega_{1}^{j}\right)F^{\bar{z}_{0i}{z}_{j}}_{1} +\mathbf{i}\sum_{l}\widehat{N^{z_{0l}\bar{z}_{0i}}_{0}}(\xi)F^{\bar{z}_{0l}{z}_{j}}_{1}\\
\nonumber
+2\mathbf{i}\sum_{l}\widehat{N^{\bar{z}_{0i}\bar{z}_{0l}}_{0}}(\xi)
F^{z_{0l}z_{j}}_{1}&=& R^{\bar{z}_{0i}z_{j}}_{1}, 1\leq i\leq b,j\in \mathbb{N_{+}}\setminus \mathcal{J};\\
\label{0287} \left(\partial_{\omega_{1}}-\mathbf{i}\Omega_{1}^{j}\right)F^{\bar{z}_{0i}\bar{z}_{j}}_{1}
-\mathbf{i}\sum_{l}\widehat{N^{z_{0l}\bar{z}_{0i}}_{0}}(\xi)
F^{\bar{z}_{0l}\bar{z}_{j}}_{1}\\
\nonumber +2\mathbf{i}\sum_{l}\widehat{N^{\bar{z}_{0i}\bar{z}_{0l}}_{0}}(\xi)F^{{z}_{0l}\bar{z}_{j}}_{1}
&=&R^{\bar{z}_{0i}\bar{z}_{j}}_{1}, 1\leq i\leq b,j\in \mathbb{N_{+}}\setminus \mathcal{J};
\end{eqnarray}
\begin{eqnarray}
\label{0311}\partial_{\omega_{1}}F^{z_{0i}}_{1}
&+&\mathbf{i}\sum_{l}\widehat{N^{z_{0i}\bar{z}_{0l}}_{0}}(\xi)F^{z_{0l}}_{1} -2\mathbf{i}\sum_{l}\widehat{N^{z_{0i}z_{0l}}_{0}}(\xi)F^{\bar{z}_{0l}}_{1}\\
\nonumber
&=&R^{{z}_{0i}}_{1}+\mathbf{i}\sum_l\left(\widehat{N^{z_{0l}}_{0}}(\xi)F^{\bar{z}_{0l}z_{0j}}_{1}
-\widehat{N^{\bar{z}_{0l}}_{0}}(\xi)F^{z_{0l}z_{0j}}_{1}\right), 1\leq i\leq b; \\
\label{0312}\partial_{\omega_{1}}F^{\bar{z}_{0i}}_{1} &-&\mathbf{i}\sum_l\widehat{N^{z_{0l}\bar{z}_{0i}}_{0}}(\xi)F^{\bar{z}_{0l}}_{1}
+2\mathbf{i}\sum_l\widehat{N^{\bar{z}_{0i}\bar{z}_{0l}}_{0}}(\xi)F^{z_{0l}}_{1}\\
\nonumber &=&R^{\bar{z}_{0i}}_{1}+\mathbf{i}\sum_l\left(\widehat{N^{z_{0l}}_{0}}(\xi)F^{\bar{z}_{0l}\bar{z}_{0j}}_{1}
-\widehat{N^{\bar{z}_{0l}}_{0}}(\xi)F^{z_{0l}\bar{z}_{0j}}_{1}\right), 1\leq i\leq b;
\end{eqnarray}
\begin{eqnarray}
\label{030}
\partial_{\omega_{1}}F^{z_{i}z_{j}}_{1}+\mathbf{i}(\Omega_{1}^{i}+\Omega_{1}^{j})
F^{z_{i}z_{j}}_{1} &=& R^{z_{i}z_{j}}_{1}, i,j\in \mathbb{N_{+}}\setminus \mathcal{J}; \\
\partial_{\omega_{1}}F^{z_{i}\bar{z}_{j}}_{1}
+\mathbf{i}(\Omega_{1}^{i}-\Omega_{1}^{j})F^{z_{i}\bar{z}_{j}}_{1}
 &=& R^{z_{i}\bar{z}_{j}}_{1} - \delta_{ij}\widehat{N^{z_{i}\bar{z}_{j}}_{1}}(\xi), i,j\in \mathbb{N_{+}}\setminus \mathcal{J};\\
\label{031}
\partial_{\omega_{1}}F^{\bar{z}_{i}\bar{z}_{j}}_{1}
-\mathbf{i}(\Omega_{1}^{i}+\Omega_{1}^{j})F^{\bar{z}_{i}\bar{z}_{j}}_{1}
 &=& R^{\bar{z}_{i}\bar{z}_{j}}_{1}, i,j\in \mathbb{N_{+}}\setminus \mathcal{J};
\end{eqnarray}
and
\begin{eqnarray}
\label{032}(\partial_{\omega_{1}}+\mathbf{i}\Omega_{1}^{j})F^{z_{j}}_{1}
 &=& R^{z_{j}}_{1} +\mathbf{i}\sum_l\left(\widehat{N^{z_{0l}}_{0}}(\xi)F^{\bar{z}_{0l}z_{j}}_{1}
-\widehat{N^{\bar{z}_{0l}}_{0}}(\xi)F^{z_{0l}z_{j}}_{1}\right),j\in \mathbb{N_{+}}\setminus \mathcal{J} ;\\
\label{033}(\partial_{\omega_{1}}-\mathbf{i}\Omega_{1}^{j})F^{\bar{z}_{j}}_{1}
 &=& R^{\bar{z}_{j}}_{1}+\mathbf{i}\sum_l\left(\widehat{N^{z_{0l}}_{0}}(\xi)F^{\bar{z}_{0l}\bar{z}_{j}}_{1}
-\widehat{N^{\bar{z}_{0l}}_{0}}(\xi)F^{z_{0l}\bar{z}_{j}}_{1}\right),j\in \mathbb{N_{+}}\setminus \mathcal{J};\\
\label{034}\partial_{\omega_{1}}F^{y}_{1}
 &=& R^{y}_{1} -\widehat{N^{y}_{1}}(\xi);\\
\label{035}\partial_{\omega_{1}}F^{x}_{1}
 &=& R^{x}_{1}-\widehat{N^{x}_{1}}(\xi)+\mathbf{i}\sum_l\left(\widehat{N^{z_{0l}}_{0}}(\xi)F^{\bar{z}_{0l}}_{1}
-\widehat{N^{\bar{z}_{0}}_{0}}(\xi)F^{z_{0}}_{1}\right).
\end{eqnarray}
In the following we will divide the homological equation (\ref{028}) into four types.\ Before estimating the homological equation (\ref{028}) accurately,\ one firstly introduce Kronecker Product of matrices.
\begin{definition}\label{d1}
Let $ A = (a_{ij})\in C^{m\times n}$ and $ B=(b_{ij})\in C^{p\times q}$.\ Then the following partial matrix
\[
A\otimes B = \begin{pmatrix}
a_{11}B & a_{12}B &\cdots &a_{1n}B\\
a_{21}B & a_{22}B &\cdots &a_{2n}B\\
\vdots & \vdots & &\vdots\\
a_{m1}B & a_{m2}B &\cdots & a_{mn}B
\end{pmatrix} \in  C^{mp\times nq}
\]
is called Kronecker product.
\end{definition}
\begin{definition}\label{d2}
Let $ A = (a_{ij})\in C^{m\times n}$ and note $ a_{i}= (a_{1i},a_{2i},\cdots,a_{mi})^{T}$($ i= 1,2,\cdots,n $).\ Denote
\[
vec(A) = \begin{pmatrix}
a_{1}\\
a_{2}\\
\vdots\\
a_{n}
\end{pmatrix}.
\]
Then $ vec(A)$ is called column straightening of $ A$.
\end{definition}
\begin{lemma}\label{l8}
Let $ A \in C^{m\times n}, B \in C^{n\times p}$ and $ C \in C^{p\times q}$.\ Then
\[
vec(ABC) = (C^{T}\otimes A)vec(B).
\]
\end{lemma}
\begin{lemma}\label{l9}
Let $ A \in C^{m\times m}, B \in C^{n\times n}$ and $ X \in C^{m\times n}$.\ Then\\
$
(1) vec(AX)= (I_{n}\otimes A)vec(X);\\
$
$
(2) vec(XB) = (B^{T}\otimes I_{m})vec(X);\\
$
$
(3) vec(AX +XB) = (I_{n}\otimes A+B^{T}\otimes I_{m})vec(X).
$
\end{lemma}
Now we come back to divide the homological equation (\ref{028}).

$ \mathbf{Type\ (\dag1).} $ The corresponding coefficient equation of $F^{z_{0}{z}_{0}}_{1}, F^{z_{0}\bar{z}_{0}}_{1}$ and $F^{\bar{z}_{0}\bar{z}_{0}}_{1}$.\\
Then they become into
\begin{equation}
\left\{
\begin{aligned}\nonumber
\partial_{\omega_{1}}F^{z_{0}{z}_{0}}_{1}&+ \mathbf{i}\left(\widehat{N^{z_{0}\bar{z}_{0}}_{0}}(\xi)F^{z_{0}{z}_{0}}_{1}
+F^{z_{0}{z}_{0}}_{1}\widehat{N^{z_{0}\bar{z}_{0}}_{0}}(\xi)\right)\\
\nonumber
&-\mathbf{i}\left(\widehat{N^{z_{0}{z}_{0}}_{0}}(\xi)F^{z_{0}\bar{z}_{0}}_{1}
+F^{z_{0}\bar{z}_{0}}_{1}\widehat{N^{z_{0}{z}_{0}}_{0}}(\xi)\right)=R^{z_{0}z_{0}}_{1},\\
\nonumber
\partial_{\omega_{1}}F^{z_{0}\bar{z}_{0}}_{1}&+ 2\mathbf{i}\left(\widehat{N^{\bar{z}_{0}\bar{z}_{0}}_{0}}(\xi)F^{z_{0}{z}_{0}}_{1}
+F^{z_{0}{z}_{0}}_{1}\widehat{N^{\bar{z}_{0}\bar{z}_{0}}_{0}}(\xi)\right)\\
\nonumber&-\mathbf{i}\left(\widehat{N^{{z}_{0}\bar{z}_{0}}_{0}}(\xi)F^{\bar{z}_{0}\bar{z}_{0}}_{1}
-F^{\bar{z}_{0}\bar{z}_{0}}_{1}\widehat{N^{\bar{z}_{0}\bar{z}_{0}}_{0}}(\xi)\right)\\
\nonumber&+2\mathbf{i}\left(\widehat{N^{{z}_{0}{z}_{0}}_{0}}(\xi)F^{z_{0}\bar{z}_{0}}_{1}
+F^{\bar{z}_{0}\bar{z}_{0}}_{1}\widehat{N^{\bar{z}_{0}\bar{z}_{0}}_{0}}(\xi)\right)=R^{{z}_{0}\bar{z}_{0}}_{1},\\
\nonumber\partial_{\omega_{1}}F^{\bar{z}_{0}\bar{z}_{0}}_{1}&- \mathbf{i}\left(\widehat{N^{z_{0}\bar{z}_{0}}_{0}}(\xi)F^{\bar{z}_{0}\bar{z}_{0}}_{1}
+F^{\bar{z}_{0}\bar{z}_{0}}_{1}\widehat{N^{z_{0}\bar{z}_{0}}_{0}}(\xi)\right)\\
\nonumber&+\mathbf{i}\left(\widehat{N^{\bar{z}_{0}\bar{z}_{0}}_{0}}(\xi)F^{z_{0}\bar{z}_{0}}_{1}
+F^{{z}_{0}\bar{z}_{0}}_{1}\widehat{N^{\bar{z}_{0}\bar{z}_{0}}_{0}}(\xi)\right)
=R^{\bar{z}_{0}\bar{z}_{0}}_{1}.
\end{aligned}
\right.
\end{equation}
It follows that
\begin{equation}
\left\{
\begin{aligned}\nonumber
\partial_{\omega_{1}}vec(F^{z_{0}{z}_{0}}_{1})&+ \mathbf{i}(I_{b}\otimes\widehat{N^{z_{0}\bar{z}_{0}}_{0}}(\xi)
+\widehat{N^{z_{0}\bar{z}_{0}}_{0}}(\xi)\otimes I_{b})vec(F^{z_{0}{z}_{0}}_{1})\\
\nonumber
&-\mathbf{i}(I_{b}\otimes\widehat{N^{z_{0}{z}_{0}}_{0}}(\xi)+ \widehat{N^{z_{0}{z}_{0}}_{0}}(\xi)\otimes I_{b})vec(F^{z_{0}\bar{z}_{0}}_{1})
=vec(R^{z_{0}z_{0}}_{1}),\\
\nonumber
\partial_{\omega_{1}}vec(F^{z_{0}\bar{z}_{0}}_{1})&+ 2\mathbf{i}(I_{b}\otimes\widehat{N^{\bar{z}_{0}\bar{z}_{0}}_{0}}(\xi) + \widehat{N^{\bar{z}_{0}\bar{z}_{0}}_{0}}(\xi)\otimes I_{b})vec(F^{z_{0}{z}_{0}}_{1})\\
\nonumber&-\mathbf{i}(I_{b}\otimes\widehat{N^{{z}_{0}\bar{z}_{0}}_{0}}(\xi)-\widehat{N^{z_{0}\bar{z}_{0}}_{0}}(\xi)\otimes I_{b})vec(F^{\bar{z}_{0}\bar{z}_{0}}_{1})\\
\nonumber&+2\mathbf{i}(I_{b}\otimes\widehat{N^{{z}_{0}{z}_{0}}_{0}}(\xi)
+\widehat{N^{\bar{z}_{0}\bar{z}_{0}}_{0}}(\xi)\otimes I_{b})vec(F^{z_{0}\bar{z}_{0}}_{1})
=vec(R^{{z}_{0}\bar{z}_{0}}_{1}),\\
\nonumber\partial_{\omega_{1}}vec(F^{\bar{z}_{0}\bar{z}_{0}}_{1})&- \mathbf{i}(I_{b}\otimes\widehat{N^{z_{0}\bar{z}_{0}}_{0}}(\xi)
+\widehat{N^{z_{0}\bar{z}_{0}}_{0}}(\xi)\otimes I_{b}
)vec(F^{\bar{z}_{0}\bar{z}_{0}}_{1})\\
\nonumber&+\mathbf{i}(I_{b}\otimes\widehat{N^{\bar{z}_{0}\bar{z}_{0}}_{0}}(\xi)
+\widehat{N^{\bar{z}_{0}\bar{z}_{0}}_{0}}(\xi)\otimes I_{b})vec(F^{z_{0}\bar{z}_{0}}_{1})
=vec(R^{\bar{z}_{0}\bar{z}_{0}}_{1}).
\end{aligned}
\right.
\end{equation}
and rewrite the equation
\begin{eqnarray}
\label{036}
\mathcal{A}\begin{pmatrix}
vec(F^{z_{0}{z}_{0}}_{1})\\
vec(F^{z_{0}\bar{z}_{0}}_{1})\\
vec(F^{\bar{z}_{0}\bar{z}_{0}}_{1})
\end{pmatrix}
= \begin{pmatrix}
vec(R^{z_{0}z_{0}}_{1})\\
vec(R^{z_{0}\bar{z}_{0}}_{1})\\
vec(R^{\bar{z}_{0}\bar{z}_{0}}_{1})
\end{pmatrix},
\end{eqnarray}
where the operator
\[
\mathcal{A}=
\begin{pmatrix}
\partial_{\omega_{1}}+ \mathbf{i}(A_{1}+A_{2}) & -\mathbf{i}(A_{1}+A_{2}) & 0  \\
2\mathbf{i}(A_{5}+A_{6}) & \partial_{\omega_{1}}-\mathbf{i}(A_{3}-A_{4}) & 2\mathbf{i}(A_{1}+A_{2})\\
0 & \mathbf{i}(A_{5}+A_{6}) & \partial_{\omega_{1}}-2\mathbf{i}(A_{3}+A_{4})
\end{pmatrix}_{3b^{2}\times3b^{2}}
\]
and the corresponding matrix are
\begin{eqnarray}
\nonumber A_{1} &=& I_{b}\otimes\widehat{N^{z_{0}{z}_{0}}_{0}}(\xi), A_{2}= \widehat{N^{z_{0}{z}_{0}}_{0}}(\xi)\otimes I_{b},\\
\nonumber A_{3} &=& I_{b}\otimes\widehat{N^{z_{0}\bar{z}_{0}}_{0}}(\xi), A_{4}= \widehat{N^{\bar{z}_{0}\bar{z}_{0}}_{0}}(\xi)\otimes I_{b},\\
\nonumber A_{5} &=& I_{b}\otimes\widehat{N^{\bar{z}_{0}\bar{z}_{0}}_{0}}(\xi), A_{6} = \widehat{N^{\bar{z}_{0}{z}_{0}}_{0}}(\xi)\otimes I_{b}.
\end{eqnarray}

$ \mathbf{Type\ (\dag2).} $ The corresponding coefficient equation of $F^{z_{0}{z}_{j}}, F^{z_{0}\bar{z}_{j}},  F^{\bar{z}_{0}{z}_{j}}$ and $F^{\bar{z}_{0}\bar{z}_{j}}$ for any $ j\in \mathbb{N_{+}}\setminus \mathcal{J}$.\\
Then they become into
\begin{equation}
\left\{
\begin{aligned}
\nonumber \left(\partial_{\omega_{1}}+\mathbf{i}\Omega_{1}^{j}\right)
F^{z_{0}z_{j}}_{1} +\mathbf{i}\widehat{N^{z_{0}\bar{z}_{0}}_{0}}(\xi)
F^{z_{0}z_{j}}_{1} -2\mathbf{i}\widehat{N^{z_{0}z_{0}}_{0}}(\xi)F^{\bar{z}_{0}{z}_{j}}_{1}
= R^{z_{0}z_{j}}_{1},\\
\nonumber \left(\partial_{\omega_{1}}-\mathbf{i}\Omega_{1}^{j}\right)F^{{z}_{0}\bar{z}_{j}}_{1}
+\mathbf{i}\widehat{N^{{z}_{0}\bar{z}_{0}}_{0}}(\xi)
F^{{z}_{0}\bar{z}_{j}}_{1}-2\mathbf{i}\widehat{N^{z_{0}z_{0}}_{0}}(\xi)
F^{\bar{z}_{0}\bar{z}_{j}}_{1}=R^{z_{0}\bar{z}_{j}}_{1},\\
\nonumber
\left(\partial_{\omega_{1}}+\mathbf{i}\Omega_{1}^{j}\right)F^{\bar{z}_{0}{z}_{j}}_{1} -\mathbf{i}\widehat{N^{z_{0}\bar{z}_{0}}_{0}}(\xi)F^{\bar{z}_{0}{z}_{j}}_{1}
+2\mathbf{i}\widehat{N^{\bar{z}_{0}\bar{z}_{0}}_{0}}(\xi)
F^{z_{0}z_{j}}_{1}= R^{\bar{z}_{0}z_{j}}_{1},\\
\nonumber \left(\partial_{\omega_{1}}-\mathbf{i}\Omega_{1}^{j}\right)F^{\bar{z}_{0}\bar{z}_{j}}_{1}
-\mathbf{i}\widehat{N^{z_{0}\bar{z}_{0}}_{0}}(\xi)
F^{\bar{z}_{0}\bar{z}_{j}}_{1}
+2\mathbf{i}\widehat{N^{\bar{z}_{0}\bar{z}_{0}}_{0}}(\xi)F^{{z}_{0}\bar{z}_{j}}_{1}
=R^{\bar{z}_{0}\bar{z}_{j}}_{1},
\end{aligned}
\right.
\end{equation}
and rewrite the equation
\begin{eqnarray}
\label{037}
\mathcal{B}
\begin{pmatrix}
F^{z_{0}z_{j}}_{1}\\
F^{z_{0}\bar{z}_{j}}_{1}\\
F^{\bar{z}_{0}z_{j}}_{1}\\
F^{\bar{z}_{0}\bar{z}_{j}}_{1}
\end{pmatrix}
= \begin{pmatrix}
R^{z_{0}z_{j}}_{1}\\
R^{z_{0}\bar{z}_{j}}_{1}\\
R^{\bar{z}_{0}z_{j}}_{1}\\
R^{\bar{z}_{0}\bar{z}_{j}}_{1}
\end{pmatrix},
\end{eqnarray}
where the operator
\begin{eqnarray}
 \nonumber \mathcal{B} &=&
\mathbf{i}\begin{pmatrix}
\widehat{N^{z_{0}\bar{z}_{0}}_{0}}(\xi) & 0 & -2\widehat{N^{z_{0}z_{0}}_{0}}(\xi) & 0  \\
0 & \widehat{N^{z_{0}\bar{z}_{0}}_{0}}(\xi) & 0 & -2\widehat{N^{z_{0}z_{0}}_{0}}(\xi)\\
2\widehat{N^{\bar{z}_{0}\bar{z}_{0}}_{0}}(\xi) & 0
& -\widehat{N^{z_{0}\bar{z}_{0}}_{0}}(\xi)) & 0  \\
0 & 2\widehat{N^{\bar{z}_{0}\bar{z}_{0}}_{0}}(\xi) & 0
& -\widehat{N^{z_{0}\bar{z}_{0}}_{0}}(\xi)
\end{pmatrix}_{4b^{2}\times4b^{2}}  \\
\nonumber       &&+
\begin{pmatrix}
\partial_{\omega_{1}}+\mathbf{i}\Omega^{j}_{1} & 0 & 0 & 0 \\
0 & \partial_{\omega_{1}}-\mathbf{i}\Omega^{j}_{1} & 0 & 0 \\
0 & 0 & \partial_{\omega_{1}}+\mathbf{i}\Omega^{j}_{1} & 0  \\
0 & 0 & 0 & \partial_{\omega_{1}}-\mathbf{i}\Omega^{j}_{1}
\end{pmatrix}_{4b^{2}\times4b^{2}}.
\end{eqnarray}

$ \mathbf{Type\ (\dag3).} $  The corresponding coefficient equation of $F^{{z}_{0}}$ and $F^{\bar{z}_{0}}$.\
Then they become into
\begin{eqnarray}
\label{038}
\mathcal{C}\begin{pmatrix}
F^{z_{0}}_{1}\\
F^{\bar{z}_{0}}_{1}
\end{pmatrix}
=  \begin{pmatrix}
{R}^{z_{0}}_{1}\\
{R}^{\bar{z}_{0}}_{1}
\end{pmatrix}+\mathbf{i}\begin{pmatrix}
{F}^{\bar{z}_{0}z_{0}}_{1}{N}^{{z}_{0}}_{0}-{F}^{z_{0}z_{0}}_{1}{N}^{\bar{z}_{0}}_{0}\\
{F}^{\bar{z}_{0}\bar{z}_{0}}_{1}{N}^{{z}_{0}}_{0}-{F}^{z_{0}\bar{z}_{0}}_{1}{N}^{\bar{z}_{0}}_{0}
\end{pmatrix},
\end{eqnarray}
where the operator
\[
\mathcal{C} =\mathbf{i}
\begin{pmatrix}
\widehat{N^{z_{0}\bar{z}_{0}}_{0}}(\xi) &  -2\widehat{N^{z_{0}z_{0}}_{0}}(\xi)  \\
2\widehat{N^{\bar{z}_{0}\bar{z}_{0}}_{0}}(\xi) &  -\widehat{N^{z_{0}\bar{z}_{0}}_{0}}(\xi)
\end{pmatrix}_{2b^{2}\times2b^{2}}+
\begin{pmatrix}
\partial_{\omega_{1}} & 0  \\
0 & \partial_{\omega_{1}}
\end{pmatrix}_{2b^{2}\times2b^{2}}.
\]
$ \mathbf{Type\ (\dag4).} $ The corresponding coefficient equation of $ F^{z_{i}{z}_{j}}, F^{z_{i}\bar{z}_{j}}, F^{\bar{z}_{i}\bar{z}_{j}}$ for $ i,j \in \mathbb{N_{+}}\setminus \mathcal{J}$ and $ F^{z_{j}}_{1}, F^{\bar{z}_{j}}_{1}$ for $ j\in \mathbb{N_{+}}\setminus \mathcal{J}$ and $ F^{y}_{1}, F^{x}_{1}$.\ Then they become into
\begin{eqnarray}
\label{0291} \partial_{\omega_{1}}F^{z_{i}z_{j}}_{1}+\mathbf{i}(\Omega_{1}^{i}+\Omega_{1}^{j})
F^{z_{i}z_{j}}_{1} &=& R^{z_{i}z_{j}}_{1}, \\
\label{0292}
\partial_{\omega_{1}}F^{z_{i}\bar{z}_{j}}_{1}
+\mathbf{i}(\Omega_{1}^{i}-\Omega_{1}^{j})F^{z_{i}\bar{z}_{j}}_{1}
 &=& R^{z_{i}\bar{z}_{j}}_{1} - \delta_{ij}\widehat{N^{z_{i}\bar{z}_{j}}_{1}}(\xi),\\
\label{0293}
\partial_{\omega_{1}}F^{\bar{z}_{i}\bar{z}_{j}}_{1}
-\mathbf{i}(\Omega_{1}^{i}+\Omega_{1}^{j})F^{\bar{z}_{i}\bar{z}_{j}}_{1}
 &=& R^{\bar{z}_{i}\bar{z}_{j}}_{1},\\
\label{0294}(\partial_{\omega_{1}}+\mathbf{i}\Omega_{1}^{j})F^{z_{j}}_{1}
 &=& R^{z_{j}}_{1} +\mathbf{i}\sum_l\left(\widehat{N^{z_{0l}}_{0}}(\xi)F^{\bar{z}_{0l}z_{j}}_{1}
-\widehat{N^{\bar{z}_{0l}}_{0}}(\xi)F^{z_{0l}z_{j}}_{1}\right) ,\\
\label{0295}(\partial_{\omega_{1}}-\mathbf{i}\Omega_{1}^{j})F^{\bar{z}_{j}}_{1}
 &=& R^{\bar{z}_{j}}_{1}+\mathbf{i}\sum_l\left(\widehat{N^{z_{0l}}_{0}}(\xi)F^{\bar{z}_{0l}\bar{z}_{j}}_{1}
-\widehat{N^{\bar{z}_{0l}}_{0}}(\xi)F^{z_{0l}\bar{z}_{j}}_{1}\right),\\
\label{0296}\partial_{\omega_{1}}F^{y}_{1}
 &=& R^{y}_{1} -\widehat{N^{y}_{1}}(\xi),\\
\label{0297}\partial_{\omega_{1}}F^{x}_{1}
 &=& R^{x}_{1}-\widehat{N^{x}_{1}}(\xi)+\mathbf{i}\sum_l\left(\widehat{N^{z_{0l}}_{0}}(\xi)F^{\bar{z}_{0l}}_{1}
-\widehat{N^{\bar{z}_{0l}}_{0}}(\xi)F^{z_{0l}}_{1}\right).
\end{eqnarray}
\begin{remark}
Compared to the results given in \cite{ku1996in},\ one can easily see that type $(\dag4)$ is the standard equations which are essential in their analysis.\ In this paper,\ we have to prove another three type's equations except type $(\dag4)$.\ Moreover,\ to solve the type $(\dag1)$,\ type $(\dag2)$ and type $(\dag3)$ homological equations,\ we must find the inverses of operators $ \mathcal{A},\mathcal{B} $ and $ \mathcal{C} $ and calculate the measures newly.
\end{remark}
\subsection{ The solvability of homological equation (\ref{028})}\label{non} Different from the homological equation (\ref{013}) who is diagonal and in view of the four types homological equations above,\ the formal non-resonant conditions in \cite{ku1996in} are not enough for us to obtain a solution of homological equation (\ref{028}).\ In order to solve it,\ we have to introduce some new non-resonant conditions (2),(3) and (4) which are corresponding to another three type (\dag2),\ type (\dag3) and type (\dag4) homological equations below.\\
Denote $ |\cdot|_{d}$ the determinant of a matrix,\ the non-resonant conditions are \\
(1) $\mathcal{R}_{kl}(1)=\{\xi\in\Pi_{1}:\mid \langle k,\omega_{1}(\xi)\rangle+\langle l,\Omega_{1}(\xi)\rangle\mid<\frac{\gamma_{1}\langle l\rangle_{d}}{|k|^{\tau}},K_{0}<|k|\leq K_{1},
(k,l)\in \mathcal{Z}\}$,\\
where $ K_{1}=|\log\varepsilon_{1}|/(s_{1}-s_{2})$;\\
(2) $\mathcal{R}_{1k}(1)=\{\xi\in\Pi_{1}:\mid |\mathbf{i}\langle k,\omega_{1}\rangle I_{3b^{2}}- B_{11}(\xi)|_{d}\mid < \frac{\gamma_{11}}{|k|^{\tau_{1}}},0 <|k|\leq K_{1}\}$,\\
where $\tau_{1}= 3b^{2}\tau,\ \gamma_{11}=\gamma_{1},\ L_{1}=3b^{2}\times4b^{2}(diam\Pi_{1})^{n-1}$ and the $ 3b^{2}\times3b^{2}$ order matrix $ B_{11}$'s norm is small enough;\\
(3) $
\mathcal{R}_{3k}(1)=\{\xi\in\Pi_{1}:\mid |\mathbf{i}(\langle k,\omega_{1}\rangle \pm \Omega^{j}_{1}) I_{4b^{2}}+B_{31}(\xi)|_{d}\mid<\frac{\gamma_{31}}{|k|^{\tau_{3}}},|k|\leq K_{1}\}$,\\
where $ \tau_{3}= 4b^{2}\tau,\ \gamma_{31}=\gamma_{1}$,\ $ L_{3}=4b^{2}\times5b^{2}(diam\Pi_{1})^{n-1}$ and the $ 4b^{2}\times4b^{2}$ order matrix $ B_{31}$'s norm is small enough;\\
(4) $\mathcal{R}_{4k}(1)=\{\xi\in\Pi_{1}:\mid |\mathbf{i}\langle k,\omega_{1}\rangle I_{2b^{2}}+B_{41}(\xi)|_{d}\mid<\frac{\gamma_{41}}{|k|^{\tau_{4}}},0 <|k|\leq K_{1}\}$,\\
where $ \tau_{4}= 2b^{2}\tau,\ \gamma_{41}=\gamma_{1}$,\ $ L_{4}=2b^{2}\times3b^{2}(diam\Pi_{1})^{n-1}$ and the $ 2b^{2}\times2b^{2}$ order matrix $ B_{41}$'s norm is small enough.\\
Let
\[
\| Q\|_{\Pi}:= \sup_{\xi\in\Pi}\| Q(\xi)\|,
\]
where $ \|\cdot\| $ is the sup-norm of matrix.\\
We also note
\begin{eqnarray}
\label{039}\Pi_{2}=\Pi_{1}\setminus \bigcup_{K_{0}<|k|\leq K_{1},
(k,l)\in \mathcal{Z}}\mathcal{R}_{kl}(1) \setminus \bigcup_{0<|k|\leq K_{1},i=1,4}\mathcal{R}_{ik}(1)\setminus \bigcup_{|k|\leq K_{1},i=1,4}\mathcal{R}_{3k}(1)
\end{eqnarray}
and
\[
 D^{i}_{2}=D(s_{2}+\frac{i}{4}(s_{1}-s_{2},),\frac{i}{4}\eta_{1}r_{1},\frac{i}{4}\eta_{1}r_{1}),\ 0 < i \leq 4.
\]
\begin{lemma} If the parameter $\xi$ satisfies the non-resonant conditions,\ that is $ \xi\in\Pi_{2}$,\ then the homological equation (\ref{028}) has a solution $ F_{1}(x,y,z^{*},\bar{z}^{*},\xi)$ with the estimate
\begin{equation}\label{040}
\|X_{F_{1}}\|_{D^{3}_{2},\Pi_{2}} \lessdot \gamma^{-6}_{1}K^{(10b^{2}+2)\tau+10b^{2}}_{1}(s_{1}-s_{2})^{-n-1}\varepsilon_{1}.
\end{equation}
\end{lemma}
\begin{proof}
Observing that the above four types whose quadrant terms will decide its 1-th terms,\ we then solve the homological equation in the following order and divide it into six parts more clear for convenience.

$ \mathbf{Part\ 1.} $\ Writing expansions for $ F_{1}$ and $ R_{1} $ and by comparison of coefficients of equation (\ref{036}),\ one finds
\[
A_{11}
\begin{pmatrix}
vec(\widehat{F^{z_{0}z_{0}}_{1}}(k,\xi))\\
vec(\widehat{F^{z_{0}\bar{z}_{0}}_{1}}(k,\xi))\\
vec(\widehat{F^{\bar{z}_{0}\bar{z}_{0}}_{1}}(k,\xi))
\end{pmatrix}
= \begin{pmatrix}
vec(\widehat{R^{z_{0}z_{0}}_{1}}(k,\xi))\\
vec(\widehat{R^{z_{0}\bar{z}_{0}}_{1}}(k,\xi))\\
vec(\widehat{R^{\bar{z}_{0}\bar{z}_{0}}_{1}}(k,\xi))
\end{pmatrix},
\]
where
\[
A_{11}(\omega_{1})=\mathbf{i}\begin{pmatrix}
\langle k,\omega_{1}\rangle+ (A_{1}+A_{2}) & -(A_{1}+A_{2}) & 0  \\
2(A_{5}+A_{6}) & \langle k,\omega_{1}\rangle-(A_{3}-A_{4}) & 2(A_{1}+A_{2})\\
0 & (A_{5}+A_{6}) & \langle k,\omega_{1}\rangle-2(A_{3}+A_{4})
\end{pmatrix}_{3b^{2}\times3b^{2}}.
\]
We note
\[
\langle k,\omega_{1}\rangle I_{3b^{2}}=
\begin{pmatrix}
\langle k,\omega_{1}\rangle & 0 & 0 \\
0 & \langle k,\omega_{1}\rangle & 0 \\
0 & 0 & \langle k,\omega_{1}\rangle
\end{pmatrix}_{3b^{2}\times3b^{2}}
\]
and
\[
B_{11}(\xi)=\mathbf{i}\begin{pmatrix}
(A_{1}+A_{2}) & -(A_{1}+A_{2}) & 0  \\
2(A_{5}+A_{6}) & -(A_{3}-A_{4}) & 2(A_{1}+A_{2})\\
0 & (A_{5}+A_{6}) & -2(A_{3}+A_{4})
\end{pmatrix}_{3b^{2}\times3b^{2}},
\]
then $ A_{11}(\xi) =\mathbf{i}(\langle k,\omega_{1}\rangle I_{3b^{2}}+ B_{11}(\xi))$.\\
Let
\[
\begin{split}
\mathcal{{R}}_{1k}(1)&=\{\xi \in \Pi_{1}:\mid |\mathbf{i}\langle k,\omega_{1}\rangle I_{3b^{2}}- B_{11}(\omega^{-1}_{1})|_{d}\mid < \gamma_{11}/|k|^{\tau_{1}},0<|k|\leq K_{1} \},\\
\mathcal{\tilde{R}}_{1k}(1)&=\{\omega_{1} \in \Pi_{1}:\mid |\mathbf{i}\langle k,\omega_{1}\rangle I_{3b^{2}}- B_{11}(\xi)|_{d}\mid < \gamma_{11}/|k|^{\tau_{1}},0<|k|\leq K_{1} \},
\end{split}
\]
for $ \tau_{1}= 3b^{2}\tau,\ \gamma_{11}=\gamma_{1},\ L_{1}=3b^{2}\times4b^{2}(diam\Pi_{1})^{n-1}$.\\
From the relation (\ref{011}) and (\ref{012}) of
\[
\omega_{1}(\xi)=\omega_{0}(\xi)+\widehat{N^{y}_{0}}(\xi),
\]
we deduce the inequality
\[
|\omega_{1}|_{\Pi_{1}}\leq M+1,
\]
and equality
\[
\omega^{-1}_{1}(\xi)=(\omega_{0}(\xi)+\widehat{N^{y}_{0}}(\xi))^{-1}.
\]
Hence,\ taking account of the assumption $ \mathbf{(A)} $ and the equality
\[
\begin{split}
\partial_{\xi}(\omega^{-1}_{1}(\xi))&=\partial_{\xi}(\omega_{0}(\xi)(1+\omega^{-1}_{0}(\xi)\widehat{N^{y}_{0}}(\xi)))\\
&=\partial_{\xi}\omega_{0}(\xi)(1+\omega^{-1}_{0}(\xi)\widehat{N^{y}_{0}}(\xi))
+\omega_{0}(\xi)\partial_{\xi}(1+\omega^{-1}_{0}(\xi)\widehat{N^{y}_{0}}(\xi))\\
&=\partial_{\xi}\omega_{0}(\xi)(1+\omega^{-1}_{0}(\xi)\widehat{N^{y}_{0}}(\xi))
+\omega_{0}(\xi)(\partial_{\xi}\omega^{-1}_{0}(\xi)\widehat{N^{y}_{0}}(\xi)
+\partial_{\xi}\widehat{N^{y}_{0}}(\xi)\omega^{-1}_{0}(\xi)),
\end{split}
\]
we obtain
\[
|\partial_{\xi}(\omega^{-1}_{1})|_{\omega(\Pi_{1})}\leq M + 3LM\varepsilon_{1},
\]
which means the map $ \xi\mapsto \omega_{1}(\xi)$ is a lipeomorphism between $ \Pi_{1}$ and its image.\\

When $ \omega_{1} \in \Pi_{1} \setminus \cup_{0<|k|\leq K_{1}}\mathcal{\tilde{R}}_{1k}(1)$,\ since we have assumed
\[
\mid |A_{11}(\omega_{1})|_{d} \mid =\mid \mathcal{M}_{k}(\omega_{1}) \mid > \gamma_{11}/|k|^{\tau_{1}},
\]
it implies $ A_{11}^{-1}(\omega_{1})$ exists,\ and making use of the formula
\[
A_{11}^{-1}(\omega_{1})=\frac{adj A_{11}(\omega_{1})}{\mathcal{M}_{k}(\omega_{1})},
\]
where $ adj A$ means the adjoint matrix of $ A$,\ it is easy to see that there exist two constants $c_{1}$,$c_{2}$ such that
\[
\parallel A_{11}(\omega_{1}) \parallel_{\Pi_{2}} \leq c_{1}|k|,
\]
\[
\parallel A_{11}^{-1}(\omega_{1}) \parallel_{\Pi_{2}} \leq c_{2}\frac{|k|^{3b^{2}-1}}{\gamma_{11}/|k|^{\tau_{1}}} \leq c_{2}\gamma^{-1}_{11}|k|^{\tau_{1}+3b^{2}-1}.
\]
By a direct computation, we can prove that
\[
|\frac{d{\mathcal{M}_{k}(\omega_{1})}}{d^{3}\omega^{1}_{1}}|\geq \frac{1}{2}(3b^{2})!|k_{1}|^{3b^{2}},
\]
where $|k_{1}|=\max(|k_{1}|,...|k_{n}|) $,\
and
\[
\begin{split}
\mathbf{Meas} \mathcal{\tilde{R}}_{1k}(1)&\leq L_{1} (\gamma_{1}^{\frac{1}{3b^{2}}}/ |k|^{\tau}),0<|k|\leq K_{1}.
\end{split}
\]
Then,\ for the lipeomorphism of $ \xi\mapsto \omega_{1}(\xi)$ in $ \Pi_{1}$,\ when $ \xi \in \Pi_{1} \setminus \cup_{0<|k|\leq K_{1}}\mathcal{R}_{1k}(1)$,\ we have
\[
\begin{split}
\mathbf{Meas} \mathcal{R}_{1k}(1 )&\leq L_{1} (\gamma_{1}^{\frac{1}{3b^{2}}}/ |k|^{\tau}),0<|k|\leq K_{1}.
\end{split}
\]
And we also have
\[
\begin{split}
\parallel \partial_{\omega_{1}}A_{11}^{-1}(\omega_{1}) \parallel_{\Pi_{2}}
& \leq |\partial_{\omega_{1}}\mathcal{M}^{-1}_{k}(\omega_{1})|\parallel{adj A_{11}(\omega_{1})} \parallel_{\Pi_{2}}+ |\mathcal{M}^{-1}_{k}(\omega_{1})|\parallel \partial_{\omega_{1}}{adj A_{11}(\omega^{1})}\parallel_{\Pi_{2}} \\
& \leq c_{1}\frac{|k|^{2\tau_{1}+6b^{2}-1}}{\gamma^{2}_{11}}+c_{2} \frac{|k|^{\tau_{1}}}{\gamma_{11}}\parallel \partial_{\omega_{1}}(-|k\cdot\omega_{1}|^{2}I_{3b^{2}}+C_{11}(\omega_{1})) \parallel_{\Pi_{2}}\\
& \leq c_{1}\frac{|k|^{2\tau_{1}+6b^{2}-1}}{\gamma^{2}_{11}}
+c_{2}\frac{|k|^{\tau_{1}}}{\gamma_{11}}(|k|^{3b^{2}-1}+|k|^{3b^{2}-2}+...+|k|+\varepsilon^{\frac{1}{12}}_{1})\\
& \leq c_{3}\frac{|k|^{2\tau_{1}+6b^{2}-1}}{\gamma^{2}_{11}},\\
\end{split}
\]
where $ adj A_{11}=-|k\cdot\omega_{1}|^{2}I_{3b^{2}}+C_{11}$ and the constant $ c_{3}$.\\
Thus,\ one obtains
\[
\begin{pmatrix}
|vec(\widehat{F^{z_{0}z_{0}}_{1}}(k,\xi))|\\
|vec(\widehat{F^{z_{0}\bar{z}_{0}}_{1}}(k,\xi))|\\
|vec(\widehat{F^{\bar{z}_{0}\bar{z}_{0}}_{1}}(k,\xi))|
\end{pmatrix}
\leq c_{2}\gamma^{-1}_{11}|k|^{3b^{2}\tau+3b^{2}-1}
\begin{pmatrix}
|vec(\widehat{R^{z_{0}z_{0}}_{1}}(k,\xi))|\\
|vec(\widehat{R^{z_{0}\bar{z}_{0}}_{1}}(k,\xi))|\\
|vec(\widehat{R^{\bar{z}_{0}\bar{z}_{0}}_{1}}(k,\xi))|
\end{pmatrix},
\]
and
\[
\begin{pmatrix}
|vec(\widehat{F^{z_{0}z_{0}}_{1}}(k))|_{\Pi_{2}}\\
|vec(\widehat{F^{z_{0}\bar{z}_{0}}_{1}}(k))|_{\Pi_{2}}\\
|vec(\widehat{F^{\bar{z}_{0}\bar{z}_{0}}_{1}}(k))|_{\Pi_{2}}
\end{pmatrix}
\leq c_{3}\gamma^{-2}_{11}|k|^{6b^{2}\tau+6b^{2}-1}
\begin{pmatrix}
|vec(\widehat{R^{z_{0}z_{0}}_{1}}(k))|_{\Pi_{2}}\\
|vec(\widehat{R^{z_{0}\bar{z}_{0}}_{1}}(k))|_{\Pi_{2}}\\
|vec(\widehat{R^{\bar{z}_{0}\bar{z}_{0}}_{1}}(k))|_{\Pi_{2}}
\end{pmatrix}.
\]

$ \mathbf{Part\ 2.} $\ Considering the quadratic terms of $ z_{i}$ and $\bar{z}_{j}$ for $ i\neq j \in \mathbb{N}_{+}\setminus \mathcal{J}$ and comparing the Fourier coefficients,\ (\ref{0291})-(\ref{0293}) yield
\[
\begin{pmatrix}
\mathbf{i}\widehat{F^{z_{i}z_{j}}_{1}}(k,\xi)\\
\mathbf{i}\widehat{F^{z_{i}\bar{z}_{j}}_{1}}(k,\xi)\\
\mathbf{i}\widehat{F^{\bar{z}_{i}\bar{z}_{j}}_{1}}(k,\xi)
\end{pmatrix}
=
\begin{pmatrix}
\frac{\widehat{R^{z_{i}z_{j}}_{1}}(k,\xi)}{\langle k,\omega_{1}\rangle+\Omega^{i}_{1}+\Omega^{j}_{1}}\\
\frac{\widehat{R^{z_{i}\bar{z}_{j}}_{1}}(k,\xi)}{\langle k,\omega_{1}\rangle+\Omega^{i}_{1}-\Omega^{j}_{1}}\\
\frac{\widehat{R^{\bar{z}_{i}\bar{z}_{j}}_{1}}(k,\xi)}{\langle k,\omega_{1}\rangle-\Omega^{i}_{1}-\Omega^{j}_{1}}
\end{pmatrix}.
\]
Let $ |l|=2 $.\ Then for any $ \xi \in \Pi_{1} \setminus \cup_{K_{0}<|k|\leq K_{1}}\mathcal{R}_{lk}(1)$,\ we have
\[
\begin{pmatrix}
|\widehat{F^{z_{i}z_{j}}_{1}}(k,\xi)|\\
|\widehat{F^{z_{i}\bar{z}_{j}}_{1}}(k,\xi)|\\
|\widehat{F^{\bar{z}_{i}\bar{z}_{j}}_{1}}(k,\xi)|
\end{pmatrix}
\leq \gamma^{-1}_{21}|k|^{\tau}
\begin{pmatrix}
|\widehat{R^{z_{i}z_{j}}_{1}}(k,\xi)|\\
|\widehat{R^{z_{i}\bar{z}_{j}}_{1}}(k,\xi)|\\
|\widehat{R^{\bar{z}_{i}\bar{z}_{j}}_{1}}(k,\xi)|
\end{pmatrix},
\]
and
\[
\begin{pmatrix}
|\widehat{F^{z_{i}z_{j}}_{1}}(k)|_{\Pi_{2}}\\
|\widehat{F^{z_{i}\bar{z}_{j}}_{1}}(k)|_{\Pi_{2}}\\
|\widehat{F^{\bar{z}_{i}\bar{z}_{j}}_{1}}(k)|_{\Pi_{2}}
\end{pmatrix}
\leq \gamma^{-2}_{21}|k|^{2\tau+1}
\begin{pmatrix}
|\widehat{R^{z_{i}z_{j}}_{1}}(k)|_{\Pi_{2}}\\
|\widehat{R^{z_{i}\bar{z}_{j}}_{1}}(k)|_{\Pi_{2}}\\
|\widehat{R^{\bar{z}_{i}\bar{z}_{j}}_{1}}(k)|_{\Pi_{2}}
\end{pmatrix}.
\]

$ \mathbf{Part\ 3.} $\ Considering the quadratic terms of $ z_{0}$ and $\bar{z}_{j}$ for $ j\in \mathbb{N_{+}}\setminus \mathcal{J} $ and
comparing the Fourier coefficients,\ (\ref{037}) yields

\[
A_{31}
\begin{pmatrix}
\widehat{F^{z_{0}z_{j}}_{1}}(k,\xi)\\
\widehat{F^{z_{0}\bar{z}_{j}}_{1}}(k,\xi)\\
\widehat{F^{\bar{z}_{0}z_{j}}_{1}}(k,\xi)\\
\widehat{F^{\bar{z}_{0}\bar{z}_{j}}_{1}}(k,\xi)
\end{pmatrix}
= \begin{pmatrix}
\widehat{R^{z_{0}z_{j}}_{1}}(k,\xi)\\
\widehat{R^{z_{0}\bar{z}_{j}}_{1}}(k,\xi)\\
\widehat{R^{\bar{z}_{0}z_{j}}_{1}}(k,\xi)\\
\widehat{R^{\bar{z}_{0}\bar{z}_{j}}_{1}}(k,\xi)
\end{pmatrix},
\]
where
\begin{align*}
A_{31}&=
\mathbf{i}\begin{pmatrix}
\langle k,\omega_{1}\rangle+\Omega^{j}_{1} & 0 & 0 & 0 \\
0 & \langle k,\omega_{1}\rangle-\Omega^{j}_{1} & 0 & 0 \\
0 & 0 & \langle k,\omega_{1}\rangle+\Omega^{j}_{1} & 0  \\
0 & 0 & 0 & \langle k,\omega_{1}\rangle-\Omega^{j}_{1}
\end{pmatrix}_{4b^{2}\times4b^{2}}  +B_{31}
\end{align*}
with
\[ B_{31}=\begin{pmatrix}
\mathbf{i}\widehat{N^{z_{0}\bar{z}_{0}}_{0}}(\xi) & 0 & -2\widehat{N^{z_{0}z_{0}}_{0}}(\xi) & 0  \\
0 & \widehat{N^{z_{0}\bar{z}_{0}}_{0}}(\xi) & 0 & -2\widehat{N^{z_{0}z_{0}}_{0}}(\xi)\\
2\widehat{N^{\bar{z}_{0}\bar{z}_{0}}_{0}}(\xi) & 0
& -\widehat{N^{z_{0}\bar{z}_{0}}_{0}}(\xi) & 0  \\
0 & 2\widehat{N^{\bar{z}_{0}\bar{z}_{0}}_{0}}(\xi) & 0
& -\widehat{N^{z_{0}\bar{z}_{0}}_{0}})(\xi)
\end{pmatrix}_{4b^{2}\times4b^{2}}.
\]
Note
\[
\begin{split}
\mathcal{R}_{3k}(1)&=\bigcup_{|j|\leq 2K_{1}}\mathcal{R}_{3jk}(1),|k|\leq K_{1},\\
\mathcal{\tilde{R}}_{3k}(1)&=\bigcup_{|j|\leq 2K_{1}}\mathcal{\tilde{R}}_{3jk}(1), |k|\leq K_{1},
\end{split}
\]
where
\[
\begin{split}
\mathcal{R}_{3jk}(1)&=\{ \xi \in \Pi_{1}:\mid |\mathbf{i}(\langle k,\omega_{1}\rangle\pm\Omega^{j}_{1}) I_{4b^{2}}+B_{31}(\omega^{-1}_{1})|_{d}\mid<\gamma_{31}/|k|^{\tau_{3}},|k|\leq K_{1},|j|\leq 2|k| \},\\
\mathcal{\tilde{R}}_{3jk}(1)&=\{\omega_{1} \in \Pi_{1}:\mid |\mathbf{i}(\langle k,\omega_{1}\rangle\pm\Omega^{j}_{1}) I_{4b^{2}}+B_{31}|_{d}\mid<\gamma_{31}/|k|^{\tau_{3}},|k|\leq K_{1},|j|\leq 2|k| \},
\end{split}
\]
for $ \tau_{3}= 4b^{2}\tau,\gamma_{31}=\gamma_{1},L_{3}=4b^{2}\times5b^{2}(diam\Pi_{1})^{n-1}$.\\
When $ \omega_{1} \in \Pi_{1} \setminus \cup_{|k|\leq K_{1}}\mathcal{R}_{3k}(1)$,\ similar with
$ \mathbf{Part\ 1}$,\ for constant $ c_{4}, c_{5} $,\ using the formula
\[
A_{31}^{-1}(\omega_{1})=\frac{adj A_{31}(\omega_{1})}{\mathcal{M}_{k}(\omega_{1})},
\]
we then have
\[
\parallel A_{31}^{-1}(\omega_{1}) \parallel_{\Pi_{2}} \leq c_{5}\frac{|k|^{4b^{2}-1}}{\gamma_{31}/|k|^{\tau_{3}}} \leq c_{5}\gamma^{-1}_{31}|k|^{\tau_{3}+4b^{2}-1}.
\]
Similarly,\ we get the measure estimates
\[
\begin{split}
\mathbf{Meas} \mathcal{\tilde{R}}_{3jk}(1)&\leq L_{3} (\gamma_{1}^{\frac{1}{4b^{2}}}/ |k|^{\tau}),\\
\mathbf{Meas} \mathcal{R}_{3jk}(1)&\leq L_{3} (\gamma_{1}^{\frac{1}{4b^{2}}}/ |k|^{\tau}).
\end{split}
\]
Summing up all $ j$,\ we finally obtain
\[
\begin{split}
\mathbf{Meas} \mathcal{\tilde{R}}_{3k}(1)&\leq 2L_{3}(\gamma_{1}^{\frac{1}{4b^{2}}}/ |k|^{\tau-1}),\\
\mathbf{Meas} \mathcal{R}_{3k}(1)&\leq 2L_{3}(\gamma_{1}^{\frac{1}{4b^{2}}}/ |k|^{\tau-1}),
\end{split}
\]
and
\[
\begin{split}
\parallel \partial_{\omega_{1}}A_{31}^{-1}(\omega_{1}) \parallel_{\Pi_{2}}
& \leq |\partial_{\omega_{1}}\mathcal{M}^{-1}_{k}(\omega_{1})|\parallel{adj A_{31}(\omega_{1})} \parallel_{\Pi_{2}}+ |\mathcal{M}^{-1}_{k}(\omega_{1})|\parallel \partial_{\omega_{1}}{adj A_{31}(\omega_{1})}\parallel_{\Pi_{2}} \\
& \leq c_{4}\frac{|k|^{2\tau_{3}+8b^{2}-1}}{\gamma^{2}_{31}}+c_{5} \frac{|k|^{\tau_{3}}}{\gamma_{31}}\parallel \partial_{\omega_{1}}(-|k\cdot\omega_{1}|^{3}I_{4b^{2}}+C_{31}(\omega_{1})) \parallel_{\Pi_{2}}\\
& \leq c_{4}\frac{|k|^{2\tau_{3}+8b^{2}-1}}{\gamma^{2}_{31}}
+c_{5}\frac{|k|^{\tau_{3}}}{\gamma_{31}}(|k|^{4b^{2}-1}+|k|^{4b^{2}-2}+...+|k|+\varepsilon^{\frac{1}{16}}_{1})\\
& \leq c_{6}\frac{|k|^{2\tau_{3}+8b^{2}-1}}{\gamma^{2}_{31}},\\
\end{split}
\]
where $ adj A_{31}=-|k\cdot\omega_{1}|^{3}I_{4b^{2}}+C_{31}$ and constant $c_{6}$.\\
Consequently,\ for any $ \xi \in \Pi_{1}\setminus\bigcup_{|k|\leq K_{1}}\mathcal{R}_{3k}(1)$,\ we have
\[
\begin{pmatrix}
|\widehat{F^{z_{0}z_{j}}_{1}}(k,\xi)|\\
|\widehat{F^{z_{0}\bar{z}_{j}}_{1}}(k,\xi)|\\
|\widehat{F^{\bar{z}_{0}z_{j}}_{1}}(k,\xi)|\\
|\widehat{F^{\bar{z}_{0}\bar{z}_{j}}_{1}}(k,\xi)|
\end{pmatrix}
\leq c_{5}\gamma^{-1}_{31}|k|^{4b^{2}\tau+4b^{2}-1}
\begin{pmatrix}
|\widehat{R^{z_{0}z_{j}}_{1}}(k,\xi)|\\
|\widehat{R^{z_{0}\bar{z}_{j}}_{1}}(k,\xi)|\\
|\widehat{R^{\bar{z}_{0}z_{j}}_{1}}(k,\xi)|\\
|\widehat{R^{\bar{z}_{0}\bar{z}_{j}}_{1}}(k,\xi)|
\end{pmatrix},
\]
and
\[
\begin{pmatrix}
|\widehat{F^{z_{0}z_{j}}_{1}}(k)|_{\Pi_{2}}\\
|\widehat{F^{z_{0}\bar{z}_{j}}_{1}}(k)|_{\Pi_{2}}\\
|\widehat{F^{\bar{z}_{0}z_{j}}_{1}}(k)|_{\Pi_{2}}\\
|\widehat{F^{\bar{z}_{0}\bar{z}_{j}}_{1}}(k)|_{\Pi_{2}}
\end{pmatrix}
\leq c_{6}\gamma^{-2}_{31}|k|^{8b^{2}\tau+8b^{2}-1}
\begin{pmatrix}
|\widehat{R^{z_{0}z_{j}}_{1}}(k)|_{\Pi_{2}}\\
|\widehat{R^{z_{0}\bar{z}_{j}}_{1}}(k)|_{\Pi_{2}}\\
|\widehat{R^{\bar{z}_{0}z_{j}}_{1}}(k)|_{\Pi_{2}}\\
|\widehat{R^{\bar{z}_{0}\bar{z}_{j}}_{1}}(k)|_{\Pi_{2}}
\end{pmatrix}.
\]

$ \mathbf{Part\ 4.} $\ Considering equation (\ref{038}) and comparing the Fourier coefficients of two side,\ (\ref{038}) reduces to
\[
A_{41}
\begin{pmatrix}
\widehat{F^{z_{0}}_{1}}(k,\xi)\\
\widehat{F^{\bar{z}_{0}}_{1}}(k,\xi)\\
\end{pmatrix}
= \begin{pmatrix}
\widehat{\tilde{R}^{z_{0}}_{1}}(k,\xi) \\
\widehat{\tilde{R}^{\bar{z}_{0}}_{1}}(k,\xi)
\end{pmatrix},
\]
where
\[
A_{41}= \mathbf{i}
\begin{pmatrix}
\langle k,\omega_{1}\rangle & 0  \\
0 & \langle k,\omega_{1}\rangle
\end{pmatrix}_{2b^{2}\times2b^{2}} + B_{41}
\]
with
\[
B_{41}=
\begin{pmatrix}\mathbf{i}
\widehat{N^{z_{0}\bar{z}_{0}}_{0}}(\xi) &  -2\widehat{N^{z_{0}z_{0}}_{0}}(\xi)  \\
2\widehat{N^{\bar{z}_{0}\bar{z}_{0}}_{0}}(\xi) &  -\widehat{N^{z_{0}\bar{z}_{0}}_{0}}(\xi)
\end{pmatrix}_{2b^{2}\times2b^{2}},
\]
and
\[
\begin{pmatrix}
\widehat{\tilde{R}^{z_{0}}_{1}}(k,\xi)\\
\widehat{\tilde{R}^{\bar{z}_{0}}_{1}}(k,\xi)
\end{pmatrix}
= \begin{pmatrix}
\widehat{R^{z_{0}}_{1}}(k,\xi)-\mathbf{i}(\widehat{N^{z_{0}}_{0}}(\xi)\widehat{F^{z_{0}\bar{z}_{0}}_{1}}(k,\xi)
+2\widehat{N^{\bar{z}_{0}}_{0}}(\xi)\widehat{F^{z_{0}z_{0}}_{1}}(k,\xi)) \\
\widehat{R^{\bar{z}_{0}}_{1}}(k,\xi)+\mathbf{i}(\widehat{N^{\bar{z}_{0}}_{0}}(\xi)
\widehat{F^{z_{0}\bar{z}_{0}}_{1}}(k,\xi)
-2\widehat{N^{z_{0}}_{0}}(\xi)\widehat{F^{\bar{z}_{0}\bar{z}_{0}}_{1}}(k,\xi))
\end{pmatrix}.
\]
Let
\[
\begin{split}
\mathcal{R}_{4k}(1)&=\{\xi \in \Pi_{1}:\mid |\mathbf{i}\langle k,\omega_{1}\rangle I_{2b^{2}}+B_{41}(\omega^{-1}_{1})|_{d}\mid<\gamma_{41}/|k|^{\tau_{4}},0<|k|\leq K_{1} \},\\
\mathcal{\tilde{R}}_{4k}(1)&=\{\omega_{1} \in \Pi_{1}:\mid |\mathbf{i}\langle k,\omega_{1}\rangle I_{2b^{2}}+B_{41}|_{d}\mid<\gamma_{41}/|k|^{\tau_{4}},0<|k|\leq K_{1} \},
\end{split}
\]
where $ \tau_{4}= 2b^{2}\tau,\gamma_{42}=\gamma_{1},L_{4}=3b^{2}\times4b^{2}(diam\Pi_{1})^{n-1}$.\\
When $ \omega_{1} \in \Pi_{1} \setminus \bigcup_{0<|k|\leq K_{1}}\mathcal{\tilde{R}}_{4k}(1)$,\ for some constants $c_{7}, c_{8}$,\ it is easy to see that
\[
\parallel A_{41}(\omega_{1}) \parallel_{\Pi_{2}} \leq c_{7}|k|,k\neq 0,
\]

\[
\parallel A_{41}^{-1}(\omega_{1}) \parallel_{\Pi_{2}} \leq c_{8}\frac{|k|^{2b^{2}-1}}{\gamma_{41}/|k|^{\tau_{4}}} \leq c_{8}\gamma^{-1}_{41}|k|^{\tau_{4}+b^{2}-1},
\]
together with
\[
\begin{split}
\mathbf{Meas} \mathcal{\tilde{R}}_{4k}(1)&\leq L_{4} (\gamma_{1}^{\frac{1}{2b^{2}}}/ |k|^{\tau}),\\
\mathbf{Meas} \mathcal{R}_{4k}(1)&\leq L_{4} (\gamma_{1}^{\frac{1}{2b^{2}}}/ |k|^{\tau}).
\end{split}
\]
Moreover,\ one has
\[
\begin{split}
\parallel \partial_{\omega_{1}}A_{41}^{-1}(\omega_{1}) \parallel_{\Pi_{2}}
& \leq |\partial_{\omega_{1}}\mathcal{M}^{-1}_{k}(\omega_{1})|\parallel{adj A_{41}(\omega_{1})} \parallel_{\Pi_{2}}+ |\mathcal{M}^{-1}_{k}(\omega_{1})|\parallel \partial_{\omega_{1}}{adj A_{41}(\omega_{1})}\parallel_{\Pi_{2}} \\
& \leq c_{7}\frac{|k|^{2\tau_{4}+4b^{2}-1}}{\gamma^{2}_{41}}+ c_{8}\frac{|k|^{\tau_{4}}}{\gamma_{41}}\parallel \partial_{\omega_{1}}(-|k\cdot\omega_{1}|E_{2b^{2}}+C_{41}(\omega_{1})) \parallel_{\Pi_{2}}\\
& \leq c_{7}\frac{|k|^{2\tau_{4}+4b^{2}-1}}{\gamma^{2}_{41}}
+c_{8}\frac{|k|^{\tau_{4}}}{\gamma_{41}}(|k|^{2b^{2}-1}+...+|k|+\varepsilon^{\frac{1}{16}}_{1})\\
& \leq c_{9}\frac{|k|^{2\tau_{4}+4b^{2}-1}}{\gamma^{2}_{41}},\\
\end{split}
\]
where $ adj A_{41}=-|k\cdot\omega_{1}|I_{2b^{2}}+C_{41}$ and constant $ c_{9}$.\\
Following from estimates of $Part 1$,\ we also have
\begin{eqnarray}
\nonumber|\widehat{\tilde{R}^{z_{0}}_{1}}(k,\xi)|&\leq& |\widehat{R^{z_{0}}_{1}}(k,\xi)|
+c_{2}\gamma^{-1}_{11}|k|^{3b^{2}\tau+3b^{2}-1}\left(|\widehat{R^{z_{0}\bar{z}_{0}}_{1}}(k,\xi)|
+2|\widehat{R^{z_{0}z_{0}}_{1}}(k,\xi)|\right)\varepsilon_{0}\\
\nonumber &\leq& |\widehat{R^{z_{0}}_{1}}(k,\xi)|
+\gamma^{-1}_{11}|k|^{3b^{2}\tau+3b^{2}-1}\left(|\widehat{R^{z_{0}\bar{z}_{0}}_{1}}(k,\xi)|
+|\widehat{R^{z_{0}z_{0}}_{1}}(k,\xi)|\right),\\
\nonumber|\widehat{\tilde{R}^{\bar{z}_{0}}_{1}}(k,\xi)|&\leq&|\widehat{R^{\bar{z}_{0}}_{1}}(k,\xi)|
+c_{2}\varepsilon_{0}\gamma^{-1}_{11}|k|^{3b^{2}\tau+3b^{2}-1}
\left(|\widehat{R^{z_{0}\bar{z}_{0}}_{1}}(k,\xi)|
+2|\widehat{R^{\bar{z}_{0}\bar{z}_{0}}_{1}}(k,\xi)|\right)\\
\nonumber &\leq& |\widehat{R^{\bar{z}_{0}}_{1}}(k,\xi)|
+\gamma^{-1}_{11}|k|^{3b^{2}\tau+3b^{2}-1}\left(|\widehat{R^{z_{0}\bar{z}_{0}}_{1}}(k,\xi)|
+|\widehat{R^{\bar{z}_{0}\bar{z}_{0}}_{1}}(k,\xi)|\right),
\end{eqnarray}
and
\begin{eqnarray}
\nonumber|\widehat{\tilde{R}^{z_{0}}_{1}}(k)|_{\Pi_{2}} &\leq& |\widehat{R^{z_{0}}_{1}}(k)|_{\Pi_{2}}
+ c_{3}\gamma^{-2}_{11}|k|^{6b^{2}\tau+6b^{2}-1}\left(|\widehat{R^{z_{0}\bar{z}_{0}}_{1}}(k)|_{\Pi_{2}}
+2|\widehat{R^{z_{0}z_{0}}_{1}}(k)|_{\Pi_{2}}\right)\varepsilon_{0}\\
\nonumber &\leq& |\widehat{R^{z_{0}}_{1}}(k)|_{\Pi_{2}}
+\gamma^{-2}_{11}|k|^{6b^{2}\tau+6b^{2}-1}\left(|\widehat{R^{z_{0}\bar{z}_{0}}_{1}}(k)|_{\Pi_{2}}
+|\widehat{R^{z_{0}z_{0}}_{1}}(k)|_{\Pi_{2}}\right),\\
\nonumber|\widehat{\tilde{R}^{\bar{z}_{0}}_{1}}(k)|_{\Pi_{2}}& \leq & |\widehat{R^{\bar{z}_{0}}_{1}}(k)|_{\Pi_{2}}
+ c_{3}\gamma^{-2}_{11}|k|^{6b^{2}\tau+6b^{2}-1}\left(|\widehat{R^{z_{0}\bar{z}_{0}}_{1}}(k)|_{\Pi_{2}}
+2|\widehat{R^{\bar{z}_{0}\bar{z}_{0}}_{1}}(k)|_{\Pi_{2}}\right)\varepsilon_{0} \\
\nonumber& \leq& |\widehat{R^{\bar{z}_{0}}_{1}}(k)|_{\Pi_{2}}
+ \gamma^{-2}_{11}|k|^{6b^{2}\tau+6b^{2}-1}\left(|\widehat{R^{z_{0}\bar{z}_{0}}_{1}}(k)|_{\Pi_{2}}
+|\widehat{R^{\bar{z}_{0}\bar{z}_{0}}_{1}}(k)|_{\Pi_{2}}\right).
\end{eqnarray}
Therefore,\  we get the followings
\begin{eqnarray}
\nonumber|\widehat{F^{z_{0}}_{1}}(k,\xi)|&\leq& \gamma^{-1}_{41}|k|^{2b^{2}\tau+2b^{2}-1}|\widehat{R^{z_{0}}_{1}}(k,\xi)|\\
\nonumber&&+\gamma^{-1}_{41}\gamma^{-1}_{11}|k|^{5b^{2}\tau+5b^{2}-2}\left(|\widehat{R^{z_{0}\bar{z}_{0}}_{1}}(k,\xi)|
+|\widehat{R^{z_{0}z_{0}}_{1}}(k,\xi)|\right),\\
\nonumber|\widehat{F^{\bar{z}_{0}}_{1}}(k,\xi)|
&\leq& c_{7}\gamma^{-1}_{41}|k|^{2b^{2}\tau+2b^{2}-1}|\widehat{R^{\bar{z}_{0}}_{1}}(k,\xi)|\\
\nonumber&&
+\gamma^{-1}_{41}\gamma^{-1}_{11}|k|^{5b^{2}\tau+5b^{2}-2}\left(|\widehat{R^{z_{0}\bar{z}_{0}}_{1}}(k,\xi)|
+|\widehat{R^{\bar{z}_{0}\bar{z}_{0}}_{1}}(k,\xi)|\right),
\end{eqnarray}
and
\begin{eqnarray}
\nonumber|\widehat{F^{z_{0}}_{1}}(k)|_{\Pi_{2}}& \leq& \gamma^{-2}_{41}|k|^{4b^{2}\tau+4b^{2}-1}|\widehat{R^{z_{0}}_{1}}(k)|_{\Pi_{2}}\\
\nonumber &&+ \gamma^{-2}_{11}\gamma^{-2}_{41}|k|^{10b^{2}\tau+10b^{2}-2}
\left(|\widehat{R^{z_{0}\bar{z}_{0}}_{1}}(k)|_{\Pi_{2}}
+|\widehat{R^{z_{0}z_{0}}_{1}}(k)|_{\Pi_{2}}\right),\\
\nonumber|\widehat{F^{\bar{z}_{0}}_{1}}(k)|_{\Pi_{2}}& \leq& \gamma^{-2}_{41}|k|^{4b^{2}\tau+4b^{2}-1}|\widehat{R^{\bar{z}_{0}}_{1}}(k)|_{\Pi_{2}}\\
\nonumber &&+\gamma^{-2}_{11}\gamma^{-2}_{41}|k|^{10b^{2}\tau+10b^{2}-2}
\left(|\widehat{R^{z_{0}\bar{z}_{0}}_{1}}(k)|_{\Pi_{2}}
+|\widehat{R^{\bar{z}_{0}\bar{z}_{0}}_{1}}(k)|_{\Pi_{2}}\right).
\end{eqnarray}

$ \mathbf{Part\ 5.} $\ Considering the 1-th terms of $ F^{z_{j}}_{1}$ and $F_{1}^{\bar{z}_{j}}$ for $ j\in \mathbb{N}_{+}\setminus\mathcal{J} $,\ (\ref{0294})-(\ref{0295}) yield
\begin{eqnarray}\nonumber
(-\partial_{\omega_{1}}-\mathbf{i}\Omega_{1}^{j})F^{z_{j}}_{1}
+\mathbf{i}(\widehat{N^{z_{0}}_{0}}(\xi)F^{\bar{z}_{0}z_{j}}_{1}
-\widehat{N^{\bar{z}_{0}}_{0}}(\xi)F^{z_{0}z_{j}}_{1})
+R^{z_{j}}_{1} &=& 0 ,\\
\nonumber(-\partial_{\omega_{1}}+\mathbf{i}\Omega_{1}^{j})F^{\bar{z}_{j}}_{1}
+\mathbf{i}(\widehat{N^{z_{0}}_{0}}(\xi)F^{\bar{z}_{0}\bar{z}_{j}}_{1}
-\widehat{N^{\bar{z}_{0}}_{0}}(\xi)F^{z_{0}\bar{z}_{j}}_{1})
+R^{\bar{z}_{j}}_{1} &=&0.
\end{eqnarray}
Let $ |l|=1 $.\ Thus for any $ \xi \in \Pi_{1} \setminus\bigcup_{K_{0}<|k|\leq K_{1}} \mathcal{R}_{kl}(1)$,\ we have
\begin{eqnarray}
\nonumber |\widehat{F^{z_{j}}_{1}}(k,\xi)| &\leq& \gamma^{-1}_{51}|k|^{\tau}|\widehat{R^{z_{j}}_{1}}(k,\xi)|\\
\nonumber && +
\gamma^{-1}_{51}\gamma^{-1}_{31}|k|^{(4b^{2}+1)\tau+4b^{2}-1}\left(|\widehat{R^{\bar{z}_{0}z_{j}}_{1}}(k,\xi)|
+|\widehat{R^{z_{0}z_{j}}_{1}}(k,\xi)|\right),\\
\nonumber |\widehat{F^{\bar{z}_{j}}_{1}}(k,\xi)| &\leq & \gamma^{-1}_{51}|k|^{\tau}|\widehat{R^{\bar{z}_{j}}_{1}}(k,\xi)|\\
\nonumber &&+
\gamma^{-1}_{51}\gamma^{-1}_{31}|k|^{(4b^{2}+1)\tau+4b^{2}-1}\left(|\widehat{R^{\bar{z}_{0}\bar{z}_{j}}_{1}}(k,\xi)|
+|\widehat{R^{z_{0}\bar{z}_{j}}_{1}}(k,\xi)|\right),
\end{eqnarray}
and
\begin{eqnarray}
\nonumber|\widehat{F^{z_{j}}_{1}}(k)|_{\Pi_{2}}
& \leq& \gamma^{-2}_{51}|k|^{2\tau+1}|\widehat{R^{z_{j}}_{1}}(k)|_{\Pi_{2}}\\
\nonumber &&+ \gamma^{-2}_{51}\gamma^{-1}_{31}|k|^{(8b^{2}+2)\tau+8b^{2}}
\left(|\widehat{R^{\bar{z}_{0}z_{j}}_{1}}(k)|_{\Pi_{2}}
+|\widehat{R^{z_{0}z_{j}}_{1}}(k)|_{\Pi_{2}}\right),\\
\nonumber |\widehat{F^{\bar{z}_{j}}_{1}}(k)|_{\Pi_{2}}
& \leq& \gamma^{-2}_{51}|k|^{2\tau+1}|\widehat{R^{\bar{z}_{j}}_{1}}(k)|_{\Pi_{2}}\\
\nonumber&&+ \gamma^{-2}_{51}\gamma^{-1}_{31}|k|^{(8b^{2}+2)\tau+8b^{2}}
\left(|\widehat{R^{\bar{z}_{0}\bar{z}_{j}}_{1}}(k)|_{\Pi_{2}}
+|\widehat{R^{z_{0}\bar{z}_{j}}_{1}}(k)|_{\Pi_{2}}\right).
\end{eqnarray}

$ \mathbf{Part\ 6.} $\ Considering the terms of $ F_{1}^{x} $ and $ F_{1}^{y} $,\ (\ref{0296})-(\ref{0297}) reduce to
\begin{eqnarray}\nonumber
\mathbf{i}\widehat{F^{y}_{1}}(k,\xi)
&=& \frac{\widehat{R^{y}_{1}}(k,\xi)}{\langle k,\omega_{1}\rangle},\\
\nonumber \mathbf{i}\widehat{F^{x}_{1}}(k,\xi)
&=& \frac{\widehat{R^{x}_{1}}(k,\xi)+\mathbf{i}\left(\widehat{N^{z_{0}}_{0}}(\xi)\widehat{F^{\bar{z}_{0}}_{1}}(k,\xi)
-\widehat{N^{\bar{z}_{0}}_{0}}(\xi)\widehat{F^{z_{0}}_{1}}(k,\xi)\right)}{\langle k,\omega_{1}\rangle}.
\end{eqnarray}
Let $ |l|=0 $.\ Hence for any $ \xi \in \Pi_{1} \setminus\bigcup_{K_{0}<|k|\leq K_{1}} \mathcal{R}_{lk}(1)$,\ one obtains
\begin{eqnarray}
\nonumber|\widehat{F^{y}_{1}}(k,\xi)|&\leq& \gamma^{-1}_{61}|k|^{\tau}|\widehat{R^{y}_{1}}(k,\xi)|,\\
\nonumber|\widehat{F^{x}_{1}}(k,\xi)| &\leq& \gamma^{-1}_{61}|k|^{\tau}|\widehat{R^{x}_{1}}(k,\xi)|\\
\nonumber&&+\gamma^{-1}_{61}\gamma^{-1}_{41}|k|^{(2b^{2}+1)\tau+2b^{2}-1}
\left(|\widehat{R^{\bar{z}_{0}}_{1}}(k,\xi)|+|\widehat{R^{z_{0}}_{1}}(k,\xi)|\right)\\
\nonumber&&+\gamma^{-1}_{61}\gamma^{-1}_{41}\gamma^{-1}_{11}|k|^{(5b^{2}+1)\tau+5b^{2}-2}
\left(|\widehat{R^{z_{0}\bar{z}_{0}}_{1}}(k,\xi)|
+|\widehat{R^{\bar{z}_{0}\bar{z}_{0}}_{1}}(k,\xi)|
+|\widehat{R^{z_{0}z_{0}}_{1}}(k,\xi)|\right),
\end{eqnarray}
and
\begin{eqnarray}
\nonumber|\widehat{F^{y}_{1}}(k)|_{\Pi_{2}}&\leq& \gamma^{-2}_{61}|k|^{2\tau+1}|\widehat{R^{y}_{1}}(k)|_{\Pi_{2}},\\
|\widehat{F^{x}_{1}}(k)|_{\Pi_{2}}
\nonumber &\leq&
\gamma^{-2}_{61}|k|^{2\tau+1}|\widehat{R^{x}_{1}}(k)|_{\Pi_{2}}\\
\nonumber&&
+\gamma^{-2}_{61}\gamma^{-2}_{42}|k|^{(4b^{2}+2)\tau+4b^{2}}
\left(|\widehat{R^{\bar{z}_{0}}_{1}}(k)|_{\Pi_{2}}+|\widehat{R^{z_{0}}_{1}}(k)|_{\Pi_{2}}\right)\\
\nonumber&& +\gamma^{-2}_{61}\gamma^{-2}_{11}\gamma^{-2}_{41}|k|^{(10b^{2}+2)\tau+10b^{2}-1}
\left(|\widehat{R^{\bar{z}_{0}\bar{z}_{0}}_{1}}(k)|_{\Pi_{2}}
+|\widehat{R^{z_{0}\bar{z}_{0}}_{1}}(k)|_{\Pi_{2}}
+|\widehat{R^{z_{0}z_{0}}_{1}}(k)|_{\Pi_{2}}\right).
\end{eqnarray}

In view of this six parts,\ when $ \xi \in \Pi_{2}$,\ we have
\begin{eqnarray}
&& \frac{1}{r^{2}_{2}}\parallel (F_{1})_{x}\parallel_{D^{3}_{2},\Pi_{2}}\\
\nonumber &\leq& \frac{1}{r^{2}_{2}}\left( |\widehat{F^{x}_{1}}(k)|+|\widehat{F^{y}_{1}}(k)|r^{2}_{2}+ \sum_{1\leq j \leq b}\left(|\widehat{F^{z_{0j}}_{1}}(k)||z_{0j}| +|\widehat{F^{\bar{z}_{0j}}_{1}}(k)||\bar{z}_{0j}|\right)\right.\\
\nonumber&&\left.
+\sum_{j\in \mathbb{N}_{+}\setminus\mathcal{J}}\left(|\widehat{F^{z_{j}}_{1}}(k)||z_{j}| +|\widehat{F^{\bar{z}_{j}}_{1}}(k)||\bar{z}_{j}|\right)\right.\\
\nonumber&&\left. + \sum_{1\leq i,j \leq b}\left(|\widehat{F^{z_{0i}z_{0j}}_{1}}(k)||z_{0i}||z_{0j}|
+|\widehat{F^{z_{0i}\bar{z}_{0j}}_{1}}(k)||z_{0i}||\bar{z}_{0j}|
+|\widehat{F^{\bar{z}_{0i}\bar{z}_{0j}}_{1}}(k)||\bar{z}_{0i}||\bar{z}_{0j}|\right)\right.\\
\nonumber&&\left. +\sum_{i,j\in \mathbb{N}_{+}\setminus\mathcal{J}}\left(|\widehat{F^{z_{i}z_{j}}_{1}}(k)||z_{i}||z_{j}|
+|\widehat{F^{z_{i}\bar{z}_{j}}_{1}}(k)||z_{i}||\bar{z}_{j}|
+|\widehat{F^{\bar{z}_{i}\bar{z}_{j}}_{1}}(k)||\bar{z}_{i}||\bar{z}_{j}|\right) \right)\\
\nonumber &&\times\left(\sum_{|k|\leq K_{1}}|k|e^{|k|(s_{2}+\frac{3}{4}(s_{1}-s_{2}))}\right)\\
 \nonumber&\leq & \left(\gamma^{-2}_{1}K^{8b^{2}\tau+8b^{2}-1}_{1}+\gamma^{-4}_{1}K^{(10b^{2}+2)\tau+10b^{2}-2}_{1}
+\gamma^{-6}_{1}K^{(10b^{2}+2)\tau+10b^{2}-1}_{1}\right)\\
\nonumber&&\times
 \|X_{R_{1}}\|_{D^{2}_{4}\times\Pi_{2}}\sum_{|k|\leq K_{1}}|k|e^{-\frac{1}{4}|k|(s_{1}-s_{2})} \\
\nonumber &\lessdot&  (s_{1}-s_{2})^{-n-1}\left(\gamma^{-2}_{1}K^{(8b^{2}+2)\tau+8b^{2}}_{1}
+\gamma^{-4}_{1}K^{(10b^{2}+2)\tau+10b^{2}-1}_{1}+\gamma^{-6}_{1}
K^{(10b^{2}+2)\tau+10b^{2}}_{1}\right)\varepsilon_{1}\\
\nonumber &\lessdot & \gamma^{-6}_{1}(s_{1}-s_{2})^{-n-1}K^{(10b^{2}+2)\tau+10b^{2}}_{1}\varepsilon_{1}.
\end{eqnarray}
Similarly,\ we have
\[
\|(F_{1})_{y}\|_{D^{3}_{2},\Pi_{2}},
\frac{1}{r_{2}}\|(F_{1})_{z^{*}}\|_{D^{3}_{2},\Pi_{2}},
\frac{1}{r_{2}}\|
(F_{1})_{{\bar{z}^{*}}}\|_{D^{3}_{2},\Pi_{2}}
\lessdot \gamma^{-6}_{1}K^{(10b^{2}+2)\tau+10b^{2}}_{1}(s_{1}-s_{2})^{-n-1}\varepsilon_{1}.
\]
To sum up,\ we obtain
\[
\|X_{F_{1}}\|_{D^{3}_{2},\Pi_{2}}\lessdot \gamma^{-6}_{1}K^{(10b^{2}+2)\tau+10b^{2}}_{1}(s_{1}-s_{2})^{-n-1}\varepsilon_{1}.
\]
\end{proof}

\section{the new hamiltonian}
In view of  (\ref{025}),\ we obtain the new Hamiltonian
\[
H_{2} = N_{2} + R_{2},
\]
where $ N_{2}$ and $ R_{2}$ are given in (\ref{026}) and (\ref{027}) respectively.
\subsection{The new normal form $N_{2}$.} The new normal form is given in (\ref{026}).\\
Note that
\[
\begin{split}
\widehat{J^{x}_{m+1}}(\xi)&=\sum^{m}_{j=0}\widehat{N^{x}_{j}}(\xi),
\widehat{J^{z_{0}}_{m+1}}(\xi)=\sum^{m}_{j=0}\widehat{N^{z_{0}}_{j}}(\xi),\\
\widehat{J^{\bar{z}_{0}}_{m+1}}(\xi)&=\sum^{m}_{j=0}\widehat{N^{\bar{z}_{0}}_{j}}(\xi),
\widehat{J^{z_{0}z_{0}}_{m+1}}(\xi)=\sum^{m}_{j=0}\widehat{N^{z_{0}z_{0}}_{j}}(\xi),\\
\widehat{J^{z_{0}\bar{z}_{0}}_{m+1}}(\xi)&=\sum^{m}_{j=0}\widehat{N^{z_{0}\bar{z}_{0}}_{j}}(\xi),
\widehat{J^{\bar{z}_{0}\bar{z}_{0}}_{m+1}}(\xi)
=\sum^{m}_{j=0}\widehat{N^{\bar{z}_{0}\bar{z}_{0}}_{j}}(\xi).
\end{split}
\]
We then have
\begin{eqnarray}
\nonumber N_{2} &=& \widehat{J^{x}_{2}}(\xi)+\sum_{j=1}^{n}\omega^{j}_{2}(\xi)y_{j}
+\sum_{j\in \mathbb{N}_{+}\setminus\mathcal{J}}\Omega^{j}_{2}(\xi)z_{j}\bar{z}_{j}
 + \langle\widehat{J^{z_{0}}_{2}}(\xi),z_{0}\rangle\\
\nonumber&&  +\langle\widehat{J^{\bar{z}_{0}}_{2}}(\xi),\bar{z}_{0}\rangle
 +\langle\widehat{J^{z_{0}z_{0}}_{2}}(\xi)z_{0},z_{0}\rangle
 +\langle\widehat{J^{z_{0}\bar{z}_{0}}_{2}}(\xi)z_{0},\bar{z}_{0}\rangle
+\langle\widehat{J^{\bar{z}_{0}\bar{z}_{0}}_{2}}(\xi)\bar{z}_{0},\bar{z}_{0}\rangle.
\end{eqnarray}
Based on (\ref{022}),\ one has
\begin{equation}\label{041}
\begin{split}
\|{X_{\widehat{N_{2}}}}\|_{D_{1},\Pi_{1}}\lessdot\varepsilon_{1}.
\end{split}
\end{equation}

\subsection{The new perturbation $R_{2}$.} The new perturbation is given in (\ref{027}),\ i.e.
\[
R_{2}(x,y,z,\bar{z},\xi)=\int_0^1\{ (1-t)\widehat{N_{1}}+tR^{low}_{1},F_{1}\}\circ X^t_{F_{1}}\mathrm{d}t +P_{1}\circ X^t_{F_{1}}\mid_{t=1} +R_{1}^{high}\circ X^t_{F_{1}}\mid_{t=1}.
\]
where $ R_{1}(t)=(1-t)\widehat{N_{1}}+tR^{low}_{1}$,\ then
\[
 X_{R_{2}}= \int_0^1 (\Phi^{t}_{1})^{\ast}X_{\{R_{1}(t),F_{1}\}}\mathrm{d}t +(\Phi^{t}_{1})^{\ast}X_{P_{1}+R_{1}^{high}}.
\]
By the generalized Cauchy's inequality and the estimates (\ref{024}) and (\ref{040}),\ one has
\begin{equation}\label{042}
\begin{split}
\|X_{R_{2}}\|_{D^{1}_{2},\Pi_{2}}
&\leq \|X_{\{R_{1}(t),F_{1}\}}\|_{D^{1}_{2},\Pi_{2}}
+\|X_{P_{1}}\|_{D^{1}_{2},\Pi_{2}}+\|X_{R^{high}_{1}}\|_{D^{1}_{2},\Pi_{2}}\\
&\lessdot \eta^{-2}_{1}\gamma^{-6}_{1}K^{(10b^{2}+2)\tau+10b^{2}}_{1}(s_{1}-s_{2})^{-n-1}\varepsilon^{2}_{1}+
\eta_{1}\varepsilon_{1} + K^{n}_{1}e^{-K_{1}(s_{1}-s_{2})}\varepsilon_{1}.
\end{split}
\end{equation}

\section{Iteration lemma }
To iterate the KAM step infinitely,\ we should choose sequences for the pertinent parameters.\ The guiding principle
 is to choose a geometric sequences for $ s_m$ and the error estimate $ \eta_m, \gamma_m, M_m$.\ We define for all $ m\geq 2 $ the following sequences

\begin{itemize}
 \item[]$s_{m} =s_{1}(1-\sum^{m+1}_{i=1}2^{-i})$,

 \item[]$\gamma_{m}=\frac{\gamma_{1}}{2}(1+2^{-m+1})$,

 \item[]$\varepsilon_{m} = \gamma^{-6}_{m-1}(m-1) ^{64b^{4}}(s_{m-1}-s_{m})^{-n-1}\varepsilon^{\frac{4}{3}}_{m-1}$,\ which dominates the size of the perturbation,

 \item[]$\eta_{m}=\varepsilon^{\frac{1}{3}}_{m}$,

 \item[]$r_{m} =\frac{1}{8}\eta_{m}r_{m-1}$,

\item[]$K_{m}=|\log\varepsilon_{m}|/(s_{m}-s_{m+1})$,

\item[]$\varepsilon^{\frac{1}{6}}_{m} =(K_{m})^{n}e^{-K_{m}(s_{m}-s_{m+1})}$,

\item[]$ D_{m} = D(s_{m},r_{m},r_{m}) $,

\item[]$ D^{i}_{m}= D_{m-1}\left(s_{m}+\frac{i}{4}(s_{m-1}-s_{m}),
     \frac{i}{4}\eta_{m-1}r_{m-1},\frac{i}{4}\eta_{m-1}r_{m-1}\right),0<i\leq4$,
\end{itemize}
and the new non-resonant conditions
\begin{itemize}
\item[(1)]
$ \mathcal{R}_{kl}(m)=\{\xi\in\Pi_{m}:\mid \langle k,\omega_{m}(\xi)\rangle+\langle l,\Omega_{m}(\xi)\rangle\mid<\frac{\gamma_{m}\langle l\rangle_{d}}{|k|^{\tau}},K_{m-1}< |k|\leq K_{m},(k,l)\in\mathcal{Z}\}$,
\item[(2)]
$ \mathcal{R}_{1k}(m) =\{\xi \in \Pi_{m}:\mid |\mathbf{i}\langle k,\omega_{m}(\xi)\rangle I_{3b^{2}}- B_{1m}(\xi)|_{d}\mid < \frac{\gamma_{1m}}{|k|^{\tau_{1}}},0 <|k|\leq K_{m} \}$,\\
where $ \tau_{1}= 3b^{2}\tau, \gamma_{1m}=\gamma_{m}/ m^{18b^{4}},\ L_{1}=3b^{2}\times4b^{2}(diam\Pi_{m})^{n-1}$,
\item[(3)]
$ \mathcal{R}_{3k}(m) =\{\xi \in \Pi_{m}:\mid |\mathbf{i}(\langle k,\omega_{m}(\xi)\rangle\pm \Omega^{j}_{m}) I_{4b^{2}}+B_{3m}(\xi)|_{d}\mid<\frac{\gamma_{3m}}{|k|^{\tau_{3}}}, |k|\leq K_{m} \}$,\\
where  $ \tau_{3}= 4b^{2}\tau, \gamma_{3m}=\gamma_{m}/ m^{32b^{4}},\ L_{3}=4b^{2}\times5b^{2}(diam\Pi_{1})^{n-1}$,
\item[(4)]
$ \mathcal{R}_{4k}(m) =\{\xi \in \Pi_{m}:\mid |\mathbf{i}\langle k,\omega_{m}(\xi)\rangle I_{2b^{2}}+B_{4m}(\xi)|_{d}\mid<\frac{\gamma_{4m}}{|k|^{\tau_{4}}},0 <|k|\leq K_{m} \}$,\\
where $\tau_{4}= 2b^{2}\tau, \gamma_{4m}=\gamma_{m}/ m^{8b^{4}}, L_{4}=2b^{2}\times3b^{2}(diam\Pi_{1})^{n-1}$.
\end{itemize}
We note
\[
\Pi_{m+1}=\Pi_{m}\setminus \bigcup_{K_{m-1}<|k|\leq K_{m},(k,l)\in\mathcal{Z}}\mathcal{R}_{kl}(m) \setminus \bigcup_{0<|k|\leq K_{m},i=1,4}\mathcal{R}_{ik}(m)\setminus \bigcup_{|k|\leq K_{m}}\mathcal{R}_{3k}(m).
\]

\begin{lemma}\label{ml}(Iterative lemma)\ The integrable Hamiltonian with the perturbation $ H_{m} = N_{m}+R_{m}$ is analytic on $ D_{m}\times \Pi_{m}$,\ where
\begin{eqnarray}
\nonumber N_{m} &=& \widehat{J^{x}_{m}}(\xi)+\langle\omega_{m}(\xi),y\rangle
+\langle\Omega_{m}(\xi)z,\bar{z}\rangle
      + \langle\widehat{J^{z_{0}}_{m}}(\xi),z_{0}\rangle\\
\nonumber && +\langle\widehat{J^{\bar{z}_{0}}_{m}}(\xi),\bar{z}_{0}\rangle
 +\langle\widehat{J^{z_{0}z_{0}}_{m}}(\xi)z_{0},z_{0}\rangle
 +\langle\widehat{J^{z_{0}\bar{z}_{0}}_{m}}(\xi)z_{0},\bar{z}_{0}\rangle
+\langle\widehat{J^{\bar{z}_{0}\bar{z}_{0}}_{m}}(\xi)\bar{z}_{0},\bar{z}_{0}\rangle.
\end{eqnarray}
is a normal form with the estimate
\begin{equation}\label{072}
\begin{split}
\|{X_{\widehat{N_{m}}}}\|_{D_{m-1},\Pi_{m-1}}\lessdot\varepsilon_{m-1}.
\end{split}
\end{equation}
and the perturbation $ R_{m}$ satisfying
\numberwithin{equation}{section}
\begin{eqnarray}
\label{073}\|X_{R^{low}_{m}}\|_{D_{m},\Pi_{m}}&\lessdot& \varepsilon_{m},\\
\label{074}\|X_{P_{m}}\|_{D_{m+1},\Pi_{m}}&\lessdot& K^{n}_{m}e^{-K_{m}(s_{m}-s_{m+1})}\varepsilon_{m},\\
\label{075}\|X_{R^{high}_{m}}\|_{D(s_{m},4\eta_{m}r_{m},\eta_{m}r_{m}),\Pi_{m}}&\lessdot& \eta_{m}\varepsilon_{m}.
\end{eqnarray}
Suppose the assumption $ \mathbf{(A)}$ and $ \mathbf{(B)}$ are fulfilled for $ \omega_{m}(\xi)$ and $ \Omega_{m}(\xi)$ with $ m= 1$ and
\[
|\omega_{m}(\xi)-\omega_{1}(\xi)|_{\Pi_{m-1}}+
|\Omega_{m}(\xi)-\Omega_{1}(\xi)|_{-\delta,\Pi_{m-1}}\leq {\sum}^{m-1}_{i=1}\varepsilon_{i}.
\]
Then for each $\xi\in\Pi_{m+1}$,\ there exist real analytic symplectic coordinate transformations $ \Phi_{m+1} : D_{m+1}\rightarrow D_{m}$ satisfying
\begin{eqnarray}
\label{076}\|\Phi_{m+1}-id\|_{D_{m+1}^{2},\Pi_{m+1}}&\lessdot \varepsilon^{\frac{5}{6}}_{m},\\
\label{077}\|D\Phi_{m+1}-Id\|_{D_{m+1}^{1},\Pi_{m+1}}&\lessdot \varepsilon^{\frac{5}{6}}_{m},
\end{eqnarray}
such that for $ H_{m+1} = H_{m}\circ\Phi_{m} = N_{m+1} + R_{m+1}$,\ the same assumptions as above are satisfied with $ m+1$ in place of $ m$,\ that is,\
\begin{eqnarray}\label{077*}
|\omega_{m+1}(\xi)-\omega_{1}(\xi)|_{\Pi_{m}}
+|\Omega_{m+1}(\xi)-\Omega_{1}(\xi)|_{-\delta,\Pi_{m}}\leq {\sum}^{m}_{i=1}\varepsilon_{i},
\end{eqnarray}
and
\begin{equation}\label{078}
\begin{split}
\|X_{R_{m+1}}\|_{D_{m+1}^{1},\Pi_{m+1}}\lessdot \varepsilon_{m+1},
\end{split}
\end{equation}
and
\begin{equation}\label{079}
\begin{split}
\mathbf{Meas}\Pi_{m+1} \geq \mathbf{Meas}\Pi_{m}-\gamma^{\mu}_{1}\cdot\frac{1}{1+K_{m-1}}-\gamma^{\frac{1}{4b^{2}}}_{1}\cdot\frac{1}{m^{2}},
\end{split}
\end{equation}
where $ \mu$ is given in (\ref{023}).
\end{lemma}
\begin{proof}
In the step $ m \rightarrow m+1 $,\ first of all,\ dropping the index $ 1$ of the homological equation (\ref{013}),\ the $ m-th $ homological equation writes
\begin{equation}\label{080}
\{N_{m},F_{m} \} + R_{m}= \widehat{N_{m}},
\end{equation}
and the corresponding six parts are:\\
$ \mathbf{Part\ 1.} $\ Consider the Fourier coefficients of $ F_{m}^{z_{0}z_{0}}, F_{m}^{z_{0}\bar{z}_{0}}$ and $F_{m}^{\bar{z}_{0}\bar{z}_{0}}$,\ which yield
\[
A_{1m}
\begin{pmatrix}
vec(\widehat{F^{z_{0}z_{0}}_{m}}(k,\xi))\\
vec(\widehat{F^{z_{0}\bar{z}_{0}}_{m}}(k,\xi))\\
vec(\widehat{F^{\bar{z}_{0}\bar{z}_{0}}_{m}}(k,\xi))
\end{pmatrix}
= \begin{pmatrix}
vec(\widehat{R^{z_{0}z_{0}}_{m}}(k,\xi))\\
vec(\widehat{R^{z_{0}\bar{z}_{0}}_{m}}(k,\xi))\\
vec(\widehat{R^{\bar{z}_{0}\bar{z}_{0}}_{m}}(k,\xi))
\end{pmatrix},
\]
where $ A_{1m}= \mathbf{i}\langle k,\omega_{m}\rangle I_{3b^{2}} + B_{1m} $ and
\[
B_{1m}(\xi)=\mathbf{i}
\begin{pmatrix}
(A_{1m1}+A_{1m2}) & -(A_{1m1}+A_{1m2}) & 0  \\
2(A_{1m5}+A_{1m6}) & -(A_{1m3}-A_{1m4}) & 2(A_{1m1}+A_{1m2})\\
0 & (A_{1m5}+A_{1m6}) & 2(A_{1m3}+A_{1m4})
\end{pmatrix}_{3b^{2}\times3b^{2}}
\]
with the corresponding matrix are
\begin{eqnarray}
\nonumber A_{1m1} &=& I_{b}\otimes\widehat{J^{z_{0}{z}_{0}}_{m}}(\xi), A_{1m2}= \widehat{J^{z_{0}{z}_{0}}_{m}}(\xi)\otimes I_{b},\\
\nonumber A_{1m3} &=& I_{b}\otimes\widehat{J^{z_{0}\bar{z}_{0}}_{m}}(\xi), A_{1m4}= \widehat{J^{\bar{z}_{0}\bar{z}_{0}}_{m}}(\xi)\otimes I_{b},\\
\nonumber A_{1m5} &=& I_{b}\otimes\widehat{J^{\bar{z}_{0}\bar{z}_{0}}_{m}}(\xi), A_{1m6} = \widehat{J^{\bar{z}_{0}{z}_{0}}_{m}}(\xi)\otimes I_{b}.
\end{eqnarray}
Thus one obtains the estimates
\[
\begin{pmatrix}
|vec(\widehat{F^{z_{0}z_{0}}_{m}}(k,\xi))|\\
|vec(\widehat{F^{z_{0}\bar{z}_{0}}_{m}}(k,\xi))|\\
|vec(\widehat{F^{\bar{z}_{0}\bar{z}_{0}}_{m}}(k,\xi))|
\end{pmatrix}
\lessdot\gamma_{1m}^{-1}|k|^{3b^{2}\tau+3b^{2}-1}
\begin{pmatrix}
|vec(\widehat{R^{z_{0}z_{0}}_{m}}(k,\xi))|\\
|vec(\widehat{R^{z_{0}\bar{z}_{0}}_{m}}(k,\xi))|\\
|vec(\widehat{R^{\bar{z}_{0}\bar{z}_{0}}_{m}}(k,\xi))|
\end{pmatrix},\]
and
\[
\begin{pmatrix}
|vec(\widehat{F^{z_{0}z_{0}}_{m}}(k,\xi))|_{\Pi_{m+1}}\\
|vec(\widehat{F^{z_{0}\bar{z}_{0}}_{m}}(k,\xi))|_{\Pi_{m+1}}\\
|vec(\widehat{F^{\bar{z}_{0}\bar{z}_{0}}_{m}}(k,\xi))|_{\Pi_{m+1}}
\end{pmatrix}
\lessdot \gamma_{1m}^{-2}|k|^{6b^{2}\tau+6b^{2}-1}
\begin{pmatrix}
|vec(\widehat{R^{z_{0}z_{0}}_{m}}(k,\xi))|_{\Pi_{m}}\\
|vec(\widehat{R^{z_{0}\bar{z}_{0}}_{m}}(k,\xi))|_{\Pi_{m}}\\
|vec(\widehat{R^{\bar{z}_{0}\bar{z}_{0}}_{m}}(k,\xi))|_{\Pi_{m}}
\end{pmatrix}.
\]
$ \mathbf{Part\ 2.} $\ Compare the Fourier coefficients $F_{m}^{z_{i}z_{j}}, F_{m}^{z_{i}\bar{z}_{j}}$ and $ F_{m}^{\bar{z}_{i}\bar{z}_{j}}$,\ which yield
\[
A_{2m}
\begin{pmatrix}
\widehat{F^{z_{i}z_{j}}_{m}}(k,\xi)\\
\widehat{F^{z_{i}\bar{z}_{j}}_{m}}(k,\xi)\\
\widehat{F^{\bar{z}_{i}\bar{z}_{j}}_{m}}(k,\xi)
\end{pmatrix}
=
\begin{pmatrix}
\widehat{R^{z_{i}z_{j}}_{m}}(k,\xi)\\
\widehat{R^{z_{i}\bar{z}_{j}}_{m}}(k,\xi)\\
\widehat{R^{\bar{z}_{i}\bar{z}_{j}}_{m}}(k,\xi)
\end{pmatrix},
\]
for any $ i \neq j$ and $ i,j \in \mathbb{N}_{+}\setminus\mathcal{J} $,\ where
\[
A_{2m}=\mathbf{i}
\begin{pmatrix}
\langle k,\omega_{m}\rangle+ \Omega_{m}^{i}+\Omega_{m}^{j} & 0& 0\\
0 & \langle k,\omega_{m}\rangle +\Omega_{m}^{i}-\Omega_{m}^{j} & 0\\
0 & 0 &\langle k,\omega_{m}\rangle-\Omega_{m}^{i}-\Omega_{m}^{j}
\end{pmatrix}.
\]
We then have
\[
\begin{pmatrix}
|\widehat{F^{z_{i}z_{j}}_{m}}(k,\xi)|\\
|\widehat{F^{z_{i}\bar{z}_{j}}_{m}}(k,\xi)|\\
|\widehat{F^{\bar{z}_{i}\bar{z}_{j}}_{m}}(k,\xi)|
\end{pmatrix}
\lessdot \gamma^{-1}_{2m}|k|^{\tau}
\begin{pmatrix}
|\widehat{R^{z_{i}z_{j}}_{m}}(k,\xi)|\\
|\widehat{R^{z_{i}\bar{z}_{j}}_{m}}(k,\xi)|\\
|\widehat{R^{\bar{z}_{i}\bar{z}_{j}}_{m}}(k,\xi)|
\end{pmatrix},
\]
and
\[
\begin{pmatrix}
|\widehat{F^{z_{i}z_{j}}_{m}}(k,\xi)|_{\Pi_{m+1}}\\
|\widehat{F^{z_{i}\bar{z}_{j}}_{m}}(k,\xi)|_{\Pi_{m+1}}\\
|\widehat{F^{\bar{z}_{i}\bar{z}_{j}}_{m}}(k,\xi)|_{\Pi_{m+1}}
\end{pmatrix}
\lessdot \gamma^{-2}_{2m}|k|^{2\tau+1}
\begin{pmatrix}
|\widehat{R^{z_{i}z_{j}}_{m}}(k,\xi)|_{\Pi_{m}}\\
|\widehat{R^{z_{i}\bar{z}_{j}}_{m}}(k,\xi)|_{\Pi_{m}}\\
|\widehat{R^{\bar{z}_{i}\bar{z}_{j}}_{m}}(k,\xi)|_{\Pi_{m}}
\end{pmatrix}.
\]
$ \mathbf{Part\ 3.} $\ Consider the Fourier coefficients of $ F_{m}^{z_{0}z_{j}}, F_{m}^{z_{0}\bar{z}_{j}}, F_{m}^{\bar{z}_{0}z_{j}}$ and $F_{m}^{\bar{z}_{0}\bar{z}_{j}}$ for $ j\in \mathbb{N}_{+}\setminus\mathcal{J}$,\ which reduce to
\[
A_{3m}
\begin{pmatrix}
\widehat{F^{z_{0}z_{j}}_{m}}(k,\xi)\\
\widehat{F^{z_{0}\bar{z}_{j}}_{m}}(k,\xi)\\
\widehat{F^{\bar{z}_{0}z_{j}}_{m}}(k,\xi)\\
\widehat{F^{\bar{z}_{0}\bar{z}_{j}}_{m}}(k,\xi)
\end{pmatrix}
= \begin{pmatrix}
\widehat{R^{z_{0}z_{j}}_{m}}(k,\xi)\\
\widehat{R^{z_{0}\bar{z}_{j}}_{m}}(k,\xi)\\
\widehat{R^{\bar{z}_{0}z_{j}}_{m}}(k,\xi)\\
\widehat{R^{\bar{z}_{0}\bar{z}_{j}}_{m}}(k,\xi)
\end{pmatrix},
\]
where
\begin{align*}
A_{3m} &=
\mathbf{i}\begin{pmatrix}
\langle k,\omega_{m}\rangle+\Omega^{j}_{m} & 0 & 0 & 0 \\
0 & \langle k,\omega_{m}\rangle-\Omega^{j}_{m} & 0 & 0 \\
0 & 0 & \langle k,\omega_{m}\rangle+\Omega^{j}_{m} & 0  \\
0 & 0 & 0 & \langle k,\omega_{m}\rangle-\Omega^{j}_{m}
\end{pmatrix}_{4b^{2}\times 4b^{2}} + B_{3m},
\end{align*}
with
\[
B_{3m}=\mathbf{i}\begin{pmatrix}
\widehat{J^{z_{0}\bar{z}_{0}}_{m}}(\xi) & 0 & -2\widehat{J^{z_{0}z_{0}}_{m}}(\xi) & 0  \\
0 & \widehat{J^{z_{0}\bar{z}_{0}}_{m}}(\xi) & 0 & -2\widehat{J^{z_{0}z_{0}}_{m}}(\xi)\\
2\widehat{J^{\bar{z}_{0}\bar{z}_{0}}_{m}}(\xi) & 0
& -\widehat{J^{z_{0}\bar{z}_{0}}_{m}}(\xi) & 0  \\
0 & 2\widehat{J^{\bar{z}_{0}\bar{z}_{0}}_{m}}(\xi) & 0
& -\widehat{J^{z_{0}\bar{z}_{0}}_{m}})(\xi)
\end{pmatrix}_{4b^{2}\times 4b^{2}}.
\]
Similarly,\ we get the estimates
\[
\begin{pmatrix}
|\widehat{F^{z_{0}z_{j}}_{m}}(k,\xi)|\\
|\widehat{F^{z_{0}\bar{z}_{j}}_{m}}(k,\xi)|\\
|\widehat{F^{\bar{z}_{0}z_{j}}_{m}}(k,\xi)|\\
|\widehat{F^{\bar{z}_{0}\bar{z}_{j}}_{m}}(k,\xi)|
\end{pmatrix}
\lessdot \gamma^{-1}_{3m}|k|^{4b^{2}\tau+4b^{2}-1}
\begin{pmatrix}
|\widehat{R^{z_{0}z_{j}}_{m}}(k,\xi)|\\
|\widehat{R^{z_{0}\bar{z}_{j}}_{m}}(k,\xi)|\\
|\widehat{R^{\bar{z}_{0}z_{j}}_{m}}(k,\xi)|\\
|\widehat{R^{\bar{z}_{0}\bar{z}_{j}}_{m}}(k,\xi)|
\end{pmatrix},
\]
and
\[
\begin{pmatrix}
|\widehat{F^{z_{0}z_{j}}_{m}}(k,\xi)|_{\Pi_{m+1}}\\
|\widehat{F^{z_{0}\bar{z}_{j}}_{m}}(k,\xi)|_{\Pi_{m+1}}\\
|\widehat{F^{\bar{z}_{0}z_{j}}_{m}}(k,\xi)|_{\Pi_{m+1}}\\
|\widehat{F^{\bar{z}_{0}\bar{z}_{j}}_{m}}(k,\xi)|_{\Pi_{m+1}}
\end{pmatrix}
\lessdot \gamma^{-2}_{3m}|k|^{8b^{2}\tau+8b^{2}-1}
\begin{pmatrix}
|\widehat{R^{z_{0}z_{j}}_{m}}(k,\xi)|_{\Pi_{m}}\\
|\widehat{R^{z_{0}\bar{z}_{j}}_{m}}(k,\xi)|_{\Pi_{m}}\\
|\widehat{R^{\bar{z}_{0}z_{j}}_{m}}(k,\xi)|_{\Pi_{m}}\\
|\widehat{R^{\bar{z}_{0}\bar{z}_{j}}_{m}}(k,\xi)|_{\Pi_{m}}
\end{pmatrix}.
\]
$ \mathbf{Part\ 4.} $\ In the following,\ we consider the Fourier coefficients of $ F_{m}^{z_{0}}$ and $F_{m}^{\bar{z}_{0}}$,\ which yield
\[
A_{4m}
\begin{pmatrix}
\widehat{F^{z_{0}}_{m}}(k,\xi)\\
\widehat{F^{\bar{z}_{0}}_{m}}(k,\xi)\\
\end{pmatrix}
= \begin{pmatrix}
\widehat{R^{z_{0}}_{m}}(k,\xi)+\mathbf{i}(
\widehat{F^{z_{0}\bar{z}_{0}}_{m}}(k,\xi)\widehat{J^{z_{0}}_{m}}(\xi)
-2\widehat{F^{z_{0}z_{0}}_{m}}(k,\xi)\widehat{J^{\bar{z}_{0}}_{m}}(\xi)) \\
\widehat{R^{\bar{z}_{0}}_{m}}(k,\xi)+\mathbf{i}(
\widehat{F^{z_{0}\bar{z}_{0}}_{m}}(k,\xi)\widehat{J^{\bar{z}_{0}}_{m}}(\xi)
-2\widehat{F^{\bar{z}_{0}\bar{z}_{0}}_{m}}(k,\xi)\widehat{J^{z_{0}}_{m}}(\xi))
\end{pmatrix},
\]
where
\[
A_{4m}= \mathbf{i}\langle k,\omega_{m}\rangle E_{2b^{2}}+ B_{4m}
\]
with
\[
B_{4m}=\mathbf{i}
\begin{pmatrix}
\widehat{J^{z_{0}\bar{z}_{0}}_{m}}(\xi) &  -2\widehat{J^{z_{0}z_{0}}_{m}}(\xi)  \\
2\widehat{J^{\bar{z}_{0}\bar{z}_{0}}_{m}}(\xi) &  -\widehat{J^{z_{0}\bar{z}_{0}}_{m}}(\xi)
\end{pmatrix}.
\]
Therefore,\  we get the following
\begin{eqnarray}
\nonumber|\widehat{F^{z_{0}}_{m}}(k,\xi)|&\lessdot& \gamma^{-1}_{4m}|k|^{3b^{2}\tau+3b^{2}-1}|\widehat{R^{z_{0}}_{m}}(k,\xi)|\\
\nonumber &&+ \gamma^{-1}_{4m}\gamma^{-1}_{1m}|k|^{5b^{2}\tau+5b^{2}-2}
\left(|\widehat{R^{z_{0}\bar{z}_{0}}_{m}}(k,\xi)|
+|\widehat{R^{z_{0}z_{0}}_{m}}(k,\xi)|\right),\\
\nonumber|\widehat{F^{\bar{z}_{0}}_{m}}(k,\xi)|&\lessdot&  \gamma^{-1}_{4m}|k|^{2b^{2}\tau+2b^{2}-1}|\widehat{R^{\bar{z}_{0}}_{m}}(k,\xi)|\\
\nonumber&&+ \gamma^{-1}_{4m}\gamma^{-1}_{1m}|k|^{5b^{2}\tau+5b^{2}-2}
\left(|\widehat{R^{z_{0}\bar{z}_{0}}_{m}}(k,\xi)|
+|\widehat{R^{\bar{z}_{0}\bar{z}_{0}}_{m}}(k,\xi)|\right),
\end{eqnarray}
and
\begin{eqnarray}
\nonumber|\widehat{F^{z_{0}}_{m}}(k)|_{\Pi_{m+1}}& \lessdot & \gamma^{-2}_{4m}|k|^{4b^{2}\tau+4b^{2}-1}|\widehat{R^{z_{0}}_{m}}(k)|_{\Pi_{m}}\\
\nonumber&&+ \gamma^{-2}_{4m}\gamma^{-2}_{1m}|k|^{6b^{2}\tau+6b^{2}-1}
\left(|\widehat{R^{z_{0}\bar{z}_{0}}_{m}}(k)|_{\Pi_{m}}
+2|\widehat{R^{z_{0}z_{0}}_{m}}(k)|_{\Pi_{m}}\right),\\
\nonumber|\widehat{F^{\bar{z}_{0}}_{m}}(k)|_{\Pi_{m+1}}& \lessdot& \gamma^{-2}_{4m}|k|^{4b^{2}\tau+4b^{2}-1}|\widehat{R^{\bar{z}_{0}}_{m}}(k)|_{\Pi_{m}}\\
\nonumber&&+ \gamma^{-2}_{4m}\gamma^{-2}_{1m}|k|^{10b^{2}\tau+10b^{2}-2}
\left(|\widehat{R^{z_{0}\bar{z}_{0}}_{m}}(k)|_{\Pi_{m}}
+|\widehat{R^{\bar{z}_{0}\bar{z}_{0}}_{m}}(k)|_{\Pi_{m}}\right).
\end{eqnarray}
$ \mathbf{Part\ 5.} $\ Compare the Fourier coefficients of $ F^{z_{j}}_{m} $ and $F^{\bar{z}_{j}}_{m}$ for $ j\geq1 $,\ which reduce to
\[
A_{5m}
\begin{pmatrix}
\widehat{F^{z_{j}}_{m}}(k,\xi)\\
\widehat{F^{\bar{z}_{j}}_{m}}(k,\xi)
\end{pmatrix}
=
\begin{pmatrix}
\widehat{R^{z_{j}}_{m}}(k,\xi)+\mathbf{i}(\widehat{J^{z_{0}}_{m}}(\xi)
\widehat{F}^{\bar{z}_{0}z_{j}}_{m}(k,\xi)
-\widehat{J^{\bar{z}_{0}}_{m}}(\xi)\widehat{F}^{z_{0}z_{j}}_{m}(k,\xi))\\
\widehat{R^{\bar{z}_{j}}_{m}}(k,\xi)+\mathbf{i}(\widehat{J^{z_{0}}_{m}}(\xi)
\widehat{F}^{\bar{z}_{0}\bar{z}_{j}}_{m}(k,\xi)
-\widehat{J^{\bar{z}_{0}}_{m}}(\xi)\widehat{F}^{z_{0}\bar{z}_{j}}_{m}(k,\xi))
\end{pmatrix},
\]
where
\[
A_{5m}=
\mathbf{i}\begin{pmatrix}
\langle k,\omega_{m}\rangle +\Omega^{j}_{m} & 0\\
0 & \langle k,\omega_{m}\rangle -\Omega^{j}_{m}
\end{pmatrix}.
\]
Thus we have the following estimates
\begin{eqnarray}
\nonumber|\widehat{F^{z_{j}}_{m}}(k,\xi)|&\lessdot& \gamma^{-1}_{5m}|k|^{\tau}|\widehat{R^{z_{j}}_{m}}(k,\xi)|\\
\nonumber&&+
 \gamma^{-1}_{5m}\gamma^{-1}_{3m}|k|^{(4b^{2}+1)\tau+4b^{2}-1}
\left(|\widehat{R^{\bar{z}_{0}z_{j}}_{m}}(k,\xi)|
+|\widehat{R^{z_{0}z_{j}}_{m}}(k,\xi)|\right),\\
\nonumber|\widehat{F^{\bar{z}_{j}}_{m}}(k,\xi)|&\lessdot& \gamma^{-1}_{5m}|k|^{\tau}|\widehat{R^{\bar{z}_{j}}_{m}}(k,\xi)|\\
\nonumber &&+
 \gamma^{-1}_{5m}\gamma^{-1}_{3m}|k|^{(4b^{2}+1)\tau+4b^{2}-1}
\left(|\widehat{R^{\bar{z}_{0}\bar{z}_{j}}_{m}}(k,\xi)|
+|\widehat{R^{z_{0}\bar{z}_{j}}_{m}}(k,\xi)|\right),
\end{eqnarray}
and
\begin{eqnarray}
\nonumber|\widehat{F^{z_{j}}_{m}}(k)|_{\Pi_{m+1}}
& \lessdot& \gamma^{-2}_{5m}|k|^{2\tau+1}|\widehat{R^{z_{j}}_{m}}(k)|_{\Pi_{m}}\\
\nonumber&&+  \gamma^{-2}_{5m}\gamma^{-2}_{3m}|k|^{(8b^{2}+2)\tau+8b^{2}}
\left(|\widehat{R^{\bar{z}_{0}z_{j}}_{m}}(k)|_{\Pi_{m}}
+|\widehat{R^{z_{0}z_{j}}_{m}}(k)|_{\Pi_{m}}\right),\\
\nonumber|\widehat{F^{\bar{z}_{j}}_{m}}(k)|_{\Pi_{m+1}}& \lessdot& \gamma^{-2}_{5m}|k|^{2b^{2}\tau+1}|\widehat{R^{\bar{z}_{j}}_{m}}(k)|_{\Pi_{m+1}}\\
\nonumber&&+ \gamma^{-2}_{5m}\gamma^{-2}_{3m}|k|^{(8b^{2}+2)\tau+8b^{2}}
\left(|\widehat{R^{\bar{z}_{0}\bar{z}_{j}}_{m}}(k)|_{\Pi_{m}}
+|\widehat{R^{z_{0}\bar{z}_{j}}_{m}}(k)|_{\Pi_{m}}\right).
\end{eqnarray}
$ \mathbf{Part\ 6.} $\ Consider the Fourier coefficients of $ F^{x}_{m} $ and $ F^{y}_{m}$,\  which yield
\[
A_{6m}
\begin{pmatrix}
\widehat{F^{y}_{m}}(k,\xi)\\
\widehat{F^{x}_{m}}(k,\xi)
\end{pmatrix}
=
\begin{pmatrix}
\widehat{R^{y}_{m}}(k,\xi)\\
\widehat{R^{x}_{m}}(k,\xi)+\mathbf{i}(\widehat{J^{z_{0}}_{m}}(\xi)\widehat{F}^{\bar{z}_{0}}_{m}(k,\xi)
-\widehat{J^{\bar{z}_{0}}_{m}}(\xi)\widehat{F}^{z_{0}}_{m}(k,\xi))
\end{pmatrix},
\]
where
\[
A_{6m}=\mathbf{i}
\begin{pmatrix}
\langle k,\omega_{m}\rangle & 0 \\
0 & \langle k,\omega_{m}\rangle
\end{pmatrix}
\]
for $ k \neq 0$.\\
We thus obtain that
\begin{eqnarray}
\nonumber|\widehat{F^{y}_{m}}(k,\xi)|&\lessdot& \gamma^{-1}_{61}|k|^{\tau}|\widehat{R^{y}_{m}}(k,\xi)|,\\
\nonumber|\widehat{F^{x}_{m}}(k,\xi)|
&\lessdot& \gamma^{-1}_{6m}|k|^{\tau}|\widehat{R^{x}_{m}}(k,\xi)|\\
\nonumber&&
+\gamma^{-1}_{6m}\gamma^{-1}_{4m}|k|^{(2b^{2}+1)\tau+2b^{2}-1}
\left(|\widehat{R^{\bar{z}_{0}}_{m}}(k,\xi)|+|\widehat{R^{z_{0}}_{m}}(k,\xi)|\right)\\
\nonumber&&+ \gamma^{-1}_{6m}\gamma^{-1}_{4m}\gamma^{-1}_{1m}|k|^{(5b^{2}+1)\tau+5b^{2}-2}
\left(|\widehat{R^{z_{0}\bar{z}_{0}}_{m}}(k,\xi)|
+|\widehat{R^{\bar{z}_{0}\bar{z}_{0}}_{m}}(k,\xi)|
+|\widehat{R^{z_{0}z_{0}}_{m}}(k,\xi)|\right),
\end{eqnarray}
and
\begin{eqnarray}
\nonumber|\widehat{F^{y}_{m}}(k)|_{\Pi_{m+1}}&\lessdot& \gamma^{-2}_{6m}|k|^{2\tau+1}|\widehat{R^{y}_{m}}(k)|_{\Pi_{m}},\\
\nonumber|\widehat{F^{x}_{m}}(k)|_{\Pi_{m+1}}&\lessdot& \gamma^{-2}_{6m}|k|^{2\tau+1}|\widehat{R^{x}_{m}}(k)|_{\Pi_{m}}\\
\nonumber&& + \gamma^{-2}_{6m}\gamma^{-2}_{4m}|k|^{(4b^{2}+2)\tau+4b^{2}}
\left(|\widehat{R^{\bar{z}_{0}}_{m}}(k)|_{\Pi_{m}}+|\widehat{R^{z_{0}}_{m}}(k)|_{\Pi_{m}}\right) \\
\nonumber&& + \gamma^{-2}_{6m}\gamma^{-2}_{4m}\gamma^{-2}_{1m}|k|^{(10b^{2}+2)\tau+10b^{2}-1}
\left(|\widehat{R^{\bar{z}_{0}\bar{z}_{0}}_{m}}(k)|_{\Pi_{m}}
+|\widehat{R^{z_{0}\bar{z}_{0}}_{m}}(k)|_{\Pi_{m}}
+|\widehat{R^{z_{0}z_{0}}_{m}}(k)|_{\Pi_{m}}\right).
\end{eqnarray}
In view of the six parts,\ one has
\begin{eqnarray}
&& \frac{1}{r^{2}_{m+1}}\parallel (F_{m})_{x}\parallel_{D^{3}_{m+1},\Pi_{m+1}}\\
\nonumber &\leq& \frac{1}{r^{2}_{m+1}}\left( |\widehat{F^{x}_{m}}(k)|+|\widehat{F^{y}_{m}}(k)|r^{2}_{m+1}+ \sum_{1\leq j \leq b}\left(|\widehat{F^{z_{0j}}_{m}}(k)||z_{0j}| +|\widehat{F^{\bar{z}_{0j}}_{m}}(k)||\bar{z}_{0j}|\right)\right.\\
\nonumber&&\left.
+\sum_{j\in \mathbb{N}_{+}\setminus\mathcal{J}}\left(|\widehat{F^{z_{j}}_{m}}(k)||z_{j}| +|\widehat{F^{\bar{z}_{j}}_{m}}(k)||\bar{z}_{j}|\right)\right.\\
\nonumber&&\left. + \sum_{1\leq i,j \leq b}\left(|\widehat{F^{z_{0i}z_{0j}}_{m}}(k)||z_{0i}||z_{0j}|
+|\widehat{F^{z_{0i}\bar{z}_{0j}}_{m}}(k)||z_{0i}||\bar{z}_{0j}|
+|\widehat{F^{\bar{z}_{0i}\bar{z}_{0j}}_{m}}(k)||\bar{z}_{0i}||\bar{z}_{0j}|\right)\right.\\
\nonumber&&\left. +\sum_{i,j\in \mathbb{N}_{+}\setminus\mathcal{J}}\left(|\widehat{F^{z_{i}z_{j}}_{m}}(k)||z_{i}||z_{j}|
+|\widehat{F^{z_{i}\bar{z}_{j}}_{m}}(k)||z_{i}||\bar{z}_{j}|
+|\widehat{F^{\bar{z}_{i}\bar{z}_{j}}_{m}}(k)||\bar{z}_{i}||\bar{z}_{j}|\right) \right)\\
\nonumber &&\times\left(\sum_{|k|\leq K_{m}}|k|e^{|k|(s_{m}+\frac{3}{4}(s_{m}-s_{m+1}))}\right)\\
\nonumber &\lessdot &  (s_{m}-s_{m+1})^{-n-1}\left(\gamma^{-2}_{m}m^{64b^{2}}K^{8b^{2}\tau+8b^{2}}_{m}
+\gamma^{-4}_{m}m^{128b^{2}}K^{10b^{2}\tau+10b^{2}-1}_{m}\right.\\
\nonumber&&+\left. \gamma^{-6}_{m}m^{192b^{2}}K^{(10b^{2}+2)\tau+10b^{2}}_{m}
\right)\varepsilon_{m}\\
\nonumber&\lessdot &  \gamma^{-6}_{m}m^{192b^{2}}(s_{m}-s_{m+1})^{-n-1}K^{(10b^{2}+2)\tau+10b^{2}}_{m}\varepsilon_{m}\\
\nonumber&\lessdot & \varepsilon^{\frac{5}{6}}_{m}.
\end{eqnarray}
Similarly,\ we have
\[
\|(F_{m})_{y}\|_{D_{m+1}^{3},\Pi_{m+1}},
\frac{1}{r_{m+1}}\|(F_{m})_{z^{*}}\|_{D_{m+1}^{3},\Pi_{m+1}},
\frac{1}{r_{m+1}}\|
(F_{m})_{_{\bar{z}^{*}}}\|_{D_{m+1}^{3},\Pi_{m+1}}
\lessdot \varepsilon^{\frac{5}{6}}_{m}.
\]
To sum up,\ one obtain
\begin{eqnarray}\label{060}
\|X_{F_{m}}\|_{D_{m+1}^{3},\Pi_{m+1}}\lessdot \varepsilon^{\frac{5}{6}}_{m}.
\end{eqnarray}
Thus,\ (\ref{076}) and (\ref{077}) are obvious.\\

Next,\ we will show that the new non-resonant conditions preserve under small perturbation $ \widehat{N_{m}}$.\ Since
$\omega_{m+1}^{j}(\xi) = \omega_{m}^{j}(\xi)+\widehat{N^{y_{j}}_{m}}(\xi)$
and $\Omega_{m+1}^{j}(\xi)=\Omega_{m}^{j}(\xi)+\widehat{N^{z_{j}\bar{z}_{j}}_{m}}(\xi)$,\ one has
\begin{eqnarray}
\label{081}&&| \langle k,\omega_{m+1}(\xi)\rangle+\langle l,\Omega_{m+1}(\xi)\rangle|\\
\nonumber&\geq& \mid \langle k,\omega_{m}(\xi)\rangle+\langle l,\Omega_{m}(\xi)\rangle|
+ |\langle k,\widehat{N^{y}_{m}}(\xi)\rangle+\langle l,\widehat{N^{z\bar{z}}_{m}}(\xi)\rangle|\\
\nonumber&\geq& \frac{\gamma_{m}\langle l\rangle_{d}}{|k|^{\tau}}-(|k||\widehat{N^{y}_{m}}(\xi)|
+|l|_{\delta}|\widehat{N^{z\bar{z}}_{m}}(\xi)|_{-\delta})\\
\nonumber&\geq& \frac{\gamma_{m}\langle l\rangle_{d}}{|k|^{\tau}}-2|k|\varepsilon_{m}\langle l\rangle_{d}\\
\nonumber&\geq& \frac{\gamma_{m+1}\langle l\rangle_{d}}{|k|^{\tau}},
\end{eqnarray}
for $ \xi\in\Pi_{m}$ and $ K_{m-1}<|k|\leq K_{m}$.\\
On the other hand,\ another non-resonant condition becomes
\begin{eqnarray}
\nonumber &&\mid|\mathbf{i}\langle k,\omega_{m+1}\rangle I_{3b^{2}}- B_{1{m+1}}|_{d}\mid\\
\nonumber &\geq& \mid |\mathbf{i}\langle k,\omega_{m}\rangle I_{3b^{2}}- B_{1m}|_{d}\mid\cdot
\mid |I_{3b^{2}}+(\mathbf{i}\langle k,\omega_{m}\rangle I_{3b^{2}}- B_{1m})^{-1}(\mathbf{i}\langle k,\widehat{N^{y}_{m}}\rangle I_{3b^{2}}- \widetilde{B_{1m}})|_{d}\mid\\
\nonumber&\geq& \frac{\gamma_{m}}{|k|^{\tau_{1}}}\cdot\mid|I_{3b^{2}}+(\mathbf{i}\langle k,\omega_{m}\rangle I_{3b^{2}}- B_{1m})^{-1}(\mathbf{i}\langle k,\widehat{N^{y}_{m}}\rangle I_{3b^{2}}- \widetilde{B_{1m}})|_{d}\mid
\end{eqnarray}
where $ \widetilde{B_{1m}}= B_{1{m+1}}- B_{1m}$.\\
Recalling that
\[
\mathbf{i}\langle k,\widehat{N^{y}_{m}}\rangle I_{3b^{2}}- \widetilde{B_{1m}}=\mathbf{i}\begin{pmatrix}
\langle k,\widehat{N^{y}_{m}}\rangle+2\widehat{N^{z_{0}\bar{z}_{0}}_{m}}(\xi) & -2\widehat{N^{z_{0}z_{0}}_{m}}(\xi) & 0  \\
4\widehat{N^{\bar{z}_{0}\bar{z}_{0}}_{m}}(\xi) & \langle k,\widehat{N^{y}_{m}}\rangle & -4\widehat{N^{z_{0}z_{0}}_{m}}(\xi)\\
0 & 2\widehat{N^{\bar{z}_{0}\bar{z}_{0}}_{m}}(\xi) & \langle k,\widehat{N^{y}_{m}}\rangle-2\widehat{N^{z_{0}\bar{z}_{0}}_{m}}(\xi)
\end{pmatrix}_{3b^{2}\times3b^{2}}
\]
for $ 0< |k|\leq K_{m}$,\ one has
\begin{eqnarray}
\nonumber &&\mid|\mathbf{i}\langle k,\omega_{m+1}\rangle I_{3b^{2}}- B_{1{m+1}}|_{d}\mid\\
\nonumber&\geq& \frac{\gamma_{1m}}{|k|^{\tau_{1}}}\cdot\mid |I_{3b^{2}}+(\mathbf{i}\langle k,\omega_{m}\rangle I_{3b^{2}}- B_{1m})^{-1}(\mathbf{i}\langle k,\widehat{N^{y}_{m}}\rangle I_{3b^{2}}- \widetilde{B_{1m}})|_{d}\mid\\
\nonumber&\geq& \frac{\gamma_{1m}}{|k|^{\tau_{1}}}\cdot\mid |I_{3b^{2}}+\frac{adj(\mathbf{i}\langle k,\omega_{m}\rangle I_{3b^{2}}- B_{1m})}{|\mathbf{i}\langle k,\omega_{m}\rangle I_{3b^{2}}- B_{1m}|_{d}}(\mathbf{i}\langle k,\widehat{N^{y}_{m}}\rangle I_{3b^{2}}- \widetilde{B_{1m}})|_{d}\mid\\
\nonumber&\geq& \frac{\gamma_{1m}}{|k|^{\tau_{1}}}\cdot\left(1-\frac{1}{m+1}\right)^{18b^{4}}\\
\nonumber&\geq& \frac{\gamma_{1\{m+1\}}}{|k|^{\tau_{1}}}.
\end{eqnarray}
Similarly,\ for $ |k|\leq K_{m}$,\ one has
\begin{eqnarray}
\nonumber &&\mid|\mathbf{i}(\langle k,\omega_{m+1}\pm \Omega^{j}_{m+1})I_{4b^{2}}- B_{3{m+1}}|_{d}\mid \geq \frac{\gamma_{3\{m+1\}}}{|k|^{\tau_{3}}},
\end{eqnarray}
and for $ 0 < |k|\leq K_{m}$,\ we obtain
\begin{eqnarray}
\nonumber&&\mid|\mathbf{i}\langle k,\omega_{m+1}\rangle I_{2b^{2}}- B_{4{m+1}}|_{d}\mid \geq \frac{\gamma_{4\{m+1\}}}{|k|^{\tau_{4}}}.
\end{eqnarray}
Moreover,\ one has
\[
\|X_{\hat{N}_{m}}\|_{D_{m}^{1},\Pi_{m}}\leq \|X_{R_{m}}\|_{D_{m}^{1},\Pi_{m}}\leq \varepsilon_{m},
\]
which implies
\[
|\omega_{m+1}-\omega_{m}|\lessdot \|X_{R_{m}}\|_{D_{m}^{1},\Pi_{m}},
\]
and
\[
\|(\Omega_{m+1}-\Omega_{m})z\|_{l^{a,\bar{p}}}\lessdot r_{m}\|X_{R_{m}}\|_{D_{m}^{1},\Pi_{m}}.
\]
Hence,\ on $ \Pi_{m}$ and with $ -\delta \leq \bar{p}-p $,\ we have
\[
|\Omega_{m+1}-\Omega_{m}|_{-\delta}\leq |\Omega_{m+1}-\Omega_{m}|_{\bar{p}-p}\lessdot \|X_{R_{m}}\|_{D_{m}^{1},\Pi_{m}},
\]
which ends the proof of (\ref{077*}).

Thirdly,\ under the assumptions (\ref{073})-(\ref{075}) at stage $ m$,\ we get from (\ref{060}) that
\begin{eqnarray}
\nonumber\|X_{R_{m+1}}\|_{D_{m+1}^{1},\Pi_{m+1}} &\lessdot& \eta^{-2}_{m}
\varepsilon^{\frac{11}{6}}_{m}+
\eta_{m}\varepsilon_{m} + K^{n}_{m}e^{-K_{m}(s_{m}-s_{m+1})}\varepsilon_{m}\\
\nonumber &\lessdot& \varepsilon_{m+1},
\end{eqnarray}
which ends the proof of (\ref{078}).

Finally,\ we estimate the measures of the resonant zones.\ Since
\[
\Pi_{m+1}=\Pi_{m}\setminus \bigcup_{K_{m-1}<|k|\leq K_{m},(k,l)\in\mathcal{Z}}\mathcal{R}_{kl}(m) \setminus \bigcup_{0 <|k|\leq K_{m},i=1,4}\mathcal{R}_{ik}(m)\setminus \bigcup_{|k|\leq K_{m}}\mathcal{R}_{3k}(m),
\]
we have
\begin{eqnarray}
\nonumber \mathbf{Meas} \Pi_{m+1} &=&  \mathbf{Meas}\Pi_{m}- \sum_{(k,l)\in\mathcal{Z},K_{m-1}<|k|\leq K_{m}} \mathcal{R}_{kl}(m)\\
\nonumber &&-\sum_{i=1,4,0 <|k|\leq K_{m}}\mathcal{R}_{ik}(m)-\sum_{|k|\leq K_{m}}\mathcal{R}_{3k}(m)\\
\nonumber  &=&\mathbf{Meas} \Pi_{m}-\sum_{ K_{m-1} <|k|\leq K_{m}}\gamma^{\mu}_{1}\cdot\frac{1}{|k|^{\tau}}-\sum_{0 <|k|\leq K_{m}}L_{1}(\gamma^{\frac{1}{3b^{2}}}_{1}/|k|^{\tau})\cdot\frac{1}{m^{2}}\\
\nonumber &&
-\sum_{0 <|k|\leq K_{m}}L_{4}(\gamma^{\frac{1}{2b^{2}}}_{1}/| k|^{\tau})\cdot\frac{1}{m^{2}}-2\sum_{|k|\leq K_{m}}L_{3}(\gamma^{\frac{1}{4b^{2}}}_{1}/| k|^{\tau-1})\cdot\frac{1}{m^{2}}\\
\nonumber &=& \mathbf{Meas} \Pi_{m}-\gamma^{\mu}_{1}\cdot\frac{1}{1+K_{m-1}}-\gamma^{\frac{1}{4b^{2}}}_{1}\cdot\frac{1}{m^{2}},
\end{eqnarray}
which ends the proof of (\ref{079}).
\end{proof}

\section{convergence}
Let $ D(\frac{s_{1}}{2},0,0)\subset \bigcap^{\infty}_{m=1}D(s_{m},r_{m},r_{m}),
\Phi=\Pi^{\infty}_{m=1}\Phi_{m}$ and $ \Pi_{\gamma_{1}}=\bigcap^{\infty}_{m=1}\Pi_{m}$.\ By the Lemma 3.1,\ we conclude that $ \Phi,D\Phi,H_{m},X_{H_{m}}$ converge uniformly on the domain $ D(\frac{s_{1}}{2},0)\times\Pi_{\gamma_{1}}=D(\frac{7}{16}s_{0},0,0)\times\Pi_{\frac{3}{4}\gamma_{0}}$ with

\[
H_{\infty}(x,y,z^{*},\bar{z}^{*},\xi):\lim_{m\rightarrow\infty}H_{m}
=\breve{N}(y,z^{*},\bar{z}^{*},\xi)+\breve{R}(x,y,z^{*},\bar{z}^{*},\xi),
\]
where
\begin{eqnarray}
\nonumber  \breve{N}&=& \breve{N}^{x}(\xi)+\langle\breve{\omega}(\xi),y\rangle
+\langle\breve{\Omega}(\xi)z,\bar{z}\rangle
+ \langle \breve{N}^{z_{0}}(\xi),z_{0}\rangle +\langle \breve{N}^{\bar{z}_{0}}(\xi),\bar{z}_{0}\rangle\\
\nonumber      &&
 + \langle \breve{N}^{z_{0}z_{0}}(\xi)z_{0},z_{0}\rangle
 + \langle \breve{N}^{z_{0}\bar{z}_{0}}(\xi)z_{0},\bar{z}_{0}\rangle
+ \langle \breve{N}^{\bar{z}_{0}\bar{z}_{0}}(\xi)\bar{z}_{0},\bar{z}_{0}\rangle,
\end{eqnarray}
and

\[
\begin{split}
\breve{R}(x,y,z^{*},\bar{z}^{*},\xi) &=\sum_{k,\alpha\in \mathbb{N}^{n},\beta,\gamma\in \mathbb{N}^{\mathbb{N}},2|\alpha|+|\beta|+|\gamma|\geq3}
\breve{R}^{k\alpha\beta\gamma}(\xi)y^{\alpha}\{z^{*}\}^{\beta}\{\bar{z}^{*}\}^{\gamma}e^{\mathbf{i}\langle k,x\rangle}.
\end{split}
\]
Moreover,\ the following estimates hold:\\

$(\mathbf{1})$ for each $ \xi\in \Pi_{1}$,\ the symplectic map
\[
\Phi:\ D(s,r,r)\times \Pi \rightarrow D(\frac{7}{16}s,0,0)\times \Pi_{\gamma},
\]
satisfies:
\[
\|\Phi-id\|_{D(\frac{7}{16}s,0,0),\Pi_{\gamma}}
\lessdot\varepsilon;
\]
and
\[
\|D\Phi-Id\|_{D(\frac{7}{16}s,0,0),\Pi_{\gamma}}
\lessdot\varepsilon;
\]

$(\mathbf{2})$ the frequencies $\breve{\omega}(\xi) $ and $\breve{\Omega}(\xi)$ satisfy:
\[
\begin{split}
|\breve{\omega}(\xi)-\omega(\xi)|_{\Pi_{\gamma}}
+
|\breve{\Omega}(\xi)-\Omega(\xi)|_{-\delta,\Pi_{\gamma}}
&\lessdot\varepsilon;
\end{split}
\]

$(\mathbf{3})$ the perturbation $\breve{R}(x,y,z^{*},\bar{z}^{*},\xi) $ satisfies:
\begin{equation}
\begin{split}
\|X_{\breve{R}}\|_{D(\frac{7}{16}s_{0},0,0),\Pi_{\gamma}}\lessdot\varepsilon;
\end{split}
\end{equation}

$(\mathbf{4})$ the measure of the $ \Pi_{\gamma}$ satisfies:
\begin{equation}
\begin{split}
\mathbf{Meas}\Pi_{\gamma} \geq (\mathbf{Meas}\Pi_{1})(1-O(\gamma^{\nu})),
\end{split}
\end{equation}
where $ \nu=\min\{\mu,\frac{1}{4b^{2}}\}$.
\begin{proof}
\ See the details in \cite{po1996ak}.
\end{proof}

\section{Proof of Theorem \ref{mainthm}}
\begin{proof}\ In view of Lemma \ref{ml},\ we get
\[
\breve{H}=\breve{N}(y,z^{*},\bar{z}^{*},\xi)+\breve{R}(x,y,z^{*},\bar{z}^{*},\xi),
\]
where
\begin{eqnarray}
\breve{N}(x,y,z^{*},\bar{z}^{*},\xi) &=& \breve{N}^{x}(\xi)+\langle\breve{\omega}(\xi),y\rangle
+\langle\breve{\Omega}(\xi)z,\bar{z}\rangle
      + \langle\breve{N}^{z_{0}}(\xi),z_{0}\rangle +\langle\breve{N}^{\bar{z}_{0}}(\xi),\bar{z}_{0}\rangle\\
\nonumber  &&
 +\langle\breve{N}^{z_{0}z_{0}}(\xi)z_{0},z_{0}\rangle
 +\langle\breve{N}^{z_{0}\bar{z}_{0}}(\xi)z_{0},\bar{z}_{0}\rangle
+\langle\breve{N}^{\bar{z}_{0}\bar{z}_{0}}(\xi)\bar{z}_{0},\bar{z}_{0}\rangle,
\end{eqnarray}
and
\begin{eqnarray}
\breve{R}(x,y,z^{*},\bar{z}^{*},\xi) &=&\sum_{k\in\mathbb{Z}^{n},\alpha\in \mathbb{N}^{n},\beta,\gamma\in \mathbb{N}^{\mathbb{N}},2|\alpha|+|\beta|+|\gamma|\geq3}
\widehat{\breve{R}^{k\alpha\beta\gamma}}(k,\xi)y^{\alpha}\{z^{*}\}^{\beta}\{\bar{z}^{*}\}^{\gamma}e^{\mathbf{i}\langle k,x\rangle}.
\end{eqnarray}

If $ \breve{N}^{z_{0}}(\xi)=0 $\ and\ $ \breve{N}^{\bar{z}_{0}}(\xi)=0 $,\ then the corresponding Hamiltonian equation defined by (\ref{001}) can be written into the form
\begin{equation}
\begin{cases}
\dot{x}=\frac{\partial\breve{H}}{\partial y}=\breve{\omega}(\xi)+\mathcal{O}(|y|+\|z\|_{a,p}),\\
\dot{y}=\frac{\partial\breve{H}}{\partial x}=\mathcal{O}(|y|^{2}+|y|\|z\|_{a,p}+\|z\|^{3}_{a,p}),\\
\dot{z}_{0}=\mathbf{i}\frac{\partial\breve{H}}{\partial \bar{z}_{0}}= \mathbf{i}(\breve{N}^{z_{0}\bar{z}_{0}}(\xi)z_{0}
+2\breve{N}^{\bar{z}_{0}\bar{z}_{0}}(\xi)\bar{z}_{0}
+\breve{N}^{\bar{z}_{0}}(\xi)
+\mathcal{O}(|y|+\|z\|^{2}_{a,p})),\\
\dot{\bar{z}}_{0}=-\mathbf{i}\frac{\partial\breve{H}}{\partial z_{0}}=-\mathbf{i}(2\breve{N}^{z_{0}z_{0}}(\xi)z_{0}+\breve{N}^{z_{0}\bar{z}_{0}}(\xi)\bar{z}_{0}
+\breve{N}^{z_{0}}(\xi)
+\mathcal{O}(|y|+\|z\|^{2}_{a,p})),\\
\dot{z}_{j}=\mathbf{i}\frac{\partial\breve{H}}{\partial \bar{z}_{j}}=\mathbf{i}(\breve{\Omega}_{j}(\xi)z_{j}
+\mathcal{O}(|y|+\|z\|^{2}_{a,p})),\  j\in \mathbb{N}_{+}\setminus\mathcal{J},\\
\dot{\bar{z}}_{j}=-\mathbf{i}\frac{\partial\breve{H}}{\partial z_{j}}=-\mathbf{i}(\breve{\Omega}_{j}(\xi)\bar{z}_{j}
+\mathcal{O}(|y|+\|z\|^{2}_{a,p})),\  j\in \mathbb{N}_{+}\setminus\mathcal{J}.
\end{cases}
\end{equation}
It is easy to verify that
\[
\mathcal{T}^{n}_{0}={\mathbb{T}}^{n}\times\{y=0\}\times\{z^{*}=0\}\times\{\bar{z}^{*}=0\}
\]
is an embedding torus of the Hamiltonian vector field $ X_{\breve{H}}$ with frequency $ \breve{\omega}$.\ Moreover,\ $ \Phi(\mathcal{T}_{0}^{n}\times\{\xi\})$ is the invariant torus of the original Hamiltonian function $ H$.\
We finish the proof of the existence of KAM torus in this case.

If $ \breve{N}^{z_{0}}(\xi) \neq \vec{0} $\ or\ $\breve{N}^{\bar{z}_{0}}(\xi)\neq \vec{0} $,\ then we can let $ \left(|\breve{N}^{z_{0}}(\xi)|_{2}^{2}
+|\breve{N}^{\bar{z}_{0}}(\xi)|_{2}^{2}\right)^{\frac{1}{2}} = \delta_{0} > 0 $.\ Since $ \underset{m\rightarrow\infty}{\lim}\widehat{J_{m}^{z_{0}}} (\xi)= \breve{N}^{z_{0}}(\xi)$ and $ \underset{m\rightarrow\infty}{\lim}\widehat{J_{m}^{\bar{z}_{0}}} (\xi)= \breve{N}^{\bar{z}_{0}}(\xi)$,\ there exists a fixed $ m_{0}$ such that for any $ m > m_{0} $,
\begin{eqnarray}\label{a000}
\sqrt{|{J}_{m}^{{z}_{0}}(\xi)|_{2}^{2}
+|{J}_{m}^{\bar{z}_{0}}(\xi)|_{2}^{2}} \geq \frac{\delta_{0}}{2}.
\end{eqnarray}
More exactly,\ we will choose sufficiently large $ m $ such that
\begin{eqnarray}\label{a00}
 \delta_{0} > 20\varepsilon^{\frac{7}{6}}_{m-1}.
\end{eqnarray}
Thus,\ considering $ H_{m}(x,y,z^{*},\bar{z}^{*},\xi) = N_{m}(y,z^{*},\bar{z}^{*},\xi)+ R_{m}(x,y,z^{*},\bar{z}^{*},\xi)$,\ one obtains
\begin{equation}\label{a0}
\begin{cases}
\dot{x}=\frac{\partial H_{m}}{\partial y}=\omega_{m}(\xi)+\frac{\partial R_{m}}{\partial y},\\
\dot{y}=-\frac{\partial H_{m}}{\partial x}=-\frac{\partial R_{m}}{\partial x},\\
\dot{z}_{0}=\mathbf{i}\frac{\partial H_{m}}{\partial \bar{z}_{0}}= \mathbf{i}\left(\widehat{J_{m}^{\bar{z}_{0}}}(\xi)+\widehat{J_{m}^{z_{0}\bar{z}_{0}}}(\xi)z_{0}
+2\widehat{J_{m}^{\bar{z}_{0}\bar{z}_{0}}}(\xi)\bar{z}_{0}
+\frac{\partial R_{m}}{\partial \bar{z}_{0}}\right),\\
\dot{\bar{z}}_{0} =-\mathbf{i}\frac{\partial H_{m}}{\partial  z_{0}}= -\mathbf{i}\left(\widehat{J_{m}^{{z}_{0}}}(\xi)+2\widehat{J_{m}^{z_{0}z_{0}}}(\xi)z_{0}
+\widehat{J_{m}^{z_{0}\bar{z}_{0}}}(\xi)\bar{z}_{0}
+\frac{\partial R_{m}}{\partial z_{0}}\right),\\
\dot{z}_{j}=\mathbf{i}\frac{\partial H_{m}}{\partial \bar{z}_{j}}
=\mathbf{i}\left(\Omega^{j}_{m}(\xi)z_{j}
+\frac{\partial R_{m}}{\partial \bar{z}_{j}}\right),\  j\in \mathbb{N}_{+}\setminus\mathcal{J},\\
\dot{\bar{z}}_{j}=-\mathbf{i}\frac{\partial H_{m}}{\partial  z_{j}}=-\mathbf{i}\left(\Omega^{j}_{m}(\xi)\bar{z}_{j}
+\frac{\partial R_{m}}{\partial z_{j}}\right),\  j\in \mathbb{N}_{+}\setminus\mathcal{J},
\end{cases}
\end{equation}
on $\mathcal{D}(s_{m},r_{m},r_{m})$.\\
Let
\[
X_{0} = \begin{pmatrix}
z_{0}\\
\bar{z}_{0}
\end{pmatrix},
X_{j} = \begin{pmatrix}
z_{j}\\
\bar{z}_{j}
\end{pmatrix},\  j\in \mathbb{N}_{+}\setminus\mathcal{J},
\]
and
\[
\alpha_{0}= \begin{pmatrix}
 \mathbf{i}\widehat{J_{m}^{{z}_{0}}}(\xi)\\
- \mathbf{i}\widehat{J_{m}^{\bar{z}_{0}}}(\xi)
\end{pmatrix}.
\]
We also let
\[
A_{0} =
\begin{pmatrix}
\mathbf{i}\widehat{J_{m}^{z_{0}\bar{z}_{0}}}(\xi) &
2\mathbf{i}\widehat{J_{m}^{\bar{z}_{0}\bar{z}_{0}}}(\xi)\\
-2\mathbf{i}\widehat{J_{m}^{z_{0}z_{0}}}(\xi) &
-\mathbf{i}\widehat{J_{m}^{z_{0}\bar{z}_{0}}}(\xi)
\end{pmatrix},
A_{j} =
\begin{pmatrix}
\mathbf{i}\Omega_{j}(\xi) & 0\\
0 &
-\mathbf{i}\Omega_{j}(\xi)
\end{pmatrix},
\]
and
\[
g_{0}=  \begin{pmatrix}\mathbf{i}\frac{\partial R_{m}}{\partial \bar{z}_{0}}\\
-\mathbf{i}\frac{\partial R_{m}}{\partial z_{0}}
\end{pmatrix},
g_{j}=  \begin{pmatrix}\mathbf{i}\frac{\partial R_{m}}{\partial \bar{z}_{j}}\\
-\mathbf{i}\frac{\partial R_{m}}{\partial z_{j}}
\end{pmatrix}, j\in \mathbb{N}_{+}\setminus\mathcal{J},
\]
for any $ j\in \mathbb{N}_{+}\setminus\mathcal{J}$.\\
Then the last four equations of (\ref{a0}) can be written into the form
\begin{eqnarray}\label{a1}
\dot{X}_{0} = \alpha_{0} + A_{0}X_{0} + g_{0},\  \dot{X}_{j} = A_{j}X_{j} + g_{j},j\in \mathbb{N}_{+}\setminus\mathcal{J},
\end{eqnarray}
respectively.\\
Let us pass to the new variables
\begin{eqnarray}\label{a01}
\tilde{X}_{0} = e^{A_{0}t}X_{0},\  \tilde{X}_{j} = e^{A_{j}t}X_{j}, j\in \mathbb{N}_{+}\setminus\mathcal{J},
\end{eqnarray}
and rewritten (\ref{a1}) as
\begin{eqnarray}
\label{a2}
\dot{\tilde{X}}_{0}= e^{-A_{0}t}\alpha_{0}+\tilde{g}_{0},\  \dot{\tilde{X}}_{j}= \tilde{g}_{j}, j\in \mathbb{N}_{+}\setminus\mathcal{J},
\end{eqnarray}
where
\[
\tilde{g}_{0}= e^{-A_{0}t}{g}_{0},\
\tilde{g}_{j}= e^{-A_{j}t}{g}_{j}, j\in \mathbb{N}_{+}\setminus\mathcal{J}.
\]
Since
\[
\|A_{0} \| = \underset{|x|_{2}\neq 0}{\sup}\frac{|A_{0}x|_{2}}{|x|_{2}}\leq 2b\varepsilon_{0} \ll 1,\   x\in \mathbb{C}^{2b},
\]
then for any $ 0\leq t \leq 1$,\ one has
\begin{eqnarray}\label{a3*}
\|e^{-A_{0}t}\|= \left\|E + \underset{k\geq 1}{\sum}\frac{(-A_{0}t)^{k}}{k!}\right\| \leq 1+\underset{k\geq 1}{\sum}\frac{\|A_{0}\|^{k}}{k!} < 2,
\end{eqnarray}
and
\begin{eqnarray}\label{a3}
\|e^{-A_{0}t}\|= \left\|E + \underset{k\geq 1}{\sum}\frac{(-A_{0}t)^{k}}{k!}\right\| \geq 1-\underset{k\geq 1}{\sum}\frac{\|A_{0}\|^{k}}{k!}> \frac{1}{2}.
\end{eqnarray}
Fix an initial value $\|z^{*}(0)\|_{a,p}+\|\bar{z}^{*}(0)\|_{a,p}\leq \varepsilon^{\frac{7}{6}}_{m-1} $.\ Since
\[
\|z^{*}(0)\|^{2}_{a,p}=\sum_{j_{m}\in \mathcal{J}}|z_{j_{m}}(0)|^{2}j_{m}^{2p}e^{2aj_{m}}+ \sum_{j\in \mathbb{N}_{+}\setminus \mathcal{J}}|z_{j}(0)|^{2}j^{2p}e^{2aj},
\]
and
\[
\|\bar{z}^{*}(0)\|^{2}_{a,p}=\sum_{j_{m}\in \mathcal{J}}|\bar{z}_{j_{m}}(0)|^{2}j_{m}^{2p}e^{2aj_{m}}+ \sum_{j\in \mathbb{N}_{+}\setminus \mathcal{J}}|\bar{z}_{j}(0)|^{2}j^{2p}e^{2aj},
\]
then one has
\[
\left(\underset{j_{m}\in\mathcal{J}}{\sum}|{X}_{0m}(0)|^{2}_{{2}}j_{m}^{2p}e^{2aj_{m}}
+\underset{j\in \mathbb{N}_{+}\setminus\mathcal{J}}{\sum}|{X}_{j}(0)|^{2}_{{2}}j^{2p}e^{2aj}\right)^{\frac{1}{2}}\leq \sqrt{2}(\|z^{*}(0)\|_{a,p}+\|\bar{z}^{*}(0)\|_{a,p}),
\]
where
$ {X}_{0m} = \begin{pmatrix}
z_{j_m}\\
\bar{z}_{j_m}
\end{pmatrix}$ for any $ j_{m}\in \mathcal{J} $
and $ {X}_{j} = \begin{pmatrix}
z_{j}\\
\bar{z}_{j}
\end{pmatrix}$ for any $ j\in \mathbb{N}_{+}\setminus\mathcal{J}$.\\
It follows from (\ref{a01}) that
\begin{eqnarray}\label{238}
\left(\underset{j_{m}\in\mathcal{J}}{\sum}|\tilde{X}_{0m}(0)|^{2}_{{2}}j_{m}^{2p}e^{2aj_{m}}
+\underset{j\in \mathbb{N}_{+}\setminus\mathcal{J}}{\sum}|\tilde{X}_{j}(0)|^{2}_{{2}}j^{2p}e^{2aj}\right)^{\frac{1}{2}}\leq \sqrt{2}\varepsilon^{\frac{7}{6}}_{m-1}.
\end{eqnarray}
Note
\begin{eqnarray}\label{239}
\|\tilde{X}_{0}(0)\|^{2}_{l^{a,p}}= \underset{j_{m}\in\mathcal{J}}{\sum}|\tilde{X}_{0m}(0)|^{2}_{{2}}j_{m}^{2p}e^{2aj_{m}},
\end{eqnarray}
and
\begin{eqnarray}\label{240}
\|g_{0}\|^{2}_{l^{a,p}}= \underset{j_{m}\in\mathcal{J}}{\sum}|g_{0m}(0)|^{2}_{{2}}j_{m}^{2p}e^{2aj_{m}}.
\end{eqnarray}
By integrating $ t$ from $0$ to $1$ of (\ref{a2}),\ one thus obtains
\begin{eqnarray}
\nonumber\left|\tilde{X}_{0}(1)-\int^{1}_{0} e^{-A_{0}t}\alpha_{0}\mathrm{d}t\right|_{{2}}&\leq& \sqrt{2}|\tilde{X}_{0}(0)|_{{2}}+ \sqrt{2}\int^{1}_{0} |\tilde{g}_{0}|_{2}\mathrm{d}t\\
\nonumber&\leq& \sqrt{2}|\tilde{X}_{0}(0)|_{{2}}+ 2\sqrt{2}\int^{1}_{0} |{g}_{0}|_{2}\mathrm{d}t\
  (\mbox{by}\   (\ref{a3*})\  )\\
\nonumber&\leq& \sqrt{2}\|\tilde{X}_{0}(0)\|_{l^{a,p}}+ 2\sqrt{2}\int^{1}_{0} \|{g}_{0}\|_{l^{a,p}}\mathrm{d}t\   ( \mbox{by}\   (\ref{239})\  \mbox{and}\  (\ref{240}) ) \\
\nonumber&\leq& \sqrt{2}\|\tilde{X}_{0}(0)\|_{l^{a,p}}+4\int^{1}_{0} \|{X_{R_m}}\|_{\mathcal{D}_{m}\times\Pi_{m}}\mathrm{d}t\\
\nonumber &\leq& 2\varepsilon^{\frac{7}{6}}_{m-1}+ 4\varepsilon_{m}\   (\mbox{by}\  (\ref{238})\  ) \\
\nonumber &\leq& 3\varepsilon^{\frac{7}{6}}_{m-1}.
\end{eqnarray}
Consequently,\ we obtain
\begin{eqnarray}\label{a5}
|\tilde{X}_{0}(1)|_{{2}}&\geq& \left|\int^{1}_{0} e^{-A_{0}t}\alpha_{0}\mathrm{d}t\right|_{{2}}-3\varepsilon^{\frac{7}{6}}_{m-1}\\
\nonumber &\geq& {\frac{\delta_{0}}{4}}-3\varepsilon^{\frac{7}{6}}_{m-1}\  ( \mbox{by}\  (\ref{a000})\    \mbox{and}\  (\ref{a3}) )\\
\nonumber &>& 2\varepsilon^{\frac{7}{6}}_{m-1}\  ( \mbox{by}\  (\ref{a00}) ),
\end{eqnarray}
which implies that no invariant torus exists in the domain
\begin{eqnarray}\nonumber
\left(\underset{j_{m}\in\mathcal{J}}{\sum}|\tilde{X}_{0m}(0)|^{2}_{{2}}j_{m}^{2p}e^{2aj_{m}}
+\underset{j\in \mathbb{N}_{+}\setminus\mathcal{J}}{\sum}|\tilde{X}_{j}(0)|^{2}_{{2}}j^{2p}e^{2aj}\right)^{\frac{1}{2}}\leq \sqrt{2}\varepsilon^{\frac{7}{6}}_{m-1}.
\end{eqnarray}
Let $ \Xi_{m}= \{(x,y,z^{*},\bar{z}^{*}):|\Im x| \leq s_{m},|y|\leq r_{m}^{2},
\|z^{*}\|_{a,p}+\|\bar{z}^{*}\|_{a,p}\leq \varepsilon^{\frac{7}{6}}_{m-1} \}$ and $\Phi^{m-1}=\Pi^{m-1}_{j=0}\Phi_{j}$.\ It follows from (\ref{a5}) that there exists no invariant torus for the Hamiltonian system defined by (\ref{001}) on $ \Phi^{m-1}(\Xi_{m}\times\{\xi\})$.
\end{proof}

\section{application to NLS}
We discuss the nonlinear Schr$\ddot{o}$dinger equation
\begin{equation}\label{eq}
\mathbf{i}u_{t}-u_{xx}+|u|^{2}u=0
\end{equation}
on the finite $x$-interval $[0,2\pi]$ with periodic boundary conditions
\[
u(t,x)=u(t,x+2\pi)=0,\  u(x,t)=u(-x,t).
\]
Denote the Sobolev space of complex valued $ L^{2}$-functions $[0,2\pi]$ with an $ L^{2}$-derivative and vanishing boundary values by $\mathcal{P}=W_{0}^{1}([0,2\pi])$.\ With the inner product
\[
\langle u,v\rangle=Re\int_{0}^{2\pi}u\overline{v}\mathrm{d}x,
\]
and the Hamiltonian
\[
H=\frac{1}{2}\langle Au,u\rangle+\frac{1}{4}\int_{0}^{2\pi}|u|^{4}\mathrm{d}x
\]
where $ A=-\frac{d^{2}}{dx^{2}}$,\ the system can be written in the  Hamiltonian form

\[
\dot{u}=\mathbf{i}\nabla H(u)
\]
the gradient of $ H$ is defined with respect to $ \langle \cdot,\cdot\rangle$,\ and the dot indicates differentiation with respect to time.

Denote $ \mathbb{N}= \{0,1,...,n,...\}$ and $ \mathbb{N}_{+}= \{1,...,n,...\}$.\ Let
\[
\begin{cases}
\phi_{0}(x)=\sqrt{\frac{1}{2\pi}},\  \lambda_{0}= 0, \\
\phi_{j}(x)=\sqrt{\frac{1}{\pi}}\cos jx,\  \lambda_{j}=j^{2},\  j\in\mathbb{N}_{+},
\end{cases}
\]
be the basic modes and their frequencies for the linear Schr$\ddot{o}$dinger equation $\mathbf{i}u_{t}-u_{xx}=0$ with periodic boundary conditions.\ We rewrite $H$ as a Hamiltonian in infinitely many coordinates by making the ansatz
\[
u(t,x)=\sum_{j\in\mathbb{N}}q_{j}(t)\phi_{j}(x).
\]
The coordinates are taken from the Hilbert space $ l^{a,p}$ of all complex-valued sequences $ q=(q_{0}, q_{1},...)$ with
\[
|| q||^{2}_{a,p}= |q_{0}|^{2}+\sum_{j\in\mathbb{N}_{+}}|q_{j}|^{2}j^{2p}e^{2aj} < \infty,
\]
where $ a > 0$ and $ p> \frac{1}{2}$ will be fixed later.\\
We then obtain the Hamiltonian
\begin{eqnarray}\label{199}
H=\Lambda+G= \frac{1}{2}\underset{j\in\mathbb{N}}{\sum}\lambda_{j}| q_{j}|^{2}+\frac{1}{4}\int_{0}^{2\pi}|u|^{4}\mathrm{d}x
\end{eqnarray}
on the phase space $ l^{a,p}$ with the symplectic structure ${\mathbf{i}}\sum_{j\in\mathbb{N}}dq_{j}\bigwedge d\bar{q}_{j}$.\ The corresponding equation is
\begin{eqnarray}\label{200}
\dot{q}_{j}=2\mathbf{i}\frac{\partial H}{\partial \bar{q}_{j}},\  j\in\mathbb{N}.
\end{eqnarray}
\begin{lemma}\label{l1} Let $ a > 0$ and $ p> \frac{1}{2}$.\ If a curse $ I\rightarrow l^{a,p},  t\mapsto q(t)$ is an analytic solution of $(\ref{200})$,\ then
\[
u(t,x)=\sum_{j\in\mathbb{N}}q_{j}(t)\phi_{j}(x),
\]
is a solution of $(\ref{200})$ which is analytic on $ I\times [0,2\pi]$.
\end{lemma}
\begin{proof}\ More details can be found in \cite{ku1996in}.
\end{proof}
For the nonlinearity $ |u|^{2}u$,\ we find
\begin{eqnarray}\label{201}
G=\frac{1}{4}\int_{0}^{2\pi}|u|^{4}\mathrm{d}x=\frac{1}{4}\sum_{i,j,k,l}
G_{ijkl}q_{i}q_{j}\bar{q}_{k}\bar{q}_{l}
\end{eqnarray}
with
\[
G_{ijkl}=\int_{0}^{2\pi}\phi_{i}\phi_{j}\phi_{k}\phi_{l}\mathrm{d}x.
\]
\begin{proposition}\label{p1}
From the Hamiltonian $ H=\Lambda+G $ with the nonlinearity (\ref{201}),\ there exists a real analytic,\ symplectic change coordinates $ \Gamma$ in some neighborhood of the origin in $ l^{a,p} $ that takes the Hamiltonian (\ref{199}) into
\begin{eqnarray}\label{2011}
H \circ \Gamma = \Lambda + \bar{G} + K,
\end{eqnarray}
where Hamiltonian vector fields $ X_{\bar{G}}$ and $ X_{K} $ are real analytic vector fields in a neighborhood of the origin in $ l^{a,p} $,
\begin{eqnarray}\label{202}
\bar{G} = \underset{i,j\in\mathbb{N}}{\sum}\bar{G}_{ij}|q_{i}|^{2}|q_{j}|^{2},\   |K| = \mathcal{O}(||q||_{a,p}^{6}),
\end{eqnarray}
with uniquely determined coefficients $ \bar{G}_{ij} = {G}_{iijj} = \begin{cases}
\frac{2+\delta_{ij}}{16\pi},\ &i,j \in\mathbb{N}_{+},\\
\frac{1}{8\pi},\ &\mbox{either one of}\  i,j=0\  \mbox{or both}.
\end{cases}$In addition,\ $ K(q,\bar{q})= \underset{a,a'}{\sum}K_{aa'}\prod_{n=0}^{\infty}q^{a_{n}}_{n}\bar{q}^{a'_{n}}_{n}$ ($ a,a'\in \mathbb{N}^{\mathbb{N}}$) has properties : $K_{aa'}=0 $ if $ \underset{n}{\sum}(a_{n}-a'_{n})n \neq 0$ and $ \underset{n}{\sum}a_{n}+a'_{n} $ is even ($ \geq 6 $) for any monomial $ \prod_{n=0}^{\infty}q^{a_{n}}_{n}\bar{q}^{a'_{n}}_{n}$.
\end{proposition}
\begin{proof}\ Let $ \Gamma = X^{t}_{F}|_{t=1}$ be the time-1-map of the flow of the Hamiltonian vector field $ X_{F}$ given by the Hamiltonian
\[
F = \frac{1}{4}\sum_{i,j,k,l}
F_{ijkl}q_{i}q_{j}\bar{q}_{k}\bar{q}_{l}.
\]
Expanding at $ t= 0$ and using Taylor's formula we have
\begin{eqnarray}
\nonumber H\circ \Gamma &=& \Lambda + ( G +\{ \Lambda,F\}) + \mathcal{O}(||q||_{a,p}^{6})\\
\nonumber &=& \Lambda + \bar{G} +K
\end{eqnarray}
with
\[
\{\Lambda,F\}= -\frac{1}{4}\mathbf{i}\sum_{i,j,k,l}
(\lambda_{i}+\lambda_{j}-\lambda_{k}-\lambda_{l})F_{ijkl}q_{i}q_{j}\bar{q}_{k}\bar{q}_{l}.
\]
Let
\[
\mathbf{i}F_{ijkl}=
\begin{cases}
\frac{G_{ijkl}}{\lambda_{i}+\lambda_{j}-\lambda_{k}-\lambda_{l}}\  &\mbox{for}\  i\pm j \pm k \pm l=0\  \mbox{and}\  \{i,j\}\neq \{k,l\},\\
0\                     &\mbox{otherwise}.
\end{cases}
\]
Thus we have
\[
\bar{G} = \underset{i,j\in\mathbb{N}}{\sum}\bar{G}_{ij}|q_{i}|^{2}|q_{j}|^{2}
\]
with
\[
\bar{G}_{ij} = {G}_{iijj} = \begin{cases}
\frac{2+\delta_{ij}}{16\pi},\ &i,j \in\mathbb{N}_{+},\\
\frac{1}{8\pi},\  &\mbox{either one of}\  i,j=0\  \mbox{or both}.
\end{cases}
\]

Next we prove $K(q,\bar{q})= \underset{a,a'}{\sum}K_{aa'}\prod_{n=0}^{\infty}q^{a_{n}}_{n}\bar{q}^{a'_{n}}_{n}$ ($ a,a'\in \mathbb{N}^{\mathbb{N}}$) which has the properties that $K_{aa'}=0 $ if $ \underset{n}{\sum}(a_{n}-a'_{n})n \neq 0$ and $ \underset{n}{\sum}a_{n}+a'_{n} $ is even ($ \geq 6 $) for any monomial $ \prod_{n=0}^{\infty}q^{a_{n}}_{n}\bar{q}^{a'_{n}}_{n}$.\\
Since
\begin{eqnarray}
\nonumber K &=& \{G,F\} + \frac{1}{2!}\{\{ \Lambda,F\},F\} +  \frac{1}{2!}\{\{G,F\},F\}\\
\nonumber &&+ \cdot\cdot\cdot+\frac{1}{n!}\{\cdot\cdot\cdot\{ \Lambda,F\}\cdot\cdot\cdot,F\}+\frac{1}{n!}\{\cdot\cdot\cdot\{ G,F\}\cdot\cdot\cdot,F\}+\cdot\cdot\cdot
\end{eqnarray}
We first consider $  \{G,F\}$,\ due to
\begin{eqnarray}
\nonumber G(q,\bar{q})&=& \underset{a,a',\underset{n}{\sum}n(a_{n}-a'_{n})=0}{\sum}G_{aa'}\prod_{n=0}^{\infty}q^{a_{n}}_{n}\bar{q}^{a'_{n}}_{n},\\
\nonumber F(q,\bar{q})&=& \underset{b,b',\underset{n}{\sum}n(b_{n}-b'_{n})=0}{\sum}F_{bb'}\prod_{n=0}^{\infty}q^{b_{n}}_{n}\bar{q}^{b'_{n}}_{n},
\end{eqnarray}
then
\begin{eqnarray}
\nonumber \{G,F\} &=& \mathbf{i}\underset{n}{\sum}\left(\frac{\partial G}{\partial q_{n}}\frac{\partial F}{\partial \bar{q}_{n}}-\frac{\partial G}{\partial\bar{q}_{n}}\frac{\partial F}{\partial q_{n}}\right)\\
\nonumber &=& \mathbf{i}\sum_{\substack{a,a',\underset{n}{\sum}n(a_{n}-a'_{n})=0,\\ b,b',\underset{n}{\sum}n(b_{n}-b'_{n})=0}}G_{aa'}F_{bb'}\underset{j}{\sum}\left( \prod_{n\neq j}^{\infty}q^{a_{n}+b_{n}}_{n}
\bar{q}^{a'_{n}+b'_{n}}_{n}\right)\left((a_{j}b'_{j}-a'_{j}b_{j})q^{a_{j}+b_{j}-1}_{j}
\bar{q}^{a'_{j}+b'_{j}-1}_{j}\right)\\
\nonumber &=& \mathbf{i}\underset{j}{\sum}\underset{*}{\sum}(a_{j}b'_{j}-a'_{j}b_{j})G_{aa'}F_{bb'}
\prod_{n=0}^{\infty}q^{c_{n}}_{n}\bar{q}^{c'_{n}}_{n},
\end{eqnarray}
where
\begin{equation*}
\sum_{*}=\sum_{\substack{a,a',b,b'\\ \mbox{when}\ n\neq j, a_n+b_n=c_n,a_n'+b_n'=c_n',\\ \mbox{when}\ n=j, a_n+b_n-1=c_n,a_n'+b_n'-1=c_n'}}.
\end{equation*}
It follows easily that for every monomial,\ one has
\[
\underset{n}{\sum}n(c_{n}-c'_{n})= \underset{n\neq j}{\sum}n(a_{n}+b_{n}-a'_{n}-b'_{n})+ j(a_{j}+b_{j}-1-a'_{j}-b'_{j}+1)=0,
\]
and
\begin{eqnarray}
\nonumber \underset{n}{\sum}c_{n}+c'_{n}&=&\underset{n\neq j}{\sum}(a_{n}+b_{n}+a'_{n}+b'_{n})+ (a_{j}+b_{j}-1+a'_{j}+b'_{j}-1)\\
\nonumber &=&\underset{n}{\sum}(a_{n}+a'_{n})+\underset{n}{\sum}(b_{n}+b'_{n}) -2.
\end{eqnarray}
Analogously, $ \frac{1}{n!}\{\cdot\cdot\cdot\{ \Lambda,F\}\cdot\cdot\cdot,F\} $ and $\frac{1}{n!}\{\cdot\cdot\cdot\{ G,F\}\cdot\cdot\cdot,F\}$ have also this properties.\ Therefore,\ $ K$  has also this properties.
\end{proof}

Now our Hamiltonian is
\[
H = \Lambda + \bar{G} + K = \underset{j\geq0}{\sum}\lambda_{j}| q_{j}|^{2}+ \underset{i,j}{\sum}\bar{G}_{ij}|q_{i}|^{2}|q_{j}|^{2} + \mathcal{O}(||q||_{a,p}^{6}),
\]
where
\[
\bar{G}_{ij} = \begin{cases}
\frac{2+\delta_{ij}}{16\pi},\ i,j\in\mathbb{N}_{+},\\
\frac{1}{8\pi},\ \mbox{either one of}\  i,j=0\  \mbox{or both}.
\end{cases}
\]
Then consider the 4-order term $ \bar{G}$.\ By some simple calculations,\ one obtains
\begin{eqnarray}\nonumber
&&\frac{1}{8\pi}|q_{0}|^{4}+\frac{3}{16\pi}\underset{j\in\mathbb{N}_{+}}{\sum}|q_{j}|^{4}+ \frac{1}{8\pi}\underset{i,j\in\mathbb{N}_{+},i\neq j}{\sum}|q_{i}|^{2}|q_{j}|^{2}\\
\label{205}&=& \frac{1}{16\pi}\underset{j\in\mathbb{N}_{+}}{\sum}|q_{j}|^{4}
+\frac{1}{8\pi}\left(\underset{j\in\mathbb{N}}{\sum}|q_{j}|^{2}\right)^{2}.
\end{eqnarray}
Observe that the equation $ (\ref{eq})$ has a conservation $ \int_{\mathbb{T}} |u|^{2}d\mathrm x$,\ that is,\ $ \underset{j\in\mathbb{N}}{\sum}|q_{j}|^{2}= {C} $.\ Thus,\ (\ref{205}) becomes $ D+ \frac{C^{2}}{8\pi}$,\ where
\[
D = \frac{1}{16\pi}\underset{j\in\mathbb{N}^{*}}{\sum}|q_{j}|^{4}.
\]
Fix a positive $ n$.\ Pick a set
\[
 J=\{ j_1 < j_2 < ...< j_n\}\subseteq \mathbb{N}_{+},
\]
and take $ \xi = (\xi_{1},...,\xi_{n})\in \Pi\subset \mathbb{R}^{n}$ as parameters,\ where $ \Pi$ is a closed bounded set a positive Lebesgue measure.\ Introduce symplectic polar and real coordinates $ (x,y,z^*,\bar{z}^{*}) $ by setting
\[
\begin{cases}
q_{j_{b}}=\sqrt{\xi_{{b}}+y_{{b}}}e^{-\mathbf{i}x_{{b}}},\  b=1,...,n,\\
q_{j}=z^{*}_{j},\  j\notin J,
\end{cases}
\]
where $ z^{*} = (z_{0},z)$.\ Then we have,\ up to a constant term,\ Hamiltonian (\ref{2011}) can be rewritten as
\begin{eqnarray}
H&=&N+R\\
\nonumber &=& \underset{1\leq b\leq n}{\sum}\omega_{b}y_{b}+\underset{j\notin J}{\sum}\Omega_{j}z^{*}_{j}\bar{z}^{*}_{j}+R(x,y,z^{*},\bar{z}^{*},\xi)\\
&=& \langle \omega(\xi),y\rangle + \langle\Omega(\xi)z,\bar{z}\rangle+R(x,y,z^{*},\bar{z}^{*},\xi),
\end{eqnarray}
where $ \omega(\xi)=(\omega_{1}(\xi),...,\omega_{b}(\xi)),\ \Omega(\xi)=(\Omega_{j}(\xi))_{j\notin J} $ are given by
\begin{eqnarray}
\label{203}\omega(\xi) &=& \alpha + A\xi,\\
\label{204}\Omega(\xi)&=& \beta + B\xi
\end{eqnarray}
and $ \alpha = (\lambda_{j_{1}},...,\lambda_{j_{b}}),\ \beta =(\lambda_{j})_{j\notin J},\ A =(\bar{G}_{j_{k}j_{l}})_{1\leq k,l\leq b},\  B=(\bar{G}_{j_{k}j})_{1\leq k\leq b,j\notin J}= 0$.\ $ R$ is just $ K + \mathcal{O}(|y|^{2})+ \mathcal{O}(|y|\|z\|^{2}_{a,p})+\mathcal{O}(\|z\|^{4}_{a,p})$ with the variables $ q_{b},\bar{q}_{b},\ b=1,...,n$ expressed in terms of $ y$ and $ x$.\ In addition,\ $ R$ is analytic in $ x,y,z,\bar{z}$ in a sufficiently small neighborhood of the origin,\ and analytic in $ \xi$ lying on the closed bounded set $ \Pi$ in the sense of Whitney.

In order to apply the above theorem \ref{mainthm},\ all notations are the same as Theorem \ref{mainthm} in the following parts.\ Then the corresponding terms are
\begin{eqnarray}\label{207}
N_{0}(y,z^{*},\bar{z}^{*},\xi)&=&\langle \omega_{0},y\rangle+\langle \Omega_{00}z_{0},\bar{z}_{0}\rangle+\langle \Omega_{0}z,\bar{z}\rangle\\
\nonumber &=&\underset{j_{b}\in J}{\sum}\omega^{b}_{0}y_{b}+\Omega^{0}_{0}z_{0}\bar{z}_{0}+\underset{j\in \mathbb{N}_{+}\setminus J }{\sum}\Omega^{j}_{0}z_{j}\bar{z}_{j},
\end{eqnarray}
with $ \omega^{b}_{0} = j^{2}_{b} + \frac{1}{16\pi}\xi_{b}$ for $ b= 1,2,...n $ and $ \Omega^{0}_{0}=0,\ \Omega^{j}_{0}= j^{2}$ for $ j \in \mathbb{N}_{+}\setminus J$;
\[
R_{0}(x,y,z^{*},\bar{z}^{*},\xi)=R_{0}^{low}(x,y,z^{*},\bar{z}^{*},\xi)+R_{0}^{high}(x,y,z^{*},\bar{z}^{*},\xi),
\]
where
\begin{eqnarray}\label{208}
R_{0}^{low}(x,y,z^{*},\bar{z}^{*},\xi)
&=&\sum_{\alpha\in \mathbb{N}^{n},\beta,\gamma\in \mathbb{N}^{\mathbb{N}},2|\alpha|+|\beta|+|\gamma|\leq2}R_{0}^{\alpha\beta\gamma}(x,\xi)
y^{\alpha}\{z^{*}\}^{\beta}\{\bar{z}^{*}\}^{\gamma}\\
\nonumber &=&R_{0}^{x}(x,\xi)+\langle R_{0}^{y}(x,\xi),y\rangle+\langle R_{0}^{z}(x,\xi),z\rangle
    + \langle R_{0}^{\bar{z}}(x,\xi),\bar{z}\rangle \\
\nonumber &&+\langle R_{0}^{zz}(x,\xi)z,z\rangle+\langle R_{0}^{z\bar{z}}(x,\xi)z,\bar{z}\rangle
    +\langle R_{0}^{\bar{z}\bar{z}}(x,\xi)\bar{z},\bar{z}\rangle\\
\nonumber &&+\langle R_{0}^{z_{0}}(x,\xi),z_{0}\rangle
    + \langle R_{0}^{\bar{z}_{0}}(x,\xi),\bar{z}_{0}\rangle
    +\langle R_{0}^{z_{0}z_{0}}(x,\xi)z_{0},z_{0}\rangle\\
\nonumber &&+\langle R_{0}^{z_{0}\bar{z}_{0}}(x,\xi)z_{0},\bar{z}_{0}\rangle
    +\langle R_{0}^{\bar{z}_{0}\bar{z}_{0}}(x,\xi)\bar{z}_{0},\bar{z}_{0}\rangle
    +\langle R_{0}^{z_{0}z}(x,\xi)z_{0},z\rangle\\
\nonumber &&+ \langle R_{0}^{\bar{z}_{0}z}(x,\xi)\bar{z}_{0},z\rangle
    +\langle R_{0}^{z_{0}\bar{z}}(x,\xi)z_{0},\bar{z}\rangle
    +\langle R_{0}^{\bar{z}_{0}\bar{z}}(x,\xi)\bar{z}_{0},\bar{z}\rangle,
\end{eqnarray}
and
\begin{eqnarray}\label{209}
R_{0}^{high}(x,y,z^{*},\bar{z}^{*},\xi)&=& \sum_{\alpha\in \mathbb{N}^{n},\beta,\gamma\in \mathbb{N}^{\mathbb{N}},2|\alpha|+|\beta|+|\gamma|\geq3}
    R_{0}^{\alpha\beta\gamma}(x,\xi)y^{\alpha}\{z^{*}\}^{\beta}\{\bar{z}^{*}\}^{\gamma}.
\end{eqnarray}

After introducing action-angle coordinates for tangential variables,\ the monomials of $ R_{0}$ take the form
\begin{eqnarray}\label{210}
e^{\mathbf{i}k_{1}x_{1}+\cdot\cdot\cdot+\mathbf{i}k_{n}x_{n}}y_{1}^{m_{1}}\cdot\cdot\cdot y_{n}^{m_{n}}z^{l_{0}}_{0}\bar{z}^{l'_{0}}_{0}\underset{j\in \mathbb{N}_{+}\setminus J }{\prod}z^{l_{j}}_{j}\bar{z}^{l'_{j}}_{j},
\end{eqnarray}
with
\begin{eqnarray}\label{211}
-\underset{1\leq b\leq n}{\sum}k_{b}j_{b}+\underset{j\in \mathbb{N}_{+}\setminus J }{\sum}(l_{j}-l'_{j})j = 0.
\end{eqnarray}

\begin{lemma}\label{l3}
If $ |k|$ is even,\ the  corresponding Fourier coefficients of $R _{0}$ in (\ref{208}) satisfy:
\begin{eqnarray}
\label{218*}&&\widehat{R_{0}^{z_{0}}}(k)= \widehat{R_{0}^{\bar{z}_{0}}}(k)=0,\\
\label{219}&&\widehat{R_{0}^{z_{j}}}(k)= \widehat{R_{0}^{\bar{z}_{j}}}(k)=0,\  j\in \mathbb{N}_{+}\setminus J,\\
\label{220}&&\widehat{R_{0}^{y_{b}z_{0}}}(k)= \widehat{R_{0}^{y_{b}\bar{z}_{0}}}(k)= \widehat{R_{0}^{y_{b}z_{j}}}(k)= \widehat{R_{0}^{y_{b}\bar{z}_{j}}}(k)=0,\  1\leq b \leq n, j\in \mathbb{N}_{+}\setminus J,\\
\label{221}&&\widehat{R_{0}^{z^{l_{0}}_{0}\bar{z}^{l'_{0}}_{0}z^{l_{j}}_{j}\bar{z}^{l'_{j}}_{j}}}(k)=0,\   l_{0}+ l'_{0}+l_{j}+l'_{j}=3\  \mbox{and}\  j\in \mathbb{N}_{+}\setminus J.
\end{eqnarray}
\end{lemma}
\begin{proof}
Since
\begin{eqnarray}
\nonumber q_{j_{b}} &=&\sqrt{\xi_{{b}}+y_{{b}}}e^{-\mathbf{i}x_{{b}}}=(\xi_{{b}}+y_{{b}})^{\frac{1}{2}}e^{-\mathbf{i}x_{{b}}}= \xi^{\frac{1}{2}}_{b}\left(1+\frac{y_{b}}{\xi_{b}}\right)^{\frac{1}{2}}e^{-\mathbf{i}x_{{b}}}\\
\nonumber &=& \xi^{\frac{1}{2}}_{b}\left(1+ \frac{1}{2}\frac{y_{b}}{\xi_{b}}-\frac{1}{8}\left(\frac{y_{b}}{\xi_{b}}\right)^{2}\right)e^{-\mathbf{i}x_{{b}}}
,\  b=1,...,n,
\end{eqnarray}
and
\begin{eqnarray}
\nonumber q_{j}=z_{j},\  j\notin J,
\end{eqnarray}
then the monomials of the Hamiltonian $ P = P(q,\bar{q})= \underset{a,a'}{\sum}P_{aa'}\prod_{n=0}^{\infty}q^{a_{n}}_{n}\bar{q}^{a'_{n}}_{n}$ take the form
\[
e^{\mathbf{i}\delta_{1}x_{1}+\cdot\cdot\cdot+\mathbf{i}\delta_{n}x_{n}}y_{1}^{m_{1}}\cdot\cdot\cdot y_{n}^{m_{n}}z^{l_{0}}_{0}\bar{z}^{l'_{0}}_{0}\underset{j\in \mathbb{N}_{+}\setminus J }{\prod}z^{l_{j}}_{j}\bar{z}^{l'_{j}}_{j},
\]
where $$\begin{cases}
q_{j_{b}}= \sqrt{\xi_{{b}}+y_{{b}}}e^{\mathbf{i}\delta_{b}x_{{b}}}, \delta_{b}=-1,\\
\bar{q}_{j_{b}}= \sqrt{\xi_{{b}}+y_{{b}}}e^{\mathbf{i}\delta_{b}x_{{b}}}, \delta_{b}=1.
\end{cases}$$
The corresponding coefficients of $ P$ are as follows:
\[
\begin{split}
P^{y_{b}}(x,\xi) &=\underset{a,a'}{\sum}P_{aa'}\sum_{\underset{1\leq b\leq n}{\sum}\delta_{b}j_{b}=0}\left(\sqrt{\xi_{i}\xi_{k}}\xi_{b}e^{\mathbf{i}(\delta_{i}x_{i}+\delta_{k}x_{k})}
+\frac{1}{2}\sqrt{\xi_{i}\xi_{k}\xi_{l}/\xi_{b}}
e^{\mathbf{i}(\delta_{i}x_{i}+\delta_{k}x_{k}+\delta_{l}x_{l}-x_{b})}\right),\\
P^{z_{0}}(x,\xi) &=\underset{a,a'}{\sum}P_{aa'}\sum_{\underset{1\leq b\leq n}{\sum}\delta_{b}j_{b}=0}\sqrt{\xi_{i}\xi_{k}\xi_{l}}e^{\mathbf{i}(\delta_{i}x_{i}+\delta_{k}x_{k}+\delta_{l}x_{l})},
\widehat{P^{z_{0}}}(0,\xi) =0,\\
P^{\bar{z}_{0}}(x,\xi) &=\underset{a,a'}{\sum}P_{aa'}\sum_{\underset{1\leq b\leq n}{\sum}\delta_{b}j_{b}=0}\sqrt{\xi_{i}\xi_{k}\xi_{l}}e^{\mathbf{i}(\delta_{i}x_{i}+\delta_{k}x_{k}+\delta_{l}x_{l})},
\widehat{P^{\bar{z}_{0}}}(0,\xi)=0,\\
P^{z_{j}}(x,\xi) &=\underset{a,a'}{\sum}P_{aa'}\sum_{\underset{1\leq b\leq n}{\sum}\delta_{b}j_{b}-j=0}\sqrt{\xi_{i}\xi_{k}\xi_{l}}e^{\mathbf{i}(\delta_{i}x_{i}+\delta_{k}x_{k}+\delta_{l}x_{l})},
\widehat{P^{z_{j}}}(0,\xi) =0, j\in \mathbb{N}^{*}\setminus J\\
P^{\bar{z}_{j}}(x,\xi) &=\underset{a,a'}{\sum}P_{aa'}\sum_{\underset{1\leq b\leq n}{\sum}\delta_{b}j_{b}+j=0}\sqrt{\xi_{i}\xi_{k}\xi_{l}}e^{\mathbf{i}(\delta_{i}x_{i}+\delta_{k}x_{k}+\delta_{l}x_{l})},
\widehat{P^{\bar{z}_{j}}}(0,\xi) =0, j\in \mathbb{N}^{*}\setminus J\\
P^{y_{b}z_{0}}(x,\xi) &=\underset{a,a'}{\sum}P_{aa'}\sum_{\underset{1\leq b\leq n}{\sum}\delta_{b}j_{b}=0}\left(\sqrt{\xi_{i}}\xi_{b}e^{\mathbf{i}\delta_{i}x_{i}}
+\frac{1}{2}\sqrt{\xi_{i}\xi_{k}/\xi_{b}}e^{\mathbf{i}(\delta_{i}x_{i}+\delta_{k}x_{k}-x_{b})}\right),
\widehat{P^{z_{0}}}(0,\xi) =0,\\
P^{y_{b}\bar{z}_{0}}(x,\xi) &=\underset{a,a'}{\sum}P_{aa'}\sum_{\underset{1\leq b\leq n}{\sum}\delta_{b}j_{b}=0}\left(\sqrt{\xi_{i}}\xi_{b}e^{\mathbf{i}\delta_{i}x_{i}}
+\frac{1}{2}\sqrt{\xi_{i}\xi_{k}/\xi_{b}}e^{\mathbf{i}(\delta_{i}x_{i}+\delta_{k}x_{k}-x_{b})}\right),
\widehat{P^{\bar{z}_{0}}}(0,\xi)=0,\\
P^{y_{b}z_{j}}(x,\xi) &=\underset{a,a'}{\sum}P_{aa'}\sum_{\underset{1\leq b\leq n}{\sum}\delta_{b}j_{b}-j=0}\left(\sqrt{\xi_{i}}\xi_{b}e^{\mathbf{i}\delta_{i}x_{i}}
+\frac{1}{2}\sqrt{\xi_{i}\xi_{k}/\xi_{b}}e^{\mathbf{i}(\delta_{i}x_{i}+\delta_{k}x_{k}-x_{b})}\right),
\widehat{P^{y_{b}z_{j}}}(0,\xi) =0,\\
P^{y_{b}\bar{z}_{j}}(x,\xi) &=\underset{a,a'}{\sum}P_{aa'}\sum_{\underset{1\leq b\leq n}{\sum}\delta_{b}j_{b}+j=0}\left(\sqrt{\xi_{i}}\xi_{b}e^{\mathbf{i}\delta_{i}x_{i}}
+\frac{1}{2}\sqrt{\xi_{i}\xi_{k}/\xi_{b}}e^{\mathbf{i}(\delta_{i}x_{i}+\delta_{k}x_{k}-x_{b})}\right),
\widehat{P^{y_{b}\bar{z}_{j}}}(0,\xi) =0,
\end{split}
\]
and
\begin{eqnarray}\nonumber
P^{z^{l_{0}}_{0}\bar{z}^{l'_{0}}_{0}z^{l_{j}}_{j}\bar{z}^{l'_{j}}_{j}}(x,\xi) &=\underset{a,a'}{\sum}P_{aa'}\underset{\delta_{b}j_{b}+\underset{j\in \mathbb{N}_{+}\setminus J }{\sum}(l_{j}-l'_{j})j = 0}{\sum}\sqrt{\xi_{b}}e^{\mathbf{i}\delta_{b}x_{b}},
\widehat{P^{z^{l_{0}}_{0}\bar{z}^{l'_{0}}_{0}z^{l_{j}}_{j}\bar{z}^{l'_{j}}_{j}}}(0,\xi) =0,
\end{eqnarray}
with $ l_{0}+ l'_{0}+l_{j}+l'_{j}=3 $ and  $j\in \mathbb{N}_{+}\setminus J$.\\
Since $ R_{0}$ has the same structure with $ P$,\ (\ref{218*}), (\ref{219}), (\ref{220}) and (\ref{221}) can be obtained in the same way.
\end{proof}
In order to use the Theorem \ref{mainthm} above,\ we have to verify  that whether the values of $ \breve{R}_{0}^{z_{0}}(0,\xi)$ and $\breve{R}_{0}^{\bar{z}_{0}}(0,\xi)$ are equal to $ 0$.\ Since
\[
\breve{R}_{0}^{z_{0}}(0,\xi)=0,\breve{R}_{0}^{\bar{z}_{0}}(0,\xi)=0,
\]
then for each $ \xi \in \Pi_{\gamma}$,\ the map $ \Phi $ restricted to $ \mathbb{T}^{n} \times \{\xi \}$ is a real analytic embedding of a rotational torus with the frequencies $ \omega_{\ast}$ for the Hamiltonian $ H $ at $ \xi $;\ otherwise in the small neighborhood of initial data,\ non-invariant torus exists.

\subsection{Some computation}
In fact,\ one has
\[
\breve{R}_{0}^{z_{0}}(\xi)=\widehat{N^{z_{0}}_{\infty}}(\xi)
=\underset{m\rightarrow\infty}{\lim}\widehat{N^{z_{0}}_{m}}(\xi)
=\overset{\infty}{\underset{m=1}{\sum}}\widehat{R^{z_{0}}_{m}}(0,\xi),
\]
and $ \widehat{R^{z_{0}}_{0}}(0,\xi)=0$.\ It is obvious that we have to estimate the value of $\widehat{R^{z_{0}}_{m}}(0,\xi) $ for any $ m $.\ In the following,\ we will complete it step by step.\\
Expanding $  H_{0}\circ {X}^t_{F_{0}}\mid_{t=1} $,\ we obtain the homological equation
\begin{equation}
\{N_{0},F_{0}\}+ R_{0}^{low} = N_{1}- N_{0}.
\end{equation}
Now we get new Hamiltonian
\[
H_{1}= N_{1}+ R_{1},
\]
where
\[
N_{1}=N_{0}+\widehat{N}_{0},
\]
and
\[
R_{1}=\int^{1}_{0}\{(1-t)\widehat{N}_{0}+tR^{low}_{0},F_{0}\}\circ X^{t}_{F_{0}}\mathrm{d}t+R^{high}_{0}\circ X^{1}_{F_{0}}.
\]

First of all,\ we begin to verify $ R^{z_{0}}_{1}=0 $.\ For the convenience of notations,\ we note $ z=(z_{1},...,z_{j},...)$ for any $ j\geq 1$ while $z_{0}$ will be represented alone.\\
Denote
\[
P_{0}=\int^{1}_{0}\{(1-t)\widehat{N}_{0}+tR^{low}_{0},F_{0}\}\circ X^{t}_{F_{0}}\mathrm{d}t,
\]
and
\[
Q_{0}= R^{high}_{0}\circ X^{1}_{F_{0}}.
\]
It follows easily that
\[
\widehat{R^{z_{0}}_{1}}(0,\xi) = \widehat{P^{z_{0}}_{0}}(0,\xi) + \widehat{Q^{z_{0}}_{0}}(0,\xi).
\]

\begin{lemma}\label{l5} $ \widehat{P^{z_{0}}_{0}}(0,\xi) = 0.$
\end{lemma}
\begin{proof}
By using Taylor formula,\ one has
\begin{eqnarray}\label{230}
&&(1-t)\widehat{N}_{0}+tR^{low}_{0},F_{0}\}\circ X^{t}_{F_{0}}\\
\nonumber &=&\{(1-t)\widehat{N}_{0}+tR^{low}_{0},F_{0}\}+\{\{(1-t)\widehat{N}_{0}+tR^{low}_{0},F_{0}\},F_{0}\}t
+...\\
\nonumber &&+\frac{1}{n!}\{...\{(1-t)\widehat{N}_{0}+tR^{low}_{0},F_{0}\},F_{0}\}...,F_{0}\}{t^{n}}+....
\end{eqnarray}
Let
\[
P_{1}=\{(1-t)\widehat{N}_{0}+tR^{low}_{0},F_{0}\}
\]
where $ P_{1}$ is of the same form as $ R_{0}$:
\begin{eqnarray}\nonumber
P_{1} &=&\{(1-t)\widehat{N}_{0}+tR^{low}_{0},F_{0}\}\\
\nonumber  &=&P_{1}^{x}(x,\xi)+\langle P_{1}^{y}(x,\xi),y\rangle+\langle P_{1}^{z}(x,\xi),z\rangle
      + \langle P_{1}^{\bar{z}}(x,\xi),\bar{z}\rangle \\
\nonumber      &&+\langle P_{1}^{zz}(x,\xi)z,z\rangle+\langle P_{1}^{z\bar{z}}(x,\xi)z,\bar{z}\rangle
      +\langle P_{1}^{\bar{z}\bar{z}}(x,\xi)\bar{z},\bar{z}\rangle\\
\nonumber      &&+\langle P_{1}^{z_{0}}(x,\xi),z_{0}\rangle
    + \langle P_{1}^{\bar{z}_{0}}(x,\xi),\bar{z}_{0}\rangle
    +\langle P_{1}^{z_{0}z_{0}}(x,\xi)z_{0},z_{0}\rangle\\
\nonumber      &&+\langle P_{1}^{z_{0}\bar{z}_{0}}(x,\xi)z_{0},\bar{z}_{0}\rangle
    +\langle P_{1}^{\bar{z}_{0}\bar{z}_{0}}(x,\xi)\bar{z}_{0},\bar{z}_{0}\rangle
    +\langle P_{1}^{z_{0}z}(x,\xi)z_{0},z\rangle\\
\nonumber      &&+ \langle P_{1}^{\bar{z}_{0}z}(x,\xi)\bar{z}_{0},z\rangle
    +\langle P_{1}^{z_{0}\bar{z}}(x,\xi)z_{0},\bar{z}\rangle
    +\langle P_{1}^{\bar{z}_{0}\bar{z}}(x,\xi)\bar{z}_{0},\bar{z}\rangle,
\end{eqnarray}
and let
\[
P_{2}=\{P_{1},F_{0}\},...P_{n}=\{P_{n-1},F_{0}\},...,
\]
where $ P_{2},...,P_{n}$ has the same form as $ P_{1}$.
Thus
\[
P_{0}=\int^{1}_{0} (P_{1}+P_{2}t+...+P_{n+1}\frac{t^{n}}{n!}+...)\mathrm{d}t,
\]
and more precisely,
\begin{eqnarray}\label{222}
\widehat{R^{z_{0}}_{1}}(0,\xi)=\int^{1}_{0}( \widehat{P^{z_{0}}_{1}}(0,\xi)
+\widehat{P^{z_{0}}_{2}}(0,\xi)t+...+\widehat{P^{z_{0}}_{n+1}}(0,\xi)\frac{t^{n}}{n!}+...)\mathrm{d}t.
\end{eqnarray}
Due to
\begin{eqnarray}\nonumber
P_{1}&=& \{(1-t)\widehat{N}_{0}+tR^{low}_{0},F_{0}\}\\
\nonumber &=&(1-t)\left(\frac{\partial\widehat{N}_{0}}{\partial x}\frac{\partial F_{0}}{\partial y}-\frac{\partial\widehat{N}_{0}}{\partial y}\frac{\partial F_{0}}{\partial x}\right.\\
\nonumber &&\left.+\mathbf{i}\left(\frac{\partial\widehat{N}_{0}}{\partial z_{0}}\frac{\partial F_{0}}{\partial\bar{z}_{0}}-\frac{\partial\widehat{N}_{0}}{\partial\bar{z}_{0}}\frac{\partial F_{0}}{\partial z_{0}}+\frac{\partial\widehat{N}_{0}}{\partial z}\frac{\partial F_{0}}{\partial\bar{z}}-\frac{\partial\widehat{N}_{0}}{\partial\bar{z}}\frac{\partial F_{0}}{\partial z}\right)\right)\\
\nonumber &&+t\left(\frac{\partial R^{low}_{0}}{\partial x}\frac{\partial F_{0}}{\partial y}-\frac{\partial R^{low}_{0}}{\partial y}\frac{\partial F_{0}}{\partial x}\right.\\
\nonumber &&\left.+\mathbf{i}\left(\frac{\partial R^{low}_{0}}{\partial z_{0}}\frac{\partial F_{0}}{\partial\bar{z}_{0}}-\frac{\partial R^{low}_{0}}{\partial\bar{z}_{0}}\frac{\partial F_{0}}{\partial z_{0}}
 +\frac{\partial R^{low}_{0}}{\partial z}\frac{\partial F_{0}}{\partial\bar{z}}-\frac{\partial R^{low}_{0}}{\partial\bar{z}}\frac{\partial F_{0}}{\partial z}\right)\right),
\end{eqnarray}
by some calculations,\ we have the followings which can be divided into two categories:\\
$\textbf{Case. 1.}$  The 0-th Fourier coefficients of the following functions equal to 0.
\begin{eqnarray}\nonumber
\nonumber P^{z_{0}}_{1}(x,\xi) &=&(1-t)(-\widehat{\omega}_{0}\partial_{x}F^{z_{0}}_{0}
-\mathbf{i}\widehat{\Omega_{00}}F^{z_{0}}_{0})+t(\partial_{x}R^{z_{0}}_{0}F_{0}^{y}
-\partial_{x}F^{z_{0}}_{0}R_{0}^{y})\\
\nonumber&&+\mathbf{i}t\left(2R^{z_{0}z_{0}}_{0}F^{\bar{z}_{0}}_{0}-2R^{\bar{z}_{0}}_{0}F^{z_{0}z_{0}}_{0}
+R^{z_{0}}_{0}F^{z_{0}\bar{z}_{0}}_{0}-R^{z_{0}\bar{z}_{0}}_{0}F^{z_{0}}_{0}\right.\\
\nonumber&&\left.+R^{z_{0}z}_{0}F^{\bar{z}}_{0}-R^{\bar{z}}_{0}F^{z_{0}z}_{0}
+R^{z}_{0}F^{z_{0}\bar{z}}_{0}-R^{z_{0}\bar{z}}_{0}F^{z}_{0}\right),\\
\nonumber P^{\bar{z}_{0}}_{1}(x,\xi) &=&(1-t)(-\widehat{\omega}_{0}\partial_{x}F^{\bar{z}_{0}}_{0}
-\mathbf{i}\widehat{\Omega_{00}}F^{\bar{z}_{0}}_{0})+t(\partial_{x}R^{\bar{z}_{0}}_{0}F_{0}^{y}
-\partial_{x}F^{\bar{z}_{0}}_{0}R_{0}^{y})\\
\nonumber&&+\mathbf{i}t(2R^{z_{0}}_{0}F^{\bar{z}_{0}\bar{z}_{0}}_{0}-2R^{\bar{z}_{0}\bar{z}_{0}}_{0}F^{z_{0}}_{0}
+R^{z_{0}\bar{z}_{0}}_{0}F^{\bar{z}_{0}}_{0}-R^{\bar{z}_{0}}_{0}F^{z_{0}\bar{z}_{0}}_{0}\\
\nonumber&&+R^{\bar{z}_{0}z}_{0}F^{\bar{z}}_{0}-R^{\bar{z}}_{0}F^{\bar{z}_{0}z}_{0}
+R^{z}_{0}F^{\bar{z}_{0}\bar{z}}_{0}-R^{\bar{z}_{0}\bar{z}}_{0}F^{z}_{0}),\\
\nonumber P^{z_{j}}_{1}(x,\xi) &=&(1-t)(-\widehat{\omega}_{0}\partial_{x}F^{z_{j}}_{0}
-\mathbf{i}\widehat{\Omega^{j}_{0}}F^{z_{j}}_{0})+t(\partial_{x}R^{z_{j}}_{0}F_{0}^{y}
-\partial_{x}F^{z_{j}}_{0}R_{0}^{y})\\
\nonumber&&+\mathbf{i}t(R^{z_{0}z_{j}}_{0}F^{\bar{z}_{0}}_{0}-R^{\bar{z}_{0}}_{0}F^{z_{0}z_{j}}_{0}
+R^{z_{0}}_{0}F^{z_{j}\bar{z}_{0}}_{0}-R^{z_{j}\bar{z}_{0}}_{0}F^{z_{0}}_{0}\\
\nonumber&&+R^{z_{j}z}_{0}F^{\bar{z}}_{0}-R^{\bar{z}}_{0}F^{z_{j}z}_{0}
+R^{z}_{0}F^{z_{j}\bar{z}}_{0}-R^{z_{j}\bar{z}}_{0}F^{z}_{0}),\\
\nonumber P^{\bar{z}_{j}}_{1}(x,\xi) &=&(1-t)(-\widehat{\omega}_{0}\partial_{x}F^{\bar{z}_{j}}_{0}
-\mathbf{i}\widehat{\Omega^{j}_{0}}F^{\bar{z}_{j}}_{0})+t(\partial_{x}R^{\bar{z}_{j}}_{0}F_{0}^{y}
-\partial_{x}F^{\bar{z}_{j}}_{0}R_{0}^{y})\\
\nonumber&&+\mathbf{i}t(R^{z_{0}}_{0}F^{\bar{z}_{0}\bar{z}_{j}}_{0}-R^{\bar{z}_{0}\bar{z}_{j}}_{0}F^{z_{0}}_{0}
+R^{z_{0}\bar{z}_{j}}_{0}F^{\bar{z}_{0}}_{0}-R^{\bar{z}_{0}}_{0}F^{z_{0}\bar{z}_{j}}_{0}\\
\nonumber&&+R^{\bar{z}_{j}z}_{0}F^{\bar{z}}_{0}-R^{\bar{z}}_{0}F^{\bar{z}_{j}z}_{0}
+R^{z}_{0}F^{\bar{z}_{j}\bar{z}}_{0}-R^{\bar{z}_{j}\bar{z}}_{0}F^{z}_{0}).
\end{eqnarray}
$\textbf{Case. 2.} $ The 0-th Fourier coefficients of the following functions are uncertain.
\begin{eqnarray}\nonumber
P^{y}_{1}(x,\xi) &=&(1-t)(-\widehat{\omega}_{0}\partial_{x}F^{y}_{0}+t(\partial_{x}R^{y}_{0}F_{0}^{y}
-\partial_{x}F^{y}_{0}R_{0}^{y})),\\
\nonumber P_{1}^{z_{0}z_{0}}(x,\xi) &=&(1-t)(-\widehat{\omega}_{0}\partial_{x}F^{z_{0}z_{0}}_{0}
+2\mathbf{i}\widehat{\Omega_{00}}F^{z_{0}z_{0}}_{0})+t(\partial_{x}R^{z_{0}z_{0}}_{0}F_{0}^{y}
-\partial_{x}F^{z_{0}z_{0}}_{0}R_{0}^{y})\\
\nonumber&&+\mathbf{i}t(2R^{z_{0}z_{0}}_{0}F^{z_{0}\bar{z}_{0}}_{0}-2R^{z_{0}\bar{z}_{0}}_{0}F^{z_{0}z_{0}}_{0}
+R^{z_{0}z}_{0}F^{z_{0}\bar{z}}_{0}-R^{z_{0}\bar{z}}_{0}F^{z_{0}z}_{0}),\\
\nonumber P_{1}^{z_{0}\bar{z}_{0}}(x,\xi) &=&(1-t)(-\widehat{\omega}_{0}\partial_{x}F^{z_{0}z_{0}}_{0})
+t(\partial_{x}R^{z_{0}\bar{z}_{0}}_{0}F_{0}^{y}
-\partial_{x}F^{z_{0}\bar{z}_{0}}_{0}R_{0}^{y})\\
\nonumber&&+\mathbf{i}t(4R^{z_{0}z_{0}}_{0}F^{\bar{z}_{0}\bar{z}_{0}}_{0}
-4R^{\bar{z}_{0}\bar{z}_{0}}_{0}F^{z_{0}z_{0}}_{0}
+R^{z_{0}z}_{0}F^{\bar{z}_{0}\bar{z}}_{0}-R^{\bar{z}_{0}\bar{z}}_{0}F^{z_{0}z}_{0}
+R^{\bar{z}_{0}z}_{0}F^{z_{0}\bar{z}}_{0}-R^{z_{0}\bar{z}}_{0}F^{\bar{z}_{0}z}_{0}),\\
\nonumber P_{1}^{\bar{z}_{0}\bar{z}_{0}}(x,\xi) &=&(1-t)(-\widehat{\omega}_{0}\partial_{x}F^{\bar{z}_{0}\bar{z}_{0}}_{0}
-2\mathbf{i}\widehat{\Omega_{00}}F^{\bar{z}_{0}\bar{z}_{0}}_{0})
+t(\partial_{x}R^{\bar{z}_{0}\bar{z}_{0}}_{0}F_{0}^{y}
-\partial_{x}F^{\bar{z}_{0}\bar{z}_{0}}_{0}R_{0}^{y})\\
\nonumber&&+\mathbf{i}t(2R^{z_{0}\bar{z}_{0}}_{0}F^{\bar{z}_{0}\bar{z}_{0}}_{0}
-2R^{\bar{z}_{0}\bar{z}_{0}}_{0}F^{z_{0}\bar{z}_{0}}_{0}
+R^{\bar{z}_{0}z}_{0}F^{\bar{z}_{0}\bar{z}}_{0}
-R^{\bar{z}_{0}\bar{z}}_{0}F^{\bar{z}_{0}z}_{0}),\\
\nonumber P_{1}^{z_{i}z_{j}}(x,\xi) &=&(1-t)(-\widehat{\omega}_{0}\partial_{x}F^{z_{i}z_{j}}_{0}
+\mathbf{i}(\widehat{\Omega^{i}_{0}}+\widehat{\Omega^{j}_{0}})F^{z_{i}z_{j}}_{0})
+t(\partial_{x}R^{z_{i}z_{j}}_{0}F_{0}^{y}
-\partial_{x}F^{z_{i}z_{j}}_{0}R_{0}^{y})\\
\nonumber&&+\mathbf{i}t(R^{z_{0}z_{i}}_{0}F^{z_{j}\bar{z}_{0}}_{0}-R^{z_{j}\bar{z}_{0}}_{0}F^{z_{i}z_{0}}_{0}
+R^{z_{i}z}_{0}F^{z_{j}\bar{z}}_{0}-R^{z_{j}\bar{z}}_{0}F^{z_{i}z}_{0}),\\
\nonumber P_{1}^{z_{i}\bar{z}_{j}}(x,\xi) &=&(1-t)(-\widehat{\omega}_{0}\partial_{x}F^{z_{i}\bar{z}_{j}}_{0}
+\mathbf{i}(\widehat{\Omega^{i}_{0}}-\widehat{\Omega^{j}_{0}})F^{z_{i}\bar{z}_{j}}_{0})
+t(\partial_{x}R^{z_{i}\bar{z}_{j}}_{0}F_{0}^{y}
-\partial_{x}F^{z_{i}\bar{z}_{j}}_{0}R_{0}^{y})\\
\nonumber&&+\mathbf{i}t(R^{z_{0}z_{i}}_{0}F^{\bar{z}_{0}\bar{z}_{j}}_{0}
-R^{\bar{z}_{0}\bar{z}_{j}}_{0}F^{z_{0}z_{i}}_{0}
+R^{z_{i}z}_{0}F^{\bar{z}_{j}\bar{z}}_{0}-R^{\bar{z}_{j}\bar{z}}_{0}F^{z_{i}z}_{0}
+R^{z_{0}\bar{z}_{j}}_{0}F^{\bar{z}_{0}z_{i}}_{0}
-R^{z_{0}\bar{z}_{j}}_{0}F^{\bar{z}_{0}z_{i}}_{0})\\
\nonumber P_{1}^{\bar{z}_{i}\bar{z}_{j}}(x,\xi)
&=&(1-t)(-\widehat{\omega}_{0}\partial_{x}F^{\bar{z}_{i}\bar{z}_{j}}_{0}
+\mathbf{i}(\widehat{\Omega^{i}_{0}}+\widehat{\Omega^{j}_{0}})F^{\bar{z}_{i}\bar{z}_{j}}_{0})
+t(\partial_{x}R^{\bar{z}_{i}\bar{z}_{j}}_{0}F_{0}^{y}
-\partial_{x}F^{\bar{z}_{i}\bar{z}_{j}}_{0}R_{0}^{y})\\
\nonumber&&+\mathbf{i}t(R^{z_{0}\bar{z}_{i}}_{0}F^{\bar{z}_{0}\bar{z}_{j}}_{0}
-R^{\bar{z}_{0}\bar{z}_{j}}_{0}F^{z_{0}{z}_{i}}_{0}
+R^{\bar{z}_{i}z}_{0}F^{\bar{z}_{j}\bar{z}}_{0}
-R^{\bar{z}_{j}\bar{z}}_{0}F^{\bar{z}_{i}z}_{0}).
\end{eqnarray}
We firstly consider the term
\begin{eqnarray}\nonumber
&&\partial_{x}R^{z_{0}}_{0}F_{0}^{y}-\partial_{x}F^{z_{0}}_{0}R_{0}^{y}(x,\xi)\\ \nonumber&=&\sum^{n}_{j=1}\partial_{x_{j}}R^{z_{0}}_{0}F_{0}^{y_{j}}
-\sum^{n}_{j=1}\partial_{x_{j}}F^{z_{0}}_{0}R_{0}^{y_{j}}\\
\nonumber&=&\sum^{n}_{j=1}\left(\sum_{k\neq0}\mathbf{i}k_{j}\widehat{R^{z_{0}}_{0}}(k,\xi)e^{\mathbf{i}\langle k,x\rangle}\right)\left(\sum_{l\neq0}\widehat{F_{0}^{y_{j}}}(l,\xi)e^{\mathbf{i}\langle l,x\rangle}\right)\\
\nonumber&&-\sum^{n}_{j=1}\left(\sum_{k\neq0}\mathbf{i}k_{j}\widehat{F^{z_{0}}_{0}}(k,\xi)e^{\mathbf{i}\langle k,x\rangle}\right)\left(\sum_{l\neq0}\widehat{R_{0}^{y_{j}}}(l,\xi)e^{\mathbf{i}\langle l,x\rangle}\right)\\
\nonumber&=&\sum^{n}_{j=1}\left(\sum_{k\neq0}\mathbf{i}k_{j}\widehat{R^{z_{0}}_{0}}(k,\xi)e^{\mathbf{i}\langle k,x\rangle}\right)\left(\sum_{l\neq0}\frac{\widehat{R_{0}^{y_{j}}}(l,\xi)}{\mathbf{i}\langle l,\omega_{0}\rangle}e^{\mathbf{i}\langle l,x\rangle}\right)\\
\nonumber&&-\sum^{n}_{j=1}\left(\sum_{k\neq0}\mathbf{i}k_{j}\frac{\widehat{R^{z_{0}}_{0}}(k,\xi)}{\mathbf{i}\langle k,\omega_{0}\rangle}e^{\mathbf{i}\langle k,x\rangle}\right)\left(\sum_{l\neq0}\widehat{R_{0}^{y_{j}}}(l,\xi)e^{\mathbf{i}\langle l,x\rangle}\right)\\
\nonumber&=&\sum^{n}_{j=1}\left(\sum_{k,l\neq0}k_{j}\widehat{R^{z_{0}}_{0}}(k,\xi)\widehat{R_{0}^{y_{j}}}(l,\xi)
e^{\mathbf{i}\langle k+l,x\rangle}\left(\frac{1}{\langle l,\omega_{0}\rangle}-\frac{1}{\langle k,\omega_{0}\rangle}\right)\right).
\end{eqnarray}
It follows that
\begin{eqnarray}\label{223}
&&\widehat{(\partial_{x}R^{z_{0}}_{0}F_{0}^{y}-\partial_{x}F^{z_{0}}_{0}R_{0}^{y})}(0,\xi)\\
\nonumber &=&\sum^{n}_{j=1}\left(\underset{k,l\neq0,k+l=0}{\sum}k_{j}\widehat{R^{z_{0}}_{0}}(k,\xi)
\widehat{F_{0}^{y_{j}}}(l,\xi)
\left(\frac{1}{\langle l,\omega_{0}\rangle}-\frac{1}{\langle k,\omega_{0}\rangle}\right)\right).
\end{eqnarray}
Let
\begin{eqnarray}\nonumber
V_{1}&=&\left\{x|x_{1}=(0,...,-1,...,1,...1,0,...)^{T}, x_{2}=(0,...,-1,...,2,...0,...)^{T},\right.\\
\nonumber &&\left. x_{3}=(0,...,0,...,1,...,0,...)^{T}, x_{4}=(0,...,1,...,-1,...,-1,0,...)^{T}\right.,\\
\nonumber &&\left. x_{5}=(0,...,1,...,-2,...,0,...)^{T}, x_{6}=(0,...,0,...,-1,...,0,...)^{T}\right\},\\
\nonumber V_{2}&=&\left\{y|y_{1}=(0,...,-1,...,1,...,0,...)^{T}, y_{2}=(0,...,1,...,-1,...0,...)^{T}\right.,\\
\nonumber &&\left. y_{3}=(0,...1,...1,...0,...)^{T},
y_{4}=(0,...,-1,...,-1,...,0,...)^{T}\right.,\\
\nonumber &&\left. y_{5}(0,...,0,...,2,...,0,...)^{T}, y_{6}=(0,...,0,...,-2,...,0,...)^{T}\right\},
\end{eqnarray}
where the nonzero elements in the above n-dimension vectors is at any positions.\\
Observing the structure of $ R_{0}$,\ it is easy to verify that if and only if $ k\in V_{1}$,
\[
\widehat{R^{z_{0}}_{0}}(k,\xi)\neq0,
\]
and if and only if $ l\in V_{2}$,\
\[
\widehat{R^{y_{j}}_{0}}(l,\xi)\neq0.
\]
In order to estimate (\ref{223}),\ the equation
\begin{eqnarray}\label{224}
k+l=0, k\in V_{1},l\in V_{2}
\end{eqnarray}
should be solved.\ If (\ref{224}) has a solution,\ let $ v_{0}=(1,1,...1)^{T}$,\ then $ (k+l)^{T}v_{0}=0$.\ But \[
k\cdot v_{0}=\pm 1,
\]
for any $ k\in V_{1}$,\ and
\begin{eqnarray}\nonumber
l\cdot v_{0}=0\ \mbox{or}\  \pm2,
\end{eqnarray}
for any $ l\in V_{2}$.\ Clearly,\ (\ref{224}) is unsolved.\\
Thus
\[
\widehat{(\partial_{x}R^{z_{0}}_{0}F_{0}^{y}-\partial_{x}F^{z_{0}}_{0}R_{0}^{y})}(0,\xi)=0.
\]
Similarly,\ one has
\begin{eqnarray}
\nonumber &&\widehat{(R^{z_{0}z_{0}}_{0}F^{\bar{z}_{0}}_{0}-R^{\bar{z}_{0}}_{0}F^{z_{0}z_{0}}_{0})}(0,\xi)=0,\\
\nonumber &&\widehat{(R^{z_{0}}_{0}F^{z_{0}\bar{z}_{0}}_{0}-R^{z_{0}\bar{z}_{0}}_{0}F^{z_{0}}_{0})}(0,\xi)=0,\\
\nonumber &&\widehat{(R^{z_{0}z}_{0}F^{\bar{z}}_{0}-R^{\bar{z}}_{0}F^{z_{0}z}_{0})}(0,\xi)=0,
\end{eqnarray}
and
\[
\widehat{(R^{z}_{0}F^{z_{0}\bar{z}}_{0}-R^{z_{0}\bar{z}}_{0}F^{z}_{0})}(0,\xi)=0.
\]
Therefore,
\[
\widehat{P^{z_{0}}_{1}}(0,\xi)=0.
\]
Also,\ one has
\begin{eqnarray}
\nonumber \widehat{P^{z_{0}}_{1}}(0,\xi)=0,\\
\nonumber \widehat{P^{z_{j}}_{1}}(0,\xi)=0,\\
\nonumber \widehat{P^{\bar{z}_{j}}_{1}}(0,\xi)=0.
\end{eqnarray}
Analogously,\ the coefficients of $ P_{1}$ can also be written into the form as follows:\\
$\textbf{Case. 1.} $
\begin{eqnarray}
\nonumber P_{1}^{z_{0}}(x,\xi) &=&\sum_{k\neq0}b_{1}(\xi,t)\widehat{R_{0}^{z_{0}}}(k,\xi)e^{\mathbf{i}\langle k,x\rangle}
+\sum_{k\neq0,l\neq0}b_{2}(\xi,t)
\widehat{R_{0}^{z_{0}z_{0}}}(k,\xi)\widehat{R_{0}^{\bar{z}_{0}}}(l,\xi)e^{\mathbf{i}\langle k+l,x\rangle}\\
\nonumber&&+\sum_{k\neq0,l\neq0}b_{3}(\xi,t)
\widehat{R_{0}^{z_{0}\bar{z}_{0}}}(k,\xi)\widehat{R_{0}^{z_{0}}}(l,\xi)e^{\mathbf{i}\langle k+l,x\rangle}\\
\nonumber&&+\sum_{k,l}b_{4}(\xi,t)
\widehat{R_{0}^{\bar{z}}}(k,\xi)
\widehat{R_{0}^{z_{0}\bar{z}}}(l,\xi)e^{\mathbf{i}\langle k+l,x\rangle},\\
\nonumber P_{1}^{\bar{z}_{0}}(x,\xi) &=&\sum_{k\neq0}c_{1}(\xi,t)\widehat{R_{0}^{\bar{z}_{0}}}(k,\xi)e^{\mathbf{i}\langle k,x\rangle}
+\sum_{k\neq0,l\neq0}c_{2}(\xi,t)
\widehat{R_{0}^{z_{0}\bar{z}_{0}}}(k,\xi)\widehat{R_{0}^{\bar{z}_{0}}}(l,\xi)e^{\mathbf{i}\langle k+l,x\rangle}\\
\nonumber&&+\sum_{k\neq0,l\neq0}c_{3}(\xi,t)
\widehat{R_{0}^{\bar{z}_{0}\bar{z}_{0}}}(k,\xi)\widehat{R_{0}^{z_{0}}}(l,\xi)e^{\mathbf{i}\langle k+l,x\rangle}\\
\nonumber&&+\sum_{k,l}c_{4}(\xi,t)
\widehat{R_{0}^{\bar{z}}}(k,\xi)
\widehat{R_{0}^{\bar{z}_{0}\bar{z}}}(l,\xi)e^{\mathbf{i}\langle k+l,x\rangle},\\
\nonumber P_{1}^{z_{j}}(x,\xi) &=&\sum_{k\neq0}d_{1}(\xi,t)\widehat{R_{0}^{z_{j}}}(k,\xi)e^{\mathbf{i}\langle k,x\rangle}
+\sum_{k,l\neq0}d_{2}(\xi,t)
\widehat{R_{0}^{z_{0}z_{j}}}(k,\xi)\widehat{R_{0}^{\bar{z}_{0}}}(l,\xi)e^{\mathbf{i}\langle k+l,x\rangle}\\
\nonumber&&+\sum_{k,l\neq0}d_{3}(\xi,t)
\widehat{R_{0}^{z_{j}\bar{z}_{0}}}(k,\xi)\widehat{R_{0}^{z_{0}}}(l,\xi)e^{\mathbf{i}\langle k+l,x\rangle}\\
\nonumber&&+\sum_{k,l}d_{4}(\xi,t)
\widehat{R_{0}^{\bar{z}}}(k,\xi)
\widehat{R_{0}^{z_{j}\bar{z}}}(l,\xi)e^{\mathbf{i}\langle k+l,x\rangle},\\
\nonumber P_{1}^{\bar{z}_{j}}(x,\xi) &=&\sum_{k\neq0}e_{1}(\xi,t)\widehat{R_{0}^{\bar{z}_{j}}}(k,\xi)e^{\mathbf{i}\langle k,x\rangle}
+\sum_{k,l\neq0}e_{2}(\xi,t)
\widehat{R_{0}^{z_{0}\bar{z}_{j}}}(k,\xi)\widehat{R_{0}^{\bar{z}_{0}}}(l,\xi)e^{\mathbf{i}\langle k+l,x\rangle}\\
\nonumber&&+\sum_{k,l\neq0}b_{3}(\xi,t)
\widehat{R_{0}^{\bar{z}_{j}\bar{z}_{0}}}(k,\xi)\widehat{R_{0}^{z_{0}}}(l,\xi)e^{\mathbf{i}\langle k+l,x\rangle}\\
\nonumber&&+\sum_{k,l}b_{4}(\xi,t)
\widehat{R_{0}^{\bar{z}}}(k,\xi)
\widehat{R_{0}^{\bar{z}_{j}\bar{z}}}(l,\xi)e^{\mathbf{i}\langle k+l,x\rangle},
\end{eqnarray}
$\textbf{Case. 2.}$
\begin{eqnarray}
\nonumber P_{1}^{y_{j}}(x,\xi) &=&\sum_{k\neq0}a_{1}(\xi,t)\widehat{R_{0}^{y_{j}}}(k,\xi)e^{\mathbf{i}\langle k,x\rangle}+
\sum_{k\neq0,l\neq0}a_{2}(\xi,t)\widehat{R_{0}^{y_{j}}}(k,\xi)\widehat{R_{0}^{y_{j}}}(l,\xi)e^{\mathbf{i}\langle k+l,x\rangle},\\
\nonumber P_{1}^{z_{0}z_{0}}(x,\xi) &=&\sum_{k\neq0}f_{1}(\xi,t)\widehat{R_{0}^{z_{0}z_{0}}}(k,\xi)e^{\mathbf{i}\langle k,x\rangle}+\sum_{k\neq0,l\neq0}f_{2}(\xi,t)
\widehat{R_{0}^{z_{0}z_{0}}}(k,\xi)\widehat{R_{0}^{y_{j}}}(l,\xi)e^{\mathbf{i}\langle k+l,x\rangle}\\
\nonumber&&+\sum_{k\neq0,l\neq0}f_{3}(\xi,t)
\widehat{R_{0}^{z_{0}z_{0}}}(k,\xi)\widehat{R_{0}^{z_{0}\bar{z}_{0}}}(l,\xi)e^{\mathbf{i}\langle k+l,x\rangle}\\
\nonumber&&+\sum_{k\neq0,l\neq0}f_{4}(\xi,t)
\widehat{R_{0}^{z_{0}z}}(k,\xi)
\widehat{R_{0}^{z_{0}\bar{z}}}(l,\xi)e^{\mathbf{i}\langle k+l,x\rangle},\\
\nonumber P_{1}^{z_{0}\bar{z}_{0}}(x,\xi) &=&\sum_{k\neq0}g_{1}(\xi,t)\widehat{R_{0}^{z_{0}\bar{z}_{0}}}(k,\xi)e^{\mathbf{i}\langle k,x\rangle}+\sum_{k\neq0,l\neq0}g_{2}(\xi,t)
\widehat{R_{0}^{z_{0}\bar{z}_{0}}}(k,\xi)\widehat{R_{0}^{y_{j}}}(l,\xi)e^{\mathbf{i}\langle k+l,x\rangle}\\
\nonumber&&+\sum_{k\neq0,l\neq0}f_{3}(\xi,t)
\widehat{R_{0}^{z_{0}z_{0}}}(k,\xi)\widehat{R_{0}^{\bar{z}_{0}\bar{z}_{0}}}(l,\xi)e^{\mathbf{i}\langle k+l,x\rangle}\\
\nonumber&&+\sum_{k,l}f_{4}(\xi,t)
\widehat{R_{0}^{z_{0}z}}(k,\xi)
\widehat{R_{0}^{\bar{z}_{0}\bar{z}}}(l,\xi)e^{\mathbf{i}\langle k+l,x\rangle}\\
\nonumber&&+\sum_{k,l}f_{5}(\xi,t)
\widehat{R_{0}^{z_{0}\bar{z}}}(k,\xi)
\widehat{R_{0}^{\bar{z}_{0}z}}(l,\xi)e^{\mathbf{i}\langle k+l,x\rangle},\\
\nonumber P_{1}^{\bar{z}_{0}\bar{z}_{0}}(x,\xi) &=&\sum_{k\neq0}h_{1}(\xi,t)\widehat{R_{0}^{\bar{z}_{0}\bar{z}_{0}}}(k,\xi)e^{\mathbf{i}\langle k,x\rangle}+\sum_{k\neq0,l\neq0}h_{2}(\xi,t)
\widehat{R_{0}^{\bar{z}_{0}\bar{z}_{0}}}(k,\xi)\widehat{R_{0}^{y_{j}}}(l,\xi)e^{\mathbf{i}\langle k+l,x\rangle}\\
\nonumber&&+\sum_{k\neq0,l\neq0}h_{3}(\xi,t)
\widehat{R_{0}^{z_{0}{z}_{0}}}(k,\xi)\widehat{R_{0}^{\bar{z}_{0}\bar{z}_{0}}}(l,\xi)e^{\mathbf{i}\langle k+l,x\rangle}\\
\nonumber&&+\sum_{k,l}h_{4}(\xi,t)
\widehat{R_{0}^{\bar{z}_{0}z}}(k,\xi)
\widehat{R_{0}^{\bar{z}_{0}\bar{z}}}(l,\xi)e^{\mathbf{i}\langle k+l,x\rangle},\\
\nonumber P_{1}^{z_{i}z_{j}}(x,\xi) &=&\sum_{k}l_{1}(\xi,t)\widehat{R_{0}^{z_{i}z_{j}}}(k,\xi)e^{\mathbf{i}\langle k,x\rangle}+\sum_{k,l\neq0}l_{2}(\xi,t)
\widehat{R_{0}^{z_{i}z_{j}}}(k,\xi)\widehat{R_{0}^{y_{j}}}(l,\xi)e^{\mathbf{i}\langle k+l,x\rangle}\\
\nonumber&&+\sum_{k,l}l_{3}(\xi,t)
\widehat{R_{0}^{z_{0}z_{i}}}(k,\xi)\widehat{R_{0}^{z_{j}\bar{z}_{0}}}(l,\xi)e^{\mathbf{i}\langle k+l,x\rangle}\\
\nonumber&&+\sum_{k,l}l_{4}(\xi,t)
\widehat{R_{0}^{z_{i}z}}(k,\xi)
\widehat{R_{0}^{z_{j}\bar{z}}}(l,\xi)e^{\mathbf{i}\langle k+l,x\rangle},\\
\nonumber P_{1}^{z_{i}\bar{z}_{j}}(x,\xi) &=&\sum_{k}m_{1}(\xi,t)\widehat{R_{0}^{z_{i}\bar{z}_{j}}}(k,\xi)e^{\mathbf{i}\langle k,x\rangle}+\sum_{k,l\neq0}m_{2}(\xi,t)
\widehat{R_{0}^{z_{i}\bar{z}_{j}}}(k,\xi)\widehat{R_{0}^{y_{j}}}(l,\xi)e^{\mathbf{i}\langle k+l,x\rangle}\\
\nonumber&&+\sum_{k,l}m_{3}(\xi,t)
\widehat{R_{0}^{z_{0}z_{i}}}(k,\xi)\widehat{R_{0}^{\bar{z}_{0}\bar{z}_{j}}}(l,\xi)e^{\mathbf{i}\langle k+l,x\rangle}\\
\nonumber&&+\sum_{k,l}m_{4}(\xi,t)
\widehat{R_{0}^{z_{i}z}}(k,\xi)
\widehat{R_{0}^{\bar{z}_{j}\bar{z}}}(l,\xi)e^{\mathbf{i}\langle k+l,x\rangle}\\
\nonumber&&+\sum_{k,l}m_{5}(\xi,t)
\widehat{R_{0}^{z_{0}\bar{z}_{j}}}(k,\xi)
\widehat{R_{0}^{\bar{z}_{0}z_{i}}}(l,\xi)e^{\mathbf{i}\langle k+l,x\rangle},\\
\nonumber P_{1}^{\bar{z}_{i}\bar{z}_{j}}(x,\xi) &=&\sum_{k}n_{1}(\xi,t)\widehat{R_{0}^{\bar{z}_{i}\bar{z}_{j}}}(k,\xi)e^{\mathbf{i}\langle k,x\rangle}+\sum_{k,l}n_{2}(\xi,t)
\widehat{R_{0}^{\bar{z}_{i}\bar{z}_{j}}}(k,\xi)\widehat{R_{0}^{y_{j}}}(l,\xi)e^{\mathbf{i}\langle k+l,x\rangle}\\
\nonumber&&+\sum_{k,l}n_{3}(\xi,t)
\widehat{R_{0}^{z_{0}\bar{z}_{i}}}(k,\xi)\widehat{R_{0}^{\bar{z}_{0}\bar{z}_{j}}}(l,\xi)
e^{\mathbf{i}\langle k+l,x\rangle}\\
\nonumber&&+\sum_{k,l}n_{4}(\xi,t)
\widehat{R_{0}^{\bar{z}_{i}z}}(k,\xi)
\widehat{R_{0}^{\bar{z}_{j}\bar{z}}}(l,\xi)e^{\mathbf{i}\langle k+l,x\rangle}.
\end{eqnarray}
For the convenience of notations, we omit the specific structures of the terms $ P_{1}^{z_{0}z_{i}},P_{1}^{z_{0}\bar{z}_{i}},P_{1}^{\bar{z}_{0}z_{i}}$ and $
P_{1}^{\bar{z}_{0}\bar{z}_{0}}$,\ which has the same form with $ P_{1}^{z_{i}z_{j}}$ for any $ i,j \in \mathbb{N}_{+}\setminus J$.\\
We now consider the term
\begin{eqnarray}
\nonumber P^{z_{0}}_{2}(x,\xi) &=&\partial_{x}P^{z_{0}}_{1}F_{0}^{y}
-P_{1}^{y}\partial_{x}F^{z_{0}}_{0}
+\mathbf{i}(2P^{z_{0}z_{0}}_{1}F^{\bar{z}_{0}}_{0}-2P^{\bar{z}_{0}}_{1}F^{z_{0}z_{0}}_{0}\\
\nonumber&&+P^{z_{0}}_{1}F^{z_{0}\bar{z}_{0}}_{0}-P^{z_{0}\bar{z}_{0}}_{1}F^{z_{0}}_{0}
+P^{z_{0}z}_{1}F^{{z}}_{0}-P^{\bar{z}}_{1}F^{z_{0}z}_{0}
+P^{z}_{1}F^{z_{0}{z}}_{0}-P^{z_{0}\bar{z}}_{1}F^{z}_{0}).
\end{eqnarray}
Since
\begin{eqnarray}
\nonumber&&\partial_{x}P^{z_{0}}_{1}F_{0}^{y}-P_{1}^{y}\partial_{x}F^{z_{0}}_{0}\\
\nonumber&=&\sum^{n}_{j=1}\partial_{x_{j}}P^{z_{0}}_{1}F_{0}^{y_{j}}
-\sum^{n}_{j=1}\partial_{x_{j}}F^{z_{0}}_{0}P_{0}^{y_{j}}\\
\nonumber&=&\sum^{n}_{j=1}\left(\sum_{k\neq0}b_{1}(\xi,t)\mathbf{i}k_{j}\widehat{R_{0}^{z_{0}}}(k,\xi)e^{\mathbf{i}\langle k,x\rangle}+\sum_{k\neq0,l\neq0}b_{2}(\xi,t)\mathbf{i}(k_{j}+l_{j})
\widehat{R_{0}^{\bar{z}_{0}}}(k,\xi)\widehat{R_{0}^{z_{0}z_{0}}}(l,\xi)e^{\mathbf{i}\langle k+l,x\rangle}\right.\\
\nonumber&&\left.+\sum_{k\neq0,l\neq0}b_{3}(\xi,t)\mathbf{i}(k_{j}+l_{j})
\widehat{R_{0}^{z_{0}}}(l,\xi)\widehat{R_{0}^{z_{0}\bar{z}_{0}}}(l,\xi)e^{\mathbf{i}\langle k+l,x\rangle}\right.\\
\nonumber&&\left.+\sum_{k,l}b_{4}(\xi,t)\mathbf{i}(k_{j}+l_{j})
\widehat{R_{0}^{\bar{z}}}(k,\xi)
\widehat{R_{0}^{z_{0}\bar{z}}}(l,\xi)e^{\mathbf{i}\langle k+l,x\rangle}\right)
\left(\sum_{m\neq0}\widehat{F_{0}^{y_{j}}}(m,\xi)e^{\mathbf{i}\langle m,x\rangle}\right)\\
\nonumber&&-\sum^{n}_{j=1}\left(\sum_{k\neq0}\mathbf{i}k_{j}\widehat{F^{z_{0}}_{0}}(k,\xi)e^{\mathbf{i}\langle k,x\rangle}\right)
\left(\sum_{l\neq0}a_{1}(\xi,t)\widehat{R_{0}^{y_{j}}}(l,\xi)e^{\mathbf{i}\langle l,x\rangle}\right.\\
\nonumber&&\left.+\sum_{l\neq0,m\neq0}a_{2}(\xi,t)\widehat{R_{0}^{y_{j}}}(l,\xi)\widehat{R_{0}^{y_{j}}}(m,\xi)
e^{\mathbf{i}\langle l+m,x\rangle}\right),
\end{eqnarray}
One then obtains
\begin{eqnarray}
\nonumber&&\partial_{x}P^{z_{0}}_{1}F_{0}^{y}-P_{1}^{y}\partial_{x}F^{z_{0}}_{0}\\
\nonumber&=&\sum^{n}_{j=1}
\left(\sum_{k\neq0,m\neq0}p_{1}(\xi,t)\widehat{R_{0}^{z_{0}}}(k,\xi)\widehat{R_{0}^{y_{j}}}(m,\xi)e^{\mathbf{i}\langle k+m,x\rangle}\right.\\
\nonumber&&\left.+\sum_{k\neq0,l\neq0,m\neq0}p_{2}(\xi,t)
\widehat{R_{0}^{\bar{z}_{0}}}(k,\xi)\widehat{R_{0}^{z_{0}z_{0}}}(l,\xi)\widehat{R_{0}^{y_{j}}}(m,\xi)
e^{\mathbf{i}\langle k+l+m,x\rangle}\right.\\
\nonumber&&\left.+\sum_{k\neq0,l\neq0,m\neq0}p_{3}(\xi,t)
\widehat{R_{0}^{z_{0}}}(k,\xi)\widehat{R_{0}^{z_{0}\bar{z}_{0}}}(l,\xi)\widehat{R_{0}^{y_{j}}}(m,\xi)
e^{\mathbf{i}\langle k+l+m,x\rangle}\right.\\
\nonumber&&\left.+\sum_{k,l, m\neq0}p_{4}(\xi,t)
\widehat{R_{0}^{\bar{z}}}(k,\xi)
\widehat{R_{0}^{z_{0}\bar{z}}}(l,\xi)\widehat{R_{0}^{y_{j}}}(m,\xi)e^{\mathbf{i}\langle k+l+m,x\rangle}\right)\\
\nonumber&&-\sum^{n}_{j=1}\left(\sum_{k\neq0,m\neq0}q_{1}(\xi,t)\widehat{R_{0}^{z_{0}}}(k,\xi)\widehat{R_{0}^{y_{j}}}(m,\xi)
e^{\mathbf{i}\langle k+m,x\rangle}\right.\\
\nonumber&&\left.+\sum_{k\neq0,l\neq0,m\neq0}q_{2}(\xi,t)\widehat{R_{0}^{z_{0}}}(k,\xi)\widehat{R_{0}^{y_{j}}}(l,\xi)
\widehat{R_{0}^{y_{j}}}(m,\xi)
e^{\mathbf{i}\langle k+l+m,x\rangle}\right).
\end{eqnarray}
First of all,\ it is easy to verify that
\[
\widehat{R_{0}^{z_{0}}}(k,\xi)\widehat{R_{0}^{y_{j}}}(m,\xi)=0,
\]
when $k,m$ satisfies the equation
\[
k+m=0,k\in V_{1},m\in V_{2}.
\]
Similarly,\ one has
\begin{eqnarray}
\nonumber&&\widehat{R_{0}^{\bar{z}_{0}}}(k,\xi)\widehat{R_{0}^{z_{0}z_{0}}}(l,\xi)\widehat{R_{0}^{y_{j}}}(m,\xi)=0,\\
\nonumber&&\widehat{R_{0}^{z_{0}}}(k,\xi)\widehat{R_{0}^{z_{0}\bar{z}_{0}}}(l,\xi)\widehat{R_{0}^{y_{j}}}(m,\xi)=0,\\
\nonumber&&\widehat{R_{0}^{\bar{z}}}(k,\xi)\widehat{R_{0}^{z_{0}\bar{z}}}(l,\xi)\widehat{R_{0}^{y_{j}}}(m,\xi)=0,\\
\nonumber&&\widehat{R_{0}^{z_{0}}}(k,\xi)\widehat{R_{0}^{y_{j}}}(l,\xi)\widehat{R_{0}^{y_{j}}}(m,\xi)=0,
\end{eqnarray}
when
$ k,l,m$ satisfies the equation
\[
k+l+m=0,k\in V_{1},l,m\in V_{2}.
\]
Hence
\[
\widehat{P^{z_{0}}_{2}}(0,\xi)=0.
\]
For the same process of computing  $\widehat{P_{n+1}^{z_{0}}}(0,\xi) $.\ That is,\
\begin{eqnarray}
\nonumber P^{z_{0}}_{n+1}(x,\xi) &=&\partial_{x}P^{z_{0}}_{n}F_{0}^{y}
-P_{n}^{y}\partial_{x}F^{z_{0}}_{0}
+\mathbf{i}(2P^{z_{0}z_{0}}_{n}F^{\bar{z}_{0}}_{0}-2P^{\bar{z}_{0}}_{n}F^{z_{0}z_{0}}_{0}\\
\nonumber&&+P^{z_{0}}_{n}F^{z_{0}\bar{z}_{0}}_{0}-P^{z_{0}\bar{z}_{0}}_{n}F^{z_{0}}_{0}
+P^{z_{0}z}_{n}F^{\bar{z}}_{0}-P^{\bar{z}}_{n}F^{z_{0}z}_{0}
+P^{z}_{n}F^{z_{0}\bar{z}}_{0}-P^{z_{0}\bar{z}}_{n}F^{z}_{0}),
\end{eqnarray}
and we learn that when $ |k|$ is even,
\[
S_{1}=\left\{\widehat{R_{0}^{z_{0}}}(k,\xi),\widehat{R_{0}^{\bar{z}_{0}}}(k,\xi)
\widehat{R_{0}^{z}}(k,\xi)\widehat{R_{0}^{\bar{z}}}(k,\xi)\right\}
\]
are all equal to $0$,\ and when $ |k|$ is odd,
\begin{eqnarray}\nonumber
S_{2}= &&\left\{\widehat{R_{0}^{y_{j}}}(k,\xi),\widehat{R_{0}^{z_{0}z_{0}}}(k,\xi),
\widehat{R_{0}^{z_{0}\bar{z}_{0}}}(k,\xi),\widehat{R_{0}^{\bar{z}_{0}\bar{z}_{0}}}(k,\xi),
\widehat{R_{0}^{z_{i}z_{j}}}(k,\xi),\widehat{R_{0}^{z_{i}\bar{z}_{j}}}(k,\xi),\right.\\
\nonumber&&\left.\widehat{R_{0}^{\bar{z}_{i}\bar{z}_{j}}}(k,\xi),\widehat{R_{0}^{z_{0}z_{j}}}(k,\xi),
\widehat{R_{0}^{z_{0}\bar{z}_{j}}}(k,\xi),
\widehat{R_{0}^{\bar{z}_{0}z_{j}}}(k,\xi),
\widehat{R_{0}^{\bar{z}_{0}\bar{z}_{j}}}(k,\xi)\right\}
\end{eqnarray}
are all equal to $0$.\\
Let $\widehat{R_{0}^{z_{0}}}(k,\xi)$ represent all the elements of $ S_{1}$ and
$\widehat{R_{0}^{y_{j}}}(k,\xi)$ represent all the elements of $ S_{2}$.\\
Then the coefficient of $ P_{n}$ can also be written into the form as follows:\\
$\textbf{Case. 1.}$
\begin{eqnarray}
\nonumber P_{n}^{z_{0}}(x,\xi) &=&\sum_{k\neq0}s_{1}(\xi,t)\widehat{R_{0}^{z_{0}}}(k,\xi)e^{\mathbf{i}\langle k,x\rangle}
+\sum_{k,l\ \mbox{are not all zero}}s_{2}(\xi,t)\widehat{R_{0}^{z_{0}}}(k,\xi)
\widehat{R_{0}^{y_{j}}}(l,\xi)e^{\mathbf{i}\langle k+l,x\rangle}+...\\
\nonumber&&+\sum_{k,l,...,m\ \mbox{are not all zero}}s_{n+1}(\xi,t)\widehat{R_{0}^{z_{0}}}(k,\xi)
\widehat{R_{0}^{y_{j}}}(l,\xi)\cdot\cdot\cdot\widehat{R_{0}^{y_{j}}}(m,\xi) e^{\mathbf{i}\langle k+l+...+m,x\rangle},\\
\nonumber P_{n}^{\bar{z}_{0}}(x,\xi)
&=&\sum_{k\neq0}t_{1}(\xi,t)\widehat{R_{0}^{z_{0}}}(k,\xi)e^{\mathbf{i}\langle k,x\rangle}
+\sum_{k,l\ \mbox{are not all zero}}t_{2}(\xi,t)\widehat{R_{0}^{z_{0}}}(k,\xi)
\widehat{R_{0}^{y_{j}}}(l,\xi)e^{\mathbf{i}\langle k+l,x\rangle}+...\\
\nonumber&&+\sum_{k,l,...,m\ \mbox{are not all zero}}t_{n+1}(\xi,t)\widehat{R_{0}^{z_{0}}}(k,\xi)
\widehat{R_{0}^{y_{j}}}(l,\xi)\cdot\cdot\cdot\widehat{R_{0}^{y_{j}}}(m,\xi) e^{\mathbf{i}\langle k+l+...+m,x\rangle},\\
\nonumber P_{n}^{z_{j}}(x,\xi)&=&\sum_{k\neq0}u_{1}(\xi,t)\widehat{R_{0}^{z_{0}}}(k,\xi)e^{\mathbf{i}\langle k,x\rangle}
+\sum_{k,l\ \mbox{are not all zero}}u_{2}(\xi,t)\widehat{R_{0}^{z_{0}}}(k,\xi)
\widehat{R_{0}^{y_{j}}}(l,\xi)e^{\mathbf{i}\langle k+l,x\rangle}+...\\
\nonumber&&+\sum_{k,l,...,m\ \mbox{are not all zero}}u_{n+1}(\xi,t)\widehat{R_{0}^{z_{0}}}(k,\xi)
\widehat{R_{0}^{y_{j}}}(l,\xi)\cdot\cdot\cdot\widehat{R_{0}^{y_{j}}}(m,\xi) e^{\mathbf{i}\langle k+l+...+m,x\rangle},\\
\nonumber P_{n}^{\bar{z}_{j}}(x,\xi)
&=&\sum_{k\neq0}v_{1}(\xi,t)\widehat{R_{0}^{z_{0}}}(k,\xi)e^{\mathbf{i}\langle k,x\rangle}
+\sum_{k,l\ \mbox{are not all zero}}v_{2}(\xi,t)\widehat{R_{0}^{z_{0}}}(k,\xi)
\widehat{R_{0}^{y_{j}}}(l,\xi)e^{\mathbf{i}\langle k+l,x\rangle}+...\\
\nonumber&&+\sum_{k,l,...,m\ \mbox{are not all zero}}v_{n+1}(\xi,t)\widehat{R_{0}^{z_{0}}}(k,\xi)
\widehat{R_{0}^{y_{j}}}(l,\xi)\cdot\cdot\cdot\widehat{R_{0}^{y_{j}}}(m,\xi) e^{\mathbf{i}\langle k+l+...+m,x\rangle},
\end{eqnarray}
$\textbf{Case. 2.}$
\begin{eqnarray}
\nonumber P_{n}^{y_{j}}(x,\xi) &=&\sum_{k\neq0}r_{1}(\xi,t)\widehat{R_{0}^{y_{j}}}(k,\xi)e^{\mathbf{i}\langle k,x\rangle}+
\sum_{k,l\ \mbox{are not all zero}}r_{2}(\xi,t)\widehat{R_{0}^{y_{j}}}(k,\xi)\widehat{R_{0}^{y_{j}}}(l,\xi)e^{\mathbf{i}\langle k+l,x\rangle}+...\\
\nonumber&&+\sum_{k,l,...,m\ \mbox{are not all zero}}a_{n+1}(\xi,t)\widehat{R_{0}^{y_{j}}}(k,\xi)
\widehat{R_{0}^{y_{j}}}(l,\xi)\cdot\cdot\cdot\widehat{R_{0}^{y_{j}}}(m,\xi) e^{\mathbf{i}\langle k+l+...+m,x\rangle},\\
\nonumber P_{n}^{z_{0}z_{0}}(x,\xi) &=&\sum_{k\neq0}w_{1}(\xi,t)\widehat{R_{0}^{y_{j}}}(k,\xi)e^{\mathbf{i}\langle k,x\rangle}+
\sum_{k,l\ \mbox{are not all zero}}w_{2}(\xi,t)\widehat{R_{0}^{y_{j}}}(k,\xi)\widehat{R_{0}^{y_{j}}}(l,\xi)e^{\mathbf{i}\langle k+l,x\rangle}+...\\
\nonumber&&+\sum_{k,l,...,m\ \mbox{are not all zero}}w_{n+1}(\xi,t)\widehat{R_{0}^{y_{j}}}(k,\xi)
\widehat{R_{0}^{y_{j}}}(l,\xi)\cdot\cdot\cdot\widehat{R_{0}^{y_{j}}}(m,\xi) e^{\mathbf{i}\langle k+l+...+m,x\rangle},\\
\nonumber P_{n}^{z_{0}\bar{z}_{0}}(x,\xi) &=&\sum_{k\neq0}\alpha_{1}(\xi,t)\widehat{R_{0}^{y_{j}}}(k,\xi)e^{\mathbf{i}\langle k,x\rangle}+
\sum_{k,l\ \mbox{are not all zero}}\alpha_{2}(\xi,t)\widehat{R_{0}^{y_{j}}}(k,\xi)\widehat{R_{0}^{y_{j}}}(l,\xi)e^{\mathbf{i}\langle k+l,x\rangle}+...\\
\nonumber&&+\sum_{k,l,...,m\ \mbox{are not all zero}}\alpha_{n+1}(\xi,t)\widehat{R_{0}^{y_{j}}}(k,\xi)
\widehat{R_{0}^{y_{j}}}(l,\xi)\cdot\cdot\cdot\widehat{R_{0}^{y_{j}}}(m,\xi) e^{\mathbf{i}\langle k+l+...+m,x\rangle},\\
\nonumber P_{n}^{\bar{z}_{0}\bar{z}_{0}}(x,\xi) &=&\sum_{k\neq0}\beta_{1}(\xi,t)\widehat{R_{0}^{y_{j}}}(k,\xi)e^{\mathbf{i}\langle k,x\rangle}+
\sum_{k,l\ \mbox{are not all zero}}\beta_{2}(\xi,t)\widehat{R_{0}^{y_{j}}}(k,\xi)\widehat{R_{0}^{y_{j}}}(l,\xi)e^{\mathbf{i}\langle k+l,x\rangle}+...\\
\nonumber&&+\sum_{k,l,...,m\ \mbox{are not all zero}}\beta_{n+1}(\xi,t)\widehat{R_{0}^{y_{j}}}(k,\xi)
\widehat{R_{0}^{y_{j}}}(l,\xi)\cdot\cdot\cdot\widehat{R_{0}^{y_{j}}}(m,\xi) e^{\mathbf{i}\langle k+l+...+m,x\rangle},\\
\nonumber P_{n}^{z_{i}z_{j}}(x,\xi)  &=&\sum_{k\neq0}\gamma_{1}(\xi,t)\widehat{R_{0}^{y_{j}}}(k,\xi)e^{\mathbf{i}\langle k,x\rangle}+
\sum_{k,l\ \mbox{are not all zero}}\gamma_{2}(\xi,t)\widehat{R_{0}^{y_{j}}}(k,\xi)\widehat{R_{0}^{y_{j}}}(l,\xi)e^{\mathbf{i}\langle k+l,x\rangle}+...\\
\nonumber&&+\sum_{k,l,...,m\ \mbox{are not all zero}}\gamma_{n+1}(\xi,t)\widehat{R_{0}^{y_{j}}}(k,\xi)
\widehat{R_{0}^{y_{j}}}(l,\xi)\cdot\cdot\cdot\widehat{R_{0}^{y_{j}}}(m,\xi) e^{\mathbf{i}\langle k+l+...+m,x\rangle},\\
\nonumber P_{n}^{z_{i}\bar{z}_{j}}(x,\xi)  &=&\sum_{k\neq0}\delta_{1}(\xi,t)\widehat{R_{0}^{y_{j}}}(k,\xi)e^{\mathbf{i}\langle k,x\rangle}+
\sum_{k,l\ \mbox{are not all zero}}\delta_{2}(\xi,t)\widehat{R_{0}^{y_{j}}}(k,\xi)\widehat{R_{0}^{y_{j}}}(l,\xi)e^{\mathbf{i}\langle k+l,x\rangle}+...\\
\nonumber&&+\sum_{k,l,...,m\ \mbox{are not all zero}}\delta_{n+1}(\xi,t)\widehat{R_{0}^{y_{j}}}(k,\xi)
\widehat{R_{0}^{y_{j}}}(l,\xi)\cdot\cdot\cdot\widehat{R_{0}^{y_{j}}}(m,\xi) e^{\mathbf{i}\langle k+l+...+m,x\rangle},\\
\nonumber P_{n}^{\bar{z}_{i}\bar{z}_{j}}(x,\xi)  &=&\sum_{k\neq0}\sigma_{1}(\xi,t)\widehat{R_{0}^{y_{j}}}(k,\xi)e^{\mathbf{i}\langle k,x\rangle}+
\sum_{k,l\ \mbox{are not all zero}}\sigma_{2}(\xi,t)\widehat{R_{0}^{y_{j}}}(k,\xi)\widehat{R_{0}^{y_{j}}}(l,\xi)e^{\mathbf{i}\langle k+l,x\rangle}+...\\
\nonumber&&+\sum_{k,l,...,m\ \mbox{are not all zero}}\sigma_{n+1}(\xi,t)\widehat{R_{0}^{y_{j}}}(k,\xi)
\widehat{R_{0}^{y_{j}}}(l,\xi)\cdot\cdot\cdot\widehat{R_{0}^{y_{j}}}(m,\xi) e^{\mathbf{i}\langle k+l+...+m,x\rangle}.
\end{eqnarray}
For the convenience of notations, we omit the specific structures of the terms $ P_{n}^{z_{0}z_{i}},P_{n}^{z_{0}\bar{z}_{i}},P_{n}^{\bar{z}_{0}z_{i}}$ and $
P_{n}^{\bar{z}_{0}\bar{z}_{0}}$,\ which has the same form with $ P_{n}^{z_{i}z_{j}}$ for any $ i,j \in \mathbb{N}_{+}\setminus J$.\\
We now consider the term
\begin{eqnarray}\nonumber
P^{z_{0}}_{n+1}(x,\xi) &=&\partial_{x}P^{z_{0}}_{n}F_{0}^{y}
-P_{n}^{y}\partial_{x}F^{z_{0}}_{0}
+\mathbf{i}(2P^{z_{0}z_{0}}_{n}F^{\bar{z}_{0}}_{0}-2P^{\bar{z}_{0}}_{n}F^{z_{0}z_{0}}_{0}\\
\nonumber&&+P^{z_{0}}_{n}F^{z_{0}\bar{z}_{0}}_{0}-P^{z_{0}\bar{z}_{0}}_{n}F^{z_{0}}_{0}
+P^{z_{0}z}_{n}F^{\bar{z}}_{0}-P^{\bar{z}}_{n}F^{z_{0}z}_{0}
+P^{z}_{n}F^{z_{0}\bar{z}}_{0}-P^{z_{0}\bar{z}}_{n}F^{z}_{0}).
\end{eqnarray}
Consider the term
\begin{eqnarray}
\nonumber
\partial_{x}P^{z_{0}}_{n}F_{0}^{y}-P_{n}^{y}\partial_{x}F^{z_{0}}_{0}&=&\sum^{n}_{j=1}\partial_{x_{j}}P^{z_{0}}_{n}F_{0}^{y_{j}}
-\sum^{n}_{j=1}\partial_{x_{j}}F^{z_{0}}_{0}P_{n}^{y_{j}},
\end{eqnarray}
where
\begin{eqnarray}
\nonumber &&\partial_{x_{j}}P^{z_{0}}_{n}F_{0}^{y_{j}}\\
\nonumber&=&\left(\sum_{k\neq0}s_{1}(\xi,t)\mathbf{i}k_{j}\widehat{R_{0}^{z_{0}}}(k,\xi)e^{\mathbf{i}\langle k,x\rangle}+\sum_{k,l\ \mbox{are not all zero}}s_{2}(\xi,t)\mathbf{i}(k_{j}+l_{j})\widehat{R_{0}^{z_{0}}}(k,\xi)
\widehat{R_{0}^{y_{j}}}(l,\xi)e^{\mathbf{i}\langle k+l,x\rangle}\right.\\
\nonumber&&\left.+...+\sum_{k,l,...,m\ \mbox{are not all zero}}s_{n+1}(\xi,t)\mathbf{i}(k_{j}+l_{j}+...+m_{j})\widehat{R_{0}^{z_{0}}}(k,\xi)
\widehat{R_{0}^{y_{j}}}(l,\xi)\cdot\cdot\cdot\widehat{R_{0}^{y_{j}}}(m,\xi)\right.\\
\nonumber&&\left.\times e^{\mathbf{i}\langle k+l+...+m,x\rangle}\right)\left(\sum_{n\neq0}\widehat{F_{0}^{y_{j}}}(n,\xi)e^{\mathbf{i}\langle n,x\rangle}\right)
\end{eqnarray}
and
\begin{eqnarray}
\nonumber&& \partial_{x_{j}}F^{z_{0}}_{0}P_{n}^{y_{j}}\\
\nonumber&&\left(\sum_{k\neq0}\mathbf{i}k_{j}\widehat{F_{0}^{z_{0}}}(k,\xi)e^{\mathbf{i}\langle k,x\rangle}\right)\left(\sum_{l\neq0}r_{1}(\xi,t)\widehat{R_{0}^{y_{j}}}(l,\xi)e^{\mathbf{i}\langle l,x\rangle}\right.\\
\nonumber&&\left.+\sum_{l,m\ \mbox{are not all zero}}r_{2}(\xi,t)\widehat{R_{0}^{y_{j}}}(l,\xi)
\widehat{R_{0}^{y_{j}}}(m,\xi)e^{\mathbf{i}\langle l+m,x\rangle}+...\right.\\
\nonumber&&\left.+\sum_{l,...,m,n\ \mbox{are not all zero}}r_{n+1}(\xi,t)\widehat{R_{0}^{y_{j}}}(l,\xi)\cdot\cdot\cdot
\widehat{R_{0}^{y_{j}}}(m,\xi)\widehat{R_{0}^{y_{j}}}(n,\xi)
 e^{\mathbf{i}\langle l+...+m+n,x\rangle}\right).
\end{eqnarray}
We thus have
\begin{eqnarray}
\nonumber&&\partial_{x}P^{z_{0}}_{n}F_{0}^{y}-P_{n}^{y}\partial_{x}F^{z_{0}}_{0}\\
\nonumber&=&\sum^{n}_{j=1}\left(\sum_{k,n\ \mbox{are not all zero}}\rho_{1}(\xi,t)\widehat{R_{0}^{z_{0}}}(k,\xi)
\widehat{R_{0}^{y_{j}}}(l,\xi)e^{\mathbf{i}\langle k+n,x\rangle}\right.\\
\nonumber&&\left.+\sum_{k,l,n\ \mbox{are not all zero}}\rho_{2}(\xi,t)\widehat{R_{0}^{z_{0}}}(k,\xi)
\widehat{R_{0}^{y_{j}}}(l,\xi)\widehat{R_{0}^{y_{j}}}(n,\xi)e^{\mathbf{i}\langle k+l+n,x\rangle}+...\right.\\
\nonumber&&\left.+\sum_{k,l,...,m,n\ \mbox{are not all zero}}\rho_{n+1}(\xi,t)\widehat{R_{0}^{z_{0}}}(k,\xi)
\widehat{R_{0}^{y_{j}}}(l,\xi)\cdot\cdot\cdot\widehat{R_{0}^{y_{j}}}(m,\xi)\widehat{R_{0}^{y_{j}}}(n,\xi) e^{\mathbf{i}\langle k+l+...+m+n,x\rangle}\right)\\
\nonumber&&-\left(\sum_{k,l\ \mbox{are not all zero}}\kappa_{1}(\xi,t)\widehat{R_{0}^{z_{0}}}(k,\xi)\widehat{R_{0}^{y_{j}}}(l,\xi)e^{\mathbf{i}\langle k+l,x\rangle}\right.\\
\nonumber&&\left.+\sum_{k,l,m\ \mbox{are not all zero}}\kappa_{2}(\xi,t)\widehat{R_{0}^{z_{0}}}(k,\xi)\widehat{R_{0}^{y_{j}}}(l,\xi)
\widehat{R_{0}^{y_{j}}}(m,\xi)e^{\mathbf{i}\langle k+l+m,x\rangle}+...\right.\\
\nonumber&&\left.+\sum_{k,l,...,m,n\ \mbox{are not all zero}}\kappa_{n+1}(\xi,t)\widehat{R_{0}^{z_{0}}}(k,\xi)
\widehat{R_{0}^{y_{j}}}(l,\xi)\cdot\cdot\cdot
\widehat{R_{0}^{y_{j}}}(m,\xi)\widehat{R_{0}^{y_{j}}}(n,\xi) e^{\mathbf{i}\langle k+l+...+m+n,x\rangle}\right).
\end{eqnarray}
Due to
\[
\widehat{R_{0}^{z_{0}}}(k,\xi)\widehat{R_{0}^{y_{j}}}(l,\xi)
\cdot\cdot\cdot\widehat{R_{0}^{y_{j}}}(m,\xi)\widehat{R_{0}^{y_{j}}}(n,\xi)=0,
\]
when all the $ n+1$ vectors $ k,l,...m,n $ satisfies the equation
\[
k+l+...+m+n=0,k\in V_{1},l,...,m,n\in V_{2},
\]
one finally obtains
\[
\widehat{P^{z_{0}}_{n+1}}(0,\xi)=0.
\]
Above of all proofs,\ we have
\[
\widehat{P^{z_{0}}_{0}}(0,\xi)=0.
\]
\end{proof}
We next want to compute the value $ Q^{z_{0}}_{0}$.
\begin{lemma}\label{l6} $ \widehat{Q^{z_{0}}_{0}}(0,\xi) = 0.$
\end{lemma}
\begin{proof}
From Lemma \ref{l3},\ $ R^{high}_{0}$ has the same form as $ P^{high}$.\ Due to
\begin{eqnarray}\nonumber
R^{high}_{0}\circ X^{1}_{F_{0}}&=&R^{high}_{0}+\{R^{high}_{0},F_{0}\}+\frac{1}{2!}\{\{R^{high}_{0},F_{0}\},F_{0}\}\\
\nonumber&&+...+\frac{1}{n!}\{...\{\{R^{high}_{0},F_{0}\},F_{0}\}...,F_{0}\}+....
\end{eqnarray}
Denote
\[
Q_{0}=R^{high}_{0}\circ X^{1}_{F_{0}},
\]
and
\[
Q_{1}=\{R^{high}_{0},F_{0}\}, Q_{2}= \{Q_{1},F_{0}\},..., Q_{n}=\{Q_{n-1},F_{0}\},...,
\]
one then has
\[
Q_{0}= R^{high}+ Q_{1}+\frac{1}{2!}\cdot Q_{2}+...+ \frac{1}{n!}\cdot Q_{n}+....
\]
More precisely,\ let
\[
Q_{1}=\{R^{high}_{0},F_{0}\} = Q^{low}_{1}+Q^{high}_{1},
\]
where $Q^{low}_{1}$ has the form
\begin{eqnarray}
\nonumber Q_{1}^{low} &=& Q_{1}^{x}(x,\xi)+\langle Q_{1}^{y}(x,\xi),y\rangle+\langle Q_{1}^{z}(x,\xi),z\rangle+\langle Q_{1}^{\bar{z}}(x,\xi),\bar{z}\rangle \\
\nonumber &&+\langle Q_{1}^{zz}(x,\xi)z,z\rangle+\langle Q_{1}^{z\bar{z}}(x,\xi)z,\bar{z}\rangle
      +\langle Q_{1}^{\bar{z}\bar{z}}(x,\xi)\bar{z},\bar{z}\rangle\\
\nonumber &&+\langle Q_{1}^{z_{0}}(x,\xi),z_{0}\rangle
    + \langle Q_{1}^{\bar{z}_{0}}(x,\xi),\bar{z}_{0}\rangle
    +\langle Q_{1}^{z_{0}z_{0}}(x,\xi)z_{0},z_{0}\rangle\\
\nonumber &&+\langle Q_{1}^{z_{0}\bar{z}_{0}}(x,\xi)z_{0},\bar{z}_{0}\rangle
    +\langle Q_{1}^{\bar{z}_{0}\bar{z}_{0}}(x,\xi)\bar{z}_{0},\bar{z}_{0}\rangle
    +\langle Q_{1}^{z_{0}z}(x,\xi)z_{0},z\rangle\\
\nonumber &&+ \langle Q_{1}^{\bar{z}_{0}z}(x,\xi)\bar{z}_{0},z\rangle
    +\langle Q_{1}^{z_{0}\bar{z}}(x,\xi)z_{0},\bar{z}\rangle
    +\langle Q_{1}^{\bar{z}_{0}\bar{z}}(x,\xi)\bar{z}_{0},\bar{z}\rangle,
\end{eqnarray}
and $Q^{high}_{1}$ has the form
\begin{eqnarray}
\nonumber  Q^{high}_{1} &=&\langle Q_{1}^{yy}(x,\xi)y,y\rangle+\langle Q_{1}^{yz}(x,\xi)y,z\rangle
   +\langle Q_{1}^{y\bar{z}}(x,\xi)y,\bar{z}\rangle
     +\langle Q_{1}^{yzz}(x,\xi)yz,z\rangle\\
\nonumber &&+\langle Q_{1}^{yz\bar{z}}(x,\xi)yz,\bar{z}\rangle
      +\langle Q_{1}^{y\bar{z}\bar{z}}(x,\xi)y\bar{z},\bar{z}\rangle
      +\langle Q_{1}^{yz_{0}}(x,\xi)y,z_{0}\rangle
   +\langle Q_{1}^{y\bar{z}_{0}}(x,\xi)y,\bar{z}_{0}\rangle\\
\nonumber && +\langle Q_{1}^{yz_{0}z_{0}}(x,\xi)yz_{0},z_{0}\rangle
    +\langle Q_{1}^{yz_{0}\bar{z}_{0}}(x,\xi)yz_{0},\bar{z}_{0}\rangle
    +\langle Q_{1}^{y\bar{z}_{0}\bar{z}_{0}}(x,\xi)y\bar{z}_{0},\bar{z}_{0}\rangle\\
\nonumber &&+\langle Q_{1}^{yz_{0}z}(x,\xi)yz_{0},z\rangle
    + \langle Q_{1}^{y\bar{z}_{0}z}(x,\xi)y\bar{z}_{0},z\rangle
    +\langle Q_{1}^{yz_{0}\bar{z}}(x,\xi)yz_{0},\bar{z}\rangle\\
\nonumber &&
    +\langle Q_{1}^{y\bar{z}_{0}\bar{z}}(x,\xi)y\bar{z}_{0},\bar{z}\rangle+
     \langle Q_{1}^{\bar{z}_{0}zz}(x,\xi)\bar{z}_{0}z,z\rangle
     +\langle Q_{1}^{\bar{z}zz}(x,\xi)\bar{z}z,z\rangle
     +\langle Q_{1}^{z_{0}z\bar{z}}(x,\xi)z_{0}z,\bar{z}\rangle\\
\nonumber && +\langle Q_{1}^{\bar{z}_{0}z\bar{z}}(x,\xi)\bar{z}_{0}z,\bar{z}\rangle
     +\langle Q_{1}^{z_{0}\bar{z}\bar{z}}(x,\xi)z_{0}\bar{z},\bar{z}\rangle
      +\langle Q_{1}^{z\bar{z}\bar{z}}(x,\xi)z\bar{z},\bar{z}\rangle\\
\nonumber &&+\langle Q_{1}^{\bar{z}_{0}z_{0}z_{0}}(x,\xi)\bar{z}_{0}z_{0},z_{0}\rangle
    +\langle Q_{1}^{\bar{z}z_{0}z_{0}}(x,\xi)\bar{z}z_{0},z_{0}\rangle
    +\langle Q_{1}^{zz_{0}\bar{z}_{0}}(x,\xi)zz_{0},\bar{z}_{0}\rangle\\
\nonumber &&+\langle Q_{1}^{\bar{z}z_{0}\bar{z}_{0}}(x,\xi)\bar{z}z_{0},\bar{z}_{0}\rangle
     +\langle Q_{1}^{z_{0}\bar{z}_{0}\bar{z}_{0}}(x,\xi)z_{0}\bar{z}_{0},\bar{z}_{0}\rangle
    +\langle Q_{1}^{z\bar{z}_{0}\bar{z}_{0}}(x,\xi)z\bar{z}_{0},\bar{z}_{0}\rangle\\
\nonumber &&+\langle Q_{1}^{z_{0}z_{0}{z}_{0}{z}_{0}}(x,\xi)z_{0}z_{0},\bar{z}_{0}\bar{z}_{0}\rangle
    +\langle Q_{1}^{z_{0}z_{0}\bar{z}\bar{z}}(x,\xi)z_{0}z_{0},\bar{z}\bar{z}\rangle
    +\langle Q_{1}^{z_{0}z_{0}\bar{z}_{0}\bar{z}}(x,\xi)z_{0}z_{0},\bar{z}_{0}\bar{z}\rangle\\
\nonumber &&+\langle Q_{1}^{z_{0}z{z}{z}}(x,\xi)z_{0}z,\bar{z}\bar{z}\rangle
     +\langle Q_{1}^{z_{0}z\bar{z}_{0}\bar{z}_{0}}(x,\xi)z_{0}z,\bar{z}_{0}\bar{z}_{0}\rangle
    +\langle Q_{1}^{z_{0}z\bar{z}_{0}\bar{z}}(x,\xi)z_{0}z,\bar{z}_{0}\bar{z}\rangle\\
\nonumber &&+\langle Q_{1}^{zz\bar{z}{z}}(x,\xi)zz,\bar{z}\bar{z}\rangle
     +\langle Q_{1}^{zz\bar{z}_{0}\bar{z}_{0}}(x,\xi)zz,\bar{z}_{0}\bar{z}_{0}\rangle
    +\langle Q_{1}^{zz\bar{z}_{0}\bar{z}}(x,\xi)zz,\bar{z}_{0}\bar{z}\rangle.
\end{eqnarray}
By the definition of Possion Bracket,\ one has\\
$\textbf{Case. 1.}$
\begin{eqnarray}\nonumber
Q_{1}^{y}(x,\xi)&=&-R_{0}^{yy}\partial_{x}F^{x}_{0}
+\mathbf{i}(R^{yz_{0}}_{0}F^{\bar{z}_{0}}_{0}-R^{y\bar{z}_{0}}_{0}F^{z_{0}}_{0}
+R^{yz}_{0}F^{\bar{z}}_{0}-R^{y\bar{z}}_{0}F^{z}_{0}),\\
\nonumber Q_{1}^{z_{0}}(x,\xi)&=&-R_{0}^{yz_{0}}\partial_{x}F^{x}_{0},\\
\nonumber Q_{1}^{\bar{z}_{0}}(x,\xi)&=& -R_{0}^{y\bar{z}_{0}}\partial_{x}F^{x}_{0}, \\
\nonumber Q_{1}^{z}(x,\xi)&=&-R_{0}^{yz}\partial_{x}F^{x}_{0},\\
\nonumber Q_{1}^{\bar{z}}(x,\xi)&=&-R_{0}^{y\bar{z}}\partial_{x}F^{x}_{0}, \\
\nonumber Q_{1}^{yz_{0}}(x,\xi)&=&\partial_{x}R_{0}^{yz_{0}}F^{y}_{0}-R_{0}^{yy}\partial_{x}F^{z_{0}}_{0}
+\mathbf{i}(R^{yz_{0}z_{0}}_{0}F^{\bar{z}_{0}}_{0}+R^{yz_{0}}_{0}F^{z_{0}\bar{z}_{0}}_{0}
-R^{yz_{0}\bar{z}_{0}}_{0}F^{z_{0}}_{0}\\
\nonumber &&-R^{y\bar{z}_{0}}_{0}F^{z_{0}z_{0}}_{0}
+R^{yz_{0}z}_{0}F^{\bar{z}}_{0}+R^{yz}_{0}F^{z_{0}\bar{z}}_{0}
-R^{yz_{0}\bar{z}}_{0}F^{z}_{0}-R^{y\bar{z}}_{0}F^{z_{0}z}_{0}),\\
\nonumber  Q_{1}^{\bar{z}_{0}z\bar{z}}(x,\xi)&=&\partial_{x}R_{0}^{y\bar{z}_{0}z\bar{z}}F^{y}_{0}
-R_{0}^{y\bar{z}_{0}z}\partial_{x}F^{\bar{z}}_{0}
-R_{0}^{y\bar{z}_{0}\bar{z}}\partial_{x}F^{z}_{0}
-R_{0}^{yz\bar{z}}\partial_{x}F^{\bar{z}_{0}}_{0}\\
\nonumber &&
-R_{0}^{y\bar{z}_{0}}\partial_{x}F^{z\bar{z}}_{0}
-R_{0}^{yz}\partial_{x}F^{\bar{z}_{0}\bar{z}}_{0}
-R_{0}^{y\bar{z}}\partial_{x}F^{\bar{z}_{0}z}_{0}
+\mathbf{i}(R^{\bar{z}zz_{0}}_{0}F^{\bar{z}_{0}\bar{z}_{0}}_{0}\\
\nonumber &&
-R^{\bar{z}_{0}z\bar{z}_{0}}_{0}F^{\bar{z}z_{0}}_{0}
+R^{\bar{z}_{0}\bar{z}z_{0}}_{0}F^{z\bar{z}_{0}}_{0}
-R^{z\bar{z}\bar{z}_{0}}_{0}F^{\bar{z}_{0}z_{0}}_{0}+
R^{\bar{z}_{0}zz}_{0}F^{\bar{z}\bar{z}}_{0}\\
\nonumber &&
-R^{\bar{z}_{0}z\bar{z}}_{0}F^{\bar{z}z}_{0}
+R^{\bar{z}_{0}\bar{z}z}_{0}F^{z\bar{z}}_{0}
-R^{z\bar{z}\bar{z}}_{0}F^{\bar{z}_{0}z}_{0}).
\end{eqnarray}
$\textbf{Case. 2.}$
\begin{eqnarray}
\nonumber
Q_{1}^{y}(x,\xi)&=&-R_{0}^{yy}\partial_{x}F^{x}_{0}
+\mathbf{i}(R^{yz_{0}}_{0}F^{\bar{z}_{0}}_{0}-R^{y\bar{z}_{0}}_{0}F^{z_{0}}_{0}
+R^{yz}_{0}F^{\bar{z}}_{0}-R^{y\bar{z}}_{0}F^{z}_{0}),\\
\nonumber Q_{1}^{z_{0}\bar{z}_{0}}(x,\xi)&=&-R_{0}^{yz_{0}}\partial_{x}F^{\bar{z}_{0}}_{0}
-R_{0}^{y\bar{z}_{0}}\partial_{x}F^{z_{0}}_{0}
-R_{0}^{yz_{0}\bar{z}_{0}}\partial_{x}F^{x}_{0}\\
\nonumber&&+\mathbf{i}(R^{z_{0}\bar{z}_{0}z_{0}}_{0}F^{\bar{z}_{0}}_{0}
-R^{z_{0}\bar{z}_{0}\bar{z}_{0}}_{0}F^{z_{0}}_{0}
+R^{z_{0}\bar{z}_{0}z}_{0}F^{\bar{z}}_{0}-R^{z_{0}\bar{z}_{0}\bar{z}}_{0}F^{z}_{0}),\\
\nonumber Q_{1}^{\bar{z}_{0}z}(x,\xi)&=&-R_{0}^{yz}\partial_{x}F^{\bar{z}_{0}}_{0}
-R_{0}^{y\bar{z}_{0}}\partial_{x}F^{z}_{0}
-R_{0}^{yz\bar{z}_{0}}\partial_{x}F^{x}_{0}\\
\nonumber &&
+\mathbf{i}(R^{z\bar{z}_{0}z_{0}}_{0}F^{\bar{z}_{0}}_{0}
-R^{z\bar{z}_{0}\bar{z}_{0}}_{0}F^{z_{0}}_{0}
+R^{z\bar{z}_{0}z}_{0}F^{\bar{z}}_{0}-R^{z\bar{z}_{0}\bar{z}}_{0}F^{z}_{0}),\\
\nonumber Q_{1}^{yy}(x,\xi)&=&\partial_{x}R_{0}^{yy}F^{y}_{0}-R_{0}^{yy}\partial_{x}F^{y}_{0},\\
\nonumber Q_{1}^{yz\bar{z}}(x,\xi)&=& \partial_{x}R_{0}^{yz\bar{z}}F^{y}_{0}
-R_{0}^{yy}\partial_{x}F^{z\bar{z}}_{0}
+\mathbf{i}(R^{yzz_{0}}_{0}F^{\bar{z}\bar{z}_{0}}_{0}
-R^{yz\bar{z}_{0}}_{0}F^{\bar{z}z_{0}}_{0}+R^{y\bar{z}z_{0}}_{0}F^{z\bar{z}_{0}}_{0}\\
\nonumber &&
-R^{y\bar{z}\bar{z}_{0}}_{0}F^{zz_{0}}_{0}+
R^{yzz}_{0}F^{\bar{z}\bar{z}}_{0}-R^{yz\bar{z}}_{0}F^{\bar{z}z}_{0}+
R^{y\bar{z}z}_{0}F^{z\bar{z}}_{0}-R^{y\bar{z}\bar{z}}_{0}F^{zz}_{0}),\\
\nonumber Q_{1}^{z_{0}z_{0}\bar{z}_{0}\bar{z}_{0}}(x,\xi)
&=&\partial_{x}R_{0}^{z_{0}z_{0}\bar{z}_{0}\bar{z}_{0}}F^{y}_{0}
+\mathbf{i}(
-R_{0}^{z_{0}z_{0}\bar{z}_{0}\bar{z}}F^{\bar{z}_{0}z}_{0}
+R_{0}^{z_{0}\bar{z}_{0}\bar{z}_{0}z}F^{z_{0}\bar{z}}_{0}).
\end{eqnarray}
For convenience,\ we only show part of representative terms.

Due to
\[
\widehat{Q^{z_{0}}_{0}}(0,\xi)= \widehat{Q^{z_{0}}_{1}}(0,\xi)
+ \frac{1}{2!}\widehat{Q^{z_{0}}_{2}}(0,\xi)+...+\frac{1}{n!}\widehat{Q^{z_{0}}_{n}}(0,\xi)+...,
\]
consider the term
\begin{eqnarray}
\label{231}Q^{z_{0}}_{1}(x,\xi)&=& -\sum^{n}_{j=1}\left(\sum_{k}\widehat{R_{0}^{y_{j}z_{0}}}(k,\xi)e^{\mathbf{i}\langle k,x\rangle})(\partial_{x_{j}}\sum_{l}\widehat{F^{x}_{0}}(l,\xi)e^{\mathbf{i}\langle l,x\rangle}\right)\\
\nonumber &=& -\sum^{n}_{j=1}\left(\sum_{k,l}l_{j}\widehat{R_{0}^{y_{j}z_{0}}}(k,\xi)\widehat{F^{x}_{0}}(l,\xi)e^{\mathbf{i}\langle k+l,x\rangle}\right)\\
\nonumber &=& -\sum^{n}_{j=1}\left(\sum_{k,l}l_{j}\widehat{R_{0}^{y_{j}z_{0}}}(k,\xi)
\widehat{R^{x}_{0}}(l,\xi)\frac{1}{\langle l,\omega_{0}\rangle}e^{\mathbf{i}\langle k+l,x\rangle}\right).
\end{eqnarray}
Let
\begin{eqnarray}
\nonumber V_{3} &=&\left\{x|x_{1}=(0,...-1,...0,...0,0,...)^{T}, x_{2}=(0,...1,...0,...0,...)^{T},\right.\\
      \nonumber && \left.x_{3}=(0,...-1,...1,...1,0,...)^{T}, x_{4}=(0,...-1,...2,...0,...)^{T},\right.\\
      \nonumber && \left.x_{5}=(0,...0,...1,...0,...)^{T},x_{6}=(0,...1,...-1,...-1,0,...)^{T},\right.\\
      \nonumber && \left.x_{7}=(0,...1,...-2,...0,...)^{T}, x_{8}=(0,...0,...-1,...0,...)^{T}\right\},\\
\nonumber V_{4}&=&\left\{y|y_{1}=(0,...-1,...1,...0,...)^{T}, y_{2}=(0,...1,...-1,...0,...)^{T},\right.\\
\nonumber && \left. y_{3}=(0,...1,...1,...0,...)^{T}, y_{4}=(0,...-1,...-1,...0,...)^{T},\right.\\
\nonumber && \left. y_{5}=(0,...0,...2,...0,...)^{T}, y_{6}=(0,...0,...-2,...0,...)^{T},\right.\\
\nonumber && \left. y_{7}=(0,...-1,...-1,...1,...1,...)^{T},y_{8}=(0,...1,...1,...-1,...-1...)^{T},\right.\\
\nonumber && \left. y_{9}=(0,...-1,...-1,...2,...)^{T}, y_{10}=(0,...1,...1,...-2,...)^{T},\right.\\
\nonumber && \left. y_{11}=(0,...-1,...1,...0,...)^{T}, y_{12}=(0,...1,...-1,...0,...)^{T},\right.\\
\nonumber && \left. y_{13}=(0,...-2,...2,...0,...)^{T}, y_{14}=(0,...2,...-2,...0,...)^{T}\right\},
\end{eqnarray}
where the nonzero elements are at any positions.\\
It is easy to verify that if and only if $ k\in V_{3}$,
\[
\widehat{R^{y_{j}z_{0}}_{0}}(k,\xi)\neq0,
\]
and if and only if $ l\in V_{4}$,
\[
\widehat{R^{x}_{0}}(l,\xi)\neq0.
\]
In order to estimate (\ref{231}),\ the equation
\begin{eqnarray}\label{232}
k+l=0,k\in V_{3},l\in V_{4};
\end{eqnarray}
should be solved.\ That is,\ if (\ref{232}) has a solution,\ let $ v_{0}=(1,1,...1)^{T}$,\ then $ (k+l)^{T}v_{0}=0$.\ But for any $ k\in V_{3}$,
\[
k\cdot v_{0}=\pm 1,
\]
and for any $ l\in V_{4} $,
\begin{eqnarray}\nonumber
l\cdot v_{0}=0\ \mbox{or}\ \pm2.
\end{eqnarray}
Clearly,\ (\ref{232}) is unsolved.\\
Hence
\[
\widehat{Q^{z_{0}}_{1}}(0,\xi)=0.
\]
Moreover,\ due to
\begin{eqnarray}\nonumber
Q_{2}=\{Q_{1},F_{0}\}=\{Q^{low}_{1},F_{0}\}+\{Q^{high}_{1},F_{0}\},
\end{eqnarray}
one has
\begin{eqnarray}
\nonumber Q^{z_{0}}_{2}(x,\xi)&=&\{Q^{low}_{1},F_{0}\}^{z_{0}}(x,\xi)+\{Q^{high}_{1},F_{0}\}^{z_{0}}(x,\xi)\\
\nonumber &=&\partial_{x}Q^{z_{0}}_{1}F_{0}^{y}
-\partial_{x}F^{z_{0}}_{0}Q_{1}^{y}-Q_{1}^{yz_{0}}\partial_{x}F^{x}_{0}+\mathbf{i}
(2Q^{z_{0}z_{0}}_{1}F^{\bar{z}_{0}}_{0}-2Q^{\bar{z}_{0}}_{1}F^{z_{0}z_{0}}_{0}
+Q^{z_{0}}_{1}F^{z_{0}\bar{z}_{0}}_{0}\\
\nonumber&&-Q^{z_{0}\bar{z}_{0}}_{1}F^{z_{0}}_{0}+Q^{z_{0}z}_{1}F^{\bar{z}}_{0}-Q^{\bar{z}}_{1}F^{z_{0}z}_{0}
+Q^{z}_{1}F^{z_{0}\bar{z}}_{0}-Q^{z_{0}\bar{z}}_{1}F^{z}_{0}).
\end{eqnarray}
Since
\begin{eqnarray}\nonumber
&&(\partial_{x}Q^{z_{0}}_{1}F_{0}^{y}-\partial_{x}F^{z_{0}}_{0}Q_{1}^{y})(x,\xi)\\
\nonumber&=&-\sum^{n}_{j=1}\left(\sum_{k,l,m\neq0}m_{j}\widehat{R_{0}^{y_{j}z_{0}}}(k,\xi)
\widehat{R^{x}_{0}}(l,\xi)\widehat{R^{y}_{0}}(m,\xi))\frac{1}{\langle l,\omega_{0}\rangle}\cdot\frac{1}{\langle m,\omega_{0}\rangle}\right.\\
\nonumber&&\left.+(\sum_{k,l,m\neq0}l_{j}\widehat{R_{0}^{z_{0}}}(k,\xi)\widehat{R_{0}^{yy_{j}}}(l,\xi)\widehat{R^{x}_{0}}(m,\xi)
\frac{1}{\langle k,\omega_{0}\rangle}\cdot\frac{1}{\langle m,\omega_{0}\rangle}\right.\\
\nonumber&&\left.-\mathbf{i}(\sum_{k,l,m\neq0}k_{j}\widehat{R_{0}^{z_{0}}}(k,\xi)\widehat{R_{0}^{yz_{0}}}(l,\xi)
\widehat{R^{\bar{z}_{0}}_{0}}(m,\xi)\frac{1}{\langle k,\omega_{0}\rangle}\cdot\frac{1}{\langle m,\omega_{0}\rangle}\right.\\
\nonumber&&\left.+\sum_{k,l,m\neq0}k_{j}\widehat{R_{0}^{z_{0}}}(k,\xi)\widehat{R_{0}^{y\bar{z}_{0}}}(l,\xi)
\widehat{R^{z_{0}}_{0}}(m,\xi)\frac{1}{\langle k,\omega_{0}\rangle}\cdot\frac{1}{\langle m,\omega_{0}\rangle}\right.\\
\nonumber&&\left.-\sum_{k,l,m\neq0}k_{j}\widehat{R_{0}^{z_{0}}}(k,\xi)\widehat{R_{0}^{yz}}(l,\xi)
\widehat{R^{\bar{z}}_{0}}(m,\xi)\frac{1}{\langle k,\omega_{0}\rangle}\cdot\frac{1}{\langle m,\omega_{0}\rangle+\widehat{\Omega}^{j}_{0}}\right.\\
\nonumber&&\left.+\sum_{k,l,m\neq0}k_{j}\widehat{R_{0}^{z_{0}}}(k,\xi)
\widehat{R_{0}^{y\bar{z}}}(l,\xi)\widehat{R^{z}_{0}}(m,\xi)
\frac{1}{\langle k,\omega_{0}\rangle}\cdot\frac{1}{\langle m,\omega_{0}\rangle-\widehat{\Omega}^{j}_{0}})\right)e^{\mathbf{i}\langle k+l+m,x\rangle},
\end{eqnarray}
we thus have
\begin{eqnarray}\nonumber
\widehat{(\partial_{x}Q^{z_{0}}_{1}F_{0}^{y}-\partial_{x}F^{z_{0}}_{0}Q_{1}^{y})}(0,\xi)=0.
\end{eqnarray}
In fact,\ the equations
\begin{eqnarray}\label{235}
k+l+m=0,k,l,m \in V_{3}
\end{eqnarray}
and
\begin{eqnarray}\label{236}
k+l+m=0,k\in V_{3},l,m \in V_{4}
\end{eqnarray}
are unsolved.\ If (\ref{235}) and (\ref{236}) have a solution,\ let $ v_{0}=(1,1,...1)^{T}$,\ then $ (k+l+m)^{T}v_{0}=0$.\ But for any $ k,l,m\in V_{3}$,
\[
k\cdot v_{0}= \pm 1,
\]
and for any $ l\in V_{4}$,
\begin{eqnarray}\nonumber
l\cdot v_{0}= 0\ \mbox{or}\ \pm 2.
\end{eqnarray}
Clearly,\ (\ref{235}) and (\ref{236}) are unsolved.\\
Similarly,\ one has
\begin{eqnarray}\nonumber
\nonumber\widehat{(Q^{z_{0}z_{0}}_{1}F^{\bar{z}_{0}}_{0}-Q^{\bar{z}_{0}}_{1}F^{z_{0}z_{0}}_{0})}(0,\xi)=0,\\
\nonumber\widehat{(Q^{z_{0}}_{1}F^{z_{0}\bar{z}_{0}}_{0}-Q^{z_{0}\bar{z}_{0}}_{1}F^{z_{0}}_{0})}(0,\xi)=0,\\
\nonumber\widehat{(Q^{z_{0}z}_{1}F^{\bar{z}}_{0}-Q^{\bar{z}}_{1}F^{z_{0}z}_{0})}(0,\xi)=0,\\
\nonumber\widehat{(Q^{z}_{1}F^{z_{0}\bar{z}}_{0}-Q^{z_{0}\bar{z}}_{1}F^{z}_{0})}(0,\xi)=0.
\end{eqnarray}
Moreover,\ due to
\begin{eqnarray}
\nonumber &&Q_{1}^{yz_{0}}\partial_{x}F^{x}_{0}(x,\xi)\\
\nonumber &=&(\partial_{x}R_{0}^{yz_{0}}F^{y}_{0}
-R_{0}^{yy}\partial_{x}F^{z_{0}}_{0}
+\mathbf{i}(R^{yz_{0}z_{0}}_{0}F^{\bar{z}_{0}}_{0}+R^{yz_{0}}_{0}F^{z_{0}\bar{z}_{0}}_{0}
-R^{yz_{0}\bar{z}_{0}}_{0}F^{z_{0}}_{0}-R^{y\bar{z}_{0}}_{0}F^{z_{0}z_{0}}_{0}\\
\nonumber&&+R^{yz_{0}z}_{0}F^{\bar{z}}_{0}+R^{yz}_{0}F^{z_{0}\bar{z}}_{0}
-R^{yz_{0}\bar{z}}_{0}F^{z}_{0}-R^{y\bar{z}}_{0}F^{z_{0}z}_{0}))\partial_{x}F^{x}_{0}\\
\nonumber &=&\sum^{n}_{j=1}\sum_{j,k,l\neq0}(m_{j}(\widehat{R_{0}^{yz_{0}}}(k,\xi)\widehat{R_{0}^{y_{j}}}(l,\xi)
\widehat{R_{0}^{x}}(m,\xi)\frac{1}{\langle l,\omega_{0}\rangle}\frac{1}{\langle m,\omega_{0}\rangle}\\
\nonumber&&-\widehat{R_{0}^{z_{0}}}(k,\xi)\widehat{R_{0}^{yy}}(l,\xi)
\widehat{R_{0}^{x}}(m,\xi)\frac{1}{\langle k,\omega_{0}\rangle}\frac{1}{\langle m,\omega_{0}\rangle})e^{\mathbf{i}\langle k+l+m,x\rangle}\\
\nonumber&&+\sqrt{-1}m_{j}((\widehat{R_{0}^{\bar{z}_{0}}}(k,\xi)
\widehat{R_{0}^{yz_{0}z_{0}}}(l,\xi)
\widehat{R_{0}^{x}}(m,\xi)\frac{1}{\langle k,\omega_{0}\rangle}\frac{1}{\langle m,\omega_{0}\rangle}\\
\nonumber&&+\widehat{R_{0}^{yz_{0}}}(k,\xi)\widehat{R_{0}^{z_{0}\bar{z}_{0}}}(l,\xi)
\widehat{R_{0}^{x}}(m,\xi)\frac{1}{\langle l,\omega_{0}}\frac{1}{\langle m,\omega_{0}\rangle}\\
\nonumber&&-\widehat{R_{0}^{z_{0}}}(k,\xi)\widehat{R_{0}^{yz_{0}\bar{z}_{0}}}(l,\xi)
\widehat{R_{0}^{x}}(m,\xi)\frac{1}{\langle k,\omega_{0}\rangle}\frac{1}{\langle m,\omega_{0}\rangle}\\
\nonumber&&-\widehat{R_{0}^{y\bar{z}_{0}}}(k,\xi)\widehat{R_{0}^{z_{0}z_{0}}}(l,\xi)
\widehat{R_{0}^{x}}(m,\xi)\frac{1}{\langle l,\omega_{0}+2\widehat{\Omega}^{0}_{0}}\frac{1}{\langle m,\omega_{0}\rangle})
e^{\mathbf{i}\langle k+l+m,x\rangle}\\
\nonumber&&+(
\widehat{R_{0}^{\bar{z}}}(k,\xi)\widehat{R_{0}^{yz_{0}z}}(l,\xi)
\widehat{R_{0}^{x}}(m,\xi)\frac{1}{\langle m,\omega_{0}\rangle}\frac{1}{\langle k,\omega_{0}\rangle-\widehat{\Omega}^{j}_{0}}\\
\nonumber&&+\widehat{R_{0}^{yz}}(k,\xi)\widehat{R_{0}^{z_{0}\bar{z}}}(l,\xi)
\widehat{R_{0}^{x}}(m,\xi)\frac{1}{\langle m,\omega_{0}\rangle}\\
\nonumber&&-\widehat{R_{0}^{z}}(k,\xi)\widehat{R_{0}^{yz_{0}\bar{z}}}(l,\xi)
\widehat{R_{0}^{x}}(m,\xi)\frac{1}{\langle m,\omega_{0}\rangle}\frac{1}{\langle k,\omega_{0}\rangle+\widehat{\Omega}^{j}_{0}}\\
\nonumber&&-\widehat{R_{0}^{y\bar{z}}}(k,\xi)\widehat{R_{0}^{z_{0}z}}(l,\xi)
\widehat{R_{0}^{x}}(m,\xi)\frac{1}{\langle m,\omega_{0}\rangle})
e^{\mathbf{i}\langle k+l+m,x\rangle})),
\end{eqnarray}
one obtains
\[
\widehat{(Q_{1}^{yz_{0}}\partial_{x}F^{x}_{0})}(0,\xi)=0.
\]
In fact,\ the equation
\begin{eqnarray}\label{237}
k+l+m=0,k\in V_{3},l,m\in V_{4}
\end{eqnarray}
is unsolved.\ If (\ref{237}) has a solution,\ let $ v_{0}=(1,1,...1)^{T}$,\ then $ (k+l+m)^{T}v_{0}=0$.\ But for any $ k\in V_{3}$,
\[
k \cdot v_{0}= \pm 1,
\]
and for any $ l,m \in V_{4}$,
\begin{eqnarray}\nonumber
(l+m)\cdot v_{0}=0\ \mbox{or}\ \pm2.
\end{eqnarray}
Clearly,\ (\ref{237}) is unsolved.\\
Hence,\ we obtain
\[
\widehat{Q^{z_{0}}_{2}}(0,\xi)=0.
\]
Let
\begin{eqnarray}
\nonumber S_{3}&=&\left\{\widehat{R_{0}^{yz_{0}}}(k,\xi),\widehat{R_{0}^{y\bar{z}_{0}}}(k,\xi)
\widehat{R_{0}^{yz}}(k,\xi),\widehat{R_{0}^{y\bar{z}}}(k,\xi),\right.\\
\nonumber && \left.
\widehat{R_{0}^{\bar{z}_{0}zz}}(k,\xi),\widehat{R_{0}^{\bar{z}zz}}(k,\xi),
\widehat{R_{0}^{z_{0}z\bar{z}}}(k,\xi),\widehat{R_{0}^{\bar{z}_{0}z\bar{z}}}(k,\xi),\right.\\
\nonumber && \left.\widehat{R_{0}^{z_{0}\bar{z}\bar{z}}}(k,\xi),
\widehat{R_{0}^{z\bar{z}\bar{z}}}(k,\xi),
\widehat{R_{0}^{\bar{z}_{0}z_{0}z_{0}}}(k,\xi),\widehat{R_{0}^{\bar{z}z_{0}z_{0}}}(k,\xi),\right.\\
\nonumber && \left.\widehat{R_{0}^{z_{0}z_{0}\bar{z}_{0}}}(k,\xi),\widehat{R_{0}^{\bar{z}_{0}z_{0}\bar{z}_{0}}}(k,\xi),
\widehat{R_{0}^{z_{0}\bar{z}_{0}\bar{z}_{0}}}(k,\xi),
\widehat{R_{0}^{z\bar{z}_{0}\bar{z}_{0}}}(k,\xi)\right\}.
\end{eqnarray}
Moreover,\ when $ |k|$ is even,\ all elements $ S_{3}$ are equal to $0$.\\
And let
\begin{eqnarray}\nonumber
S_{4}&=&\left\{\widehat{R_{0}^{y_{i}y_{j}}}(k,\xi),\widehat{R_{0}^{y_{j}z_{0}z_{0}}}(k,\xi),
\widehat{R_{0}^{y_{j}z_{0}\bar{z}_{0}}}(k,\xi),
\widehat{R_{0}^{y_{j}\bar{z}_{0}\bar{z}_{0}}}(k,\xi),\right.\\
\nonumber && \left.\widehat{R_{0}^{y_{l}z_{i}z_{j}}}(k,\xi),\widehat{R_{0}^{y_{l}z_{i}\bar{z}_{j}}}(k,\xi),
\widehat{R_{0}^{y_{l}\bar{z}_{i}\bar{z}_{j}}}(k,\xi),
\widehat{R_{0}^{y_{l}z_{0}z_{j}}}(k,\xi),\right.\\
\nonumber && \left.\widehat{R_{0}^{y_{l}z_{0}\bar{z}_{j}}}(k,\xi),
\widehat{R_{0}^{y_{l}\bar{z}_{0}z_{j}}}(k,\xi),
\widehat{R_{0}^{y_{l}\bar{z}_{0}\bar{z}_{j}}}(k,\xi),
\widehat{R_{0}^{z_{i}z_{j}\bar{z}_{k}\bar{z}_{l}}}(k,\xi)\right\}
\end{eqnarray}
Moreover,\ when $ |k|$ is odd,\ all elements $ S_{4}$ are equal to $0$.\\
Here we can let $\widehat{R_{0}^{yz}}(k,\xi)$ represents all the elements in $ S_{3} $ and
$\widehat{R_{0}^{yy}}(k,\xi)$ represents all the elements in $ S_{4} $.\\
Analogously,\ the coefficients of $ Q_{n}$ can be written into the form as follows:
\begin{eqnarray}\nonumber
Q_{n}^{yz}(x,\xi)&=&a_{1}(\xi,t)R_{0}^{yz}F^{y}_{0}\cdot\cdot\cdot F^{y}_{0}+a_{2}(\xi,t)F^{z}_{0}R_{0}^{yy}F^{y}_{0}\cdot\cdot\cdot F^{y}_{0}\\
\nonumber&=&\sum_{k,l,...,m\ \mbox{are not all zero}}a_{1}(\xi,t)\widehat{R_{0}^{yz}}(k,\xi)
\widehat{R_{0}^{y}}(l,\xi)\cdot\cdot\cdot\widehat{R_{0}^{y}}(m,\xi) e^{\mathbf{i}\langle k+l+...+m,x\rangle}\\
\nonumber&&+\sum_{k,l,...,m\ \mbox{are not all zero}}a_{2}(\xi,t)\widehat{R_{0}^{z}}(k,\xi)
\widehat{R_{0}^{yy}}(l,\xi)\cdot\cdot\cdot\widehat{R_{0}^{y}}(m,\xi) e^{\mathbf{i}\langle k+l+...+m,x\rangle}\\
\nonumber Q_{n}^{yy}(x,\xi)&=&b_{1}(\xi,t)R_{0}^{yz}F^{z}_{0}F^{y}_{0}\cdot\cdot\cdot F^{y}_{0}+b_{2}(\xi,t)R_{0}^{yy}F^{y}_{0}\cdot\cdot\cdot F^{y}_{0}\\
\nonumber&=&\sum_{k,l,...,m\ \mbox{are not all zero}}b_{1}(\xi,t)\widehat{R_{0}^{yz}}(k,\xi)\widehat{R_{0}^{z}}(l,\xi)
\widehat{R_{0}^{y}}(r,\xi)\cdot\cdot\cdot\widehat{R_{0}^{y}}(m,\xi) e^{\mathbf{i}\langle k+l+...+m,x\rangle}\\
\nonumber&&+\sum_{k,l,...,m\ \mbox{are not all zero}}b_{2}(\xi,t)\widehat{R_{0}^{yy}}(k,\xi)\widehat{R_{0}^{y}}(l,\xi)
\cdot\cdot\cdot\widehat{R_{0}^{y}}(m,\xi) e^{\mathbf{i}\langle k+l+...+m,x\rangle}\\
\nonumber Q_{n}^{yzz}(x,\xi)&=&c_{1}(\xi,t)R_{0}^{yz}F^{z}_{0}F^{y}_{0}\cdot\cdot\cdot F^{y}_{0}+c_{2}(\xi,t)R_{0}^{yy}F^{y}_{0}\cdot\cdot\cdot F^{y}_{0}\\
\nonumber &=&\sum_{k,l,...,m\ \mbox{are not all zero}}c_{1}(\xi,t)\widehat{R_{0}^{yz}}(k,\xi)\widehat{R_{0}^{z}}(l,\xi)
\widehat{R_{0}^{y}}(r,\xi)\cdot\cdot\cdot\widehat{R_{0}^{y}}(m,\xi) e^{\mathbf{i}\langle k+l+...+m,x\rangle}\\
\nonumber&&+\sum_{k,l,...,m\ \mbox{are not all zero}}c_{2}(\xi,t)\widehat{R_{0}^{yy}}(k,\xi)\widehat{R_{0}^{y}}(l,\xi)
\cdot\cdot\cdot\widehat{R_{0}^{y}}(m,\xi) e^{\mathbf{i}\langle k+l+...+m,x\rangle}\\
\nonumber Q_{n}^{zz}(x,\xi)&=&d_{1}(\xi,t)Q_{n-1}^{yz}F^{z}_{0}\cdot\cdot\cdot F^{y}_{0}+d_{2}(\xi,t)F^{z}_{0}F^{z}_{0}R_{0}^{yy}F^{y}_{0}\cdot\cdot\cdot F^{y}_{0}\\
\nonumber&=&\sum_{k,l,...,m\ \mbox{are not all zero}}d_{1}(\xi,t)\widehat{R_{0}^{yz}}(k,\xi)
\widehat{R_{0}^{z}}(l,\xi)\cdot\cdot\cdot\widehat{R_{0}^{y}}(m,\xi) e^{\mathbf{i}\langle k+l+...+m,x\rangle}\\
\nonumber&&+\sum_{k,l,...,m\ \mbox{are not all zero}}d_{2}(\xi,t)\widehat{R_{0}^{z}}(k,\xi)\widehat{R_{0}^{z}}(l,\xi)
\widehat{R_{0}^{yy}}(r,\xi)\cdot\cdot\cdot\widehat{R_{0}^{y}}(m,\xi) e^{\mathbf{i}\langle k+l+...+m,x\rangle}\\
\nonumber Q_{n}^{z}(x,\xi)&=&e_{1}(\xi,t)Q_{n-1}^{yz}F^{x}_{0}\cdot\cdot\cdot F^{y}_{0}+e_{2}(\xi,t)F^{z}_{0}R_{0}^{yy}F^{y}_{0}\cdot\cdot\cdot F^{y}_{0}\\
\nonumber &=&\sum_{k,l,...,m\ \mbox{are not all zero}}e_{1}(\xi,t)\widehat{R_{0}^{yz}}(k,\xi)
\widehat{R_{0}^{x}}(l,\xi)\cdot\cdot\cdot\widehat{R_{0}^{y}}(m,\xi) e^{\mathbf{i}\langle k+l+...+m,x\rangle}\\
\nonumber&&+\sum_{k,l,...,m\ \mbox{are not all zero}}e_{2}(\xi,t)\widehat{R_{0}^{z}}(k,\xi)\widehat{R_{0}^{z}}(k,\xi)
\widehat{R_{0}^{yy}}(l,\xi)\cdot\cdot\cdot\widehat{R_{0}^{y}}(m,\xi) e^{\mathbf{i}\langle k+l+...+m,x\rangle}\\
\nonumber&&+\sum_{k,l,...,m\ \mbox{are not all zero}}e_{3}(\xi,t)\widehat{R_{0}^{z}}(k,\xi)
\widehat{R_{0}^{z}}(k,\xi)\widehat{R_{0}^{z}}(r,\xi)
\cdot\cdot\cdot\widehat{R_{0}^{y}}(m,\xi) e^{\mathbf{i}\langle k+l+...+m,x\rangle}
\end{eqnarray}
Note that
\[
Q_{n+1}=\{Q_{n},F_{0}\}=\{Q^{low}_{n},F_{0}\}+\{Q^{high}_{n},F_{0}\}.
\]
It follows that
\begin{eqnarray}\nonumber
Q^{z_{0}}_{n+1}(x,\xi)&=&\{Q^{low}_{n},F_{0}\}^{z_{0}}(x,\xi)+\{Q^{high}_{n},F_{0}\}^{z_{0}}(x,\xi)\\
\nonumber &=&\partial_{x}Q^{z_{0}}_{n}F_{0}^{y}
-\partial_{x}F^{z_{0}}_{0}Q_{n}^{y}+\mathbf{i}
(2Q^{z_{0}z_{0}}_{n}F^{\bar{z}_{0}}_{0}-2Q^{\bar{z}_{0}}_{n}F^{z_{0}z_{0}}_{0}
+Q^{z_{0}}_{n}F^{z_{0}\bar{z}_{0}}_{0}-Q^{z_{0}\bar{z}_{0}}_{n}F^{z_{0}}_{0}\\
\nonumber&&+Q^{z_{0}z}_{n}F^{\bar{z}}_{0}-Q^{\bar{z}}_{n}F^{z_{0}z}_{0}
+Q^{z}_{n}F^{z_{0}\bar{z}}_{0}-Q^{z_{0}\bar{z}}_{n}F^{z}_{0})-Q_{n}^{yz_{0}}\partial_{x}F^{x}_{0}.
\end{eqnarray}
Due to
\begin{eqnarray}
\nonumber &&(\partial_{x}Q^{z_{0}}_{n}F_{0}^{y}-\partial_{x}F^{z_{0}}_{n}Q_{1}^{y})(x,\xi)\\
\nonumber&=&\sum_{k,l,...,m,n\ \mbox{are not all zero}}\left(f_{1}(\xi,t)\widehat{R_{0}^{yz}}(k,\xi)
\widehat{R_{0}^{x}}(l,\xi)\cdot\cdot\cdot\widehat{R_{0}^{y}}(m,\xi)\widehat{R_{0}^{y}}(n,\xi)\right.\\
\nonumber&&\left.+f_{2}(\xi,t)\widehat{R_{0}^{z}}(k,\xi)
\widehat{R_{0}^{yy}}(l,\xi)\cdot\cdot\cdot\widehat{R_{0}^{y}}(m,\xi)\widehat{R_{0}^{y}}(n,\xi)\right. \\
\nonumber&&\left.+f_{3}(\xi,t)\widehat{R_{0}^{z}}(k,\xi)
\widehat{R_{0}^{yz}}(k,\xi)\widehat{R_{0}^{z}}(r,\xi)
\cdot\cdot\cdot\widehat{R_{0}^{y}}(m,\xi)
\widehat{R_{0}^{y}}(n,\xi)\right)e^{\mathbf{i}\langle k+l+...+m+n,x\rangle},
\end{eqnarray}
the equations
\[
k+l+...+n=0,
\]
for any $ k\in V_{3},l,...m,n\in V_{4} $ and
\[
k+l+r+t...+n=0,
\]
for any $ k,l,r\in V_{3},t,...m,n\in V_{4} $ are unsolvable.\ Thus,
\[
\widehat{(\partial_{x}Q^{z_{0}}_{n}F_{0}^{y}-\partial_{x}F^{z_{0}}_{n}Q_{1}^{y})}(0,\xi)=0.
\]
Similarly, one obtains
\begin{eqnarray}
\nonumber \widehat{(Q^{z_{0}z_{0}}_{n}F^{\bar{z}_{0}}_{0}-Q^{\bar{z}_{0}}_{n}F^{z_{0}z_{0}}_{0})}(0,\xi)=0,\\
\nonumber \widehat{(Q^{z_{0}}_{n}F^{z_{0}\bar{z}_{0}}_{0}-Q^{z_{0}\bar{z}_{0}}_{n}F^{z_{0}}_{0})}(0,\xi)=0,\\
\nonumber \widehat{(Q^{z_{0}z}_{n}F^{\bar{z}}_{0}-Q^{\bar{z}}_{n}F^{z_{0}z}_{0})}(0,\xi)=0,\\
\nonumber \widehat{(Q^{z}_{n}F^{z_{0}\bar{z}}_{0}-Q^{z_{0}\bar{z}}_{n}F^{z}_{0})}(0,\xi)=0,\\
\nonumber \widehat{(Q_{n}^{yz_{0}}\partial_{x}F^{x}_{0})}(0,\xi)=0.
\end{eqnarray}
Hence
\[
\widehat{Q^{z_{0}}_{0}}(0,\xi)=0.
\]
\end{proof}
\begin{lemma}\label{l7} $ \widehat{R^{z_{0}}_{1}}(0,\xi) = 0.$
\end{lemma}
\begin{proof}
Combining Lemma \ref{l5} and Lemma \ref{l6},\ we have
\begin{eqnarray}
\widehat{R_{1}^{z_{0}}}(0,\xi)
&=\widehat{P^{z_{0}}_{0}}(0,\xi)+\widehat{Q^{z_{0}}_{0}}(0,\xi)=0.
\end{eqnarray}
\end{proof}

For the iterative process,\ since
\[
N_{2}=N_{1}+\widehat{N}_{1},R_{2}=\int^{1}_{0}\{(1-t)\widehat{N}_{1}+tR_{1},F_{1}\}\circ X^{t}_{F_{1}}\mathrm{d}t,
\]
and the Possion bracket keep the form,\ it is easy to check that $R_{2} $ has the same form with $ R_{1}$.\ Thus  we can also obtain
\[
\widehat{R_{2}^{z_{0}}}(0,\xi)=0,\  \widehat{R_{2}^{\bar{z}_{0}}}(0,\xi)=0.
\]

Inductively,\ for any $ n$,\ one has
\[
\widehat{R_{n}^{z_{0}}}(0,\xi)=0,\ \widehat{R_{n}^{\bar{z}_{0}}}(0,\xi)=0.
\]

Finally,\ by using the Theorem \ref{mainthm},\ we have
\[
H_{\infty}(x,y,z^{*},\bar{z}^{*},\xi)=\lim_{m\rightarrow\infty}H_{m}(x,y,z^{*},\bar{z}^{*},\xi)
=N_{\infty}(y,z^{*},\bar{z}^{*},\xi)+R_{\infty}(x,y,z^{*},\bar{z}^{*},\xi),
\]
where
\begin{eqnarray}
\nonumber N_{\infty}(y,z^{*},\bar{z}^{*},\xi) &=& \widehat{N^{x}_{\infty}}(\xi)+\langle{\omega_{\infty}}(\xi),y\rangle
+\langle{\Omega_{\infty}}(\xi)z,\bar{z}\rangle\\
\nonumber &&+\langle\widehat{N^{z_{0}z_{0}}_{\infty}}(\xi)z_{0},z_{0}\rangle
 +\langle\widehat{N^{z_{0}\bar{z}_{0}}_{\infty}}(\xi)z_{0},\bar{z}_{0}\rangle
+\langle\widehat{N^{\bar{z}_{0}\bar{z}_{0}}_{\infty}}(\xi)\bar{z}_{0},\bar{z}_{0}\rangle,
\end{eqnarray}
and
\begin{eqnarray}
\nonumber R_{\infty}(x,y,z^{*},\bar{z}^{*},\xi) &=& \sum_{\alpha\in \mathbb{N}^{n},\beta,\gamma\in \mathbb{N}^{\mathbb{N}},2|\alpha|+|\beta|+|\gamma|\geq3}
\widehat{R_{\infty}^{\alpha\beta\gamma}}(k,\xi)e^{\mathbf{i}\langle k,x\rangle}y^{\alpha}\{z^{*}\}^{\beta}\{\bar{z}^{*}\}^{\gamma}.
\end{eqnarray}
Therefore,\ it is easy to verify that
\[
\mathcal{T}^{n}_{0}=\mathbb{T}^{n}\times\{y=0\}\times\{z^{*}=0\}\times\{\bar{z}^{*}=0\}
\]
is an embedding torus with frequency $ \omega(\xi)\in\omega(\Pi_{\gamma})$ of the Hamiltonian $H_{\infty}(x,y,z^{*},\bar{z}^{*},\xi) $.\ We finish the proof of the existence of KAM torus.

\footnotesize

\end{document}